\documentclass[11pt]{article}
\usepackage{amssymb,amsmath, amsthm}
\usepackage{amssymb}
\usepackage{framed}
\usepackage{mathrsfs}
\usepackage{enumerate}
\usepackage{tikz}
\usetikzlibrary{matrix}
\usepackage[all]{xy}
\usetikzlibrary{shapes}
\usetikzlibrary{decorations.markings}
\usepackage{hyperref}
\usepackage{multirow}
\usepackage{bm}
\usepackage[algoruled,boxruled, shortend, lined]{algorithm2e}
 \SetKwProg{myproc}{Initialize:}{}{}\usepackage{algpseudocode}
\usepackage{cite}
\usepackage{enumitem}
\usepackage{subcaption}
\usepackage{pgfplots}
\usepackage{enumitem}
\usepackage{booktabs}   
\usepackage{boxedminipage}

\DeclareMathOperator*{\ee}{\mathbb{E}}

\DeclareMathOperator*{\probs}{\text{\textbf{Pr}}}

\numberwithin{equation}{section}

\DeclareMathOperator*{\dom}{dom\,}

\newcommand{\inclu}[0] {\ar@{^{(}->}}

\newcommand{\cD}{\mathcal{D}}

\newcommand{\Z}{{Z}}

\newcommand{\prox}{{\rm prox}}

\newcommand{\R}{{\bf R}}

\newcommand{\cP}{{\mathcal P}}
\newcommand{\cO}{{\mathcal O}}
\newcommand{\EE}{{\mathbb E}}

\usepackage[margin=1in]{geometry}

\newcommand{\argmin}{\operatornamewithlimits{arg\,min}}
\newcommand{\argmax}{\operatornamewithlimits{arg\,max}}

\usepackage{stackengine}

\SetAlgoSkip{bigskip}
\SetKwInput{Input}{Input\hspace{34pt}{ }}
\SetKwInput{Output}{Output\hspace{27pt}{ }}
\SetKwInput{Initialize}{Initialize\hspace{2pt}{ }}
\SetKwInput{Parameters}{Parameters\hspace{2pt}{ }}

\newtheorem{theorem}{Theorem}[section]
\newtheorem{proposition}[theorem]{Proposition}
\newtheorem{lemma}[theorem]{Lemma}

\newtheorem{defn}[theorem]{Definition}
\newtheorem{corollary}[theorem]{Corollary}

\newtheorem{assumption}{Assumption}
\newtheorem{example}{Example}[section]

\usepackage{mathtools}

\newcommand\blfootnote[1]{%
  \begingroup
  \renewcommand\thefootnote{}\footnote{#1}%
  \addtocounter{footnote}{-1}%
  \endgroup
}
\title{Stochastic optimization with decision-dependent distributions}

\author{Dmitriy Drusvyatskiy\thanks{Department of Mathematics, University of Washington, Seattle, WA 98195; 
\texttt{www.math.washington.edu/{\raise.17ex\hbox{$\scriptstyle\sim$}}ddrusv}. Research of Drusvyatskiy was supported by the NSF DMS 1651851 and CCF 1740551 awards.}
\and 
Lin Xiao\thanks{Facebook AI Research, Seattle, WA 98109. Email: \texttt{linx@fb.com}.}
\blfootnote{Part of this work was done when both authors were affiliated with Microsoft Research.}
}

\begin{document}
\date{}
\maketitle

\begin{abstract}	
 Stochastic optimization problems often involve data distributions that change in reaction to the decision variables.  This is the case for example when members of the population respond to a deployed classifier by manipulating their features  so as to improve the likelihood of being positively labeled. 
Recent works on performative prediction have identified an intriguing solution concept for such problems: 
find the decision that is optimal with respect to the static distribution that the decision induces. 
Continuing this line of work, we show that typical stochastic algorithms---originally designed for static problems---can be applied directly for finding such equilibria with little loss in efficiency. The reason is simple to explain: the main consequence of the distributional shift is that it corrupts  algorithms with a {\em bias} that decays linearly with the distance to the solution. Using this perspective, we obtain sharp convergence guarantees  for popular algorithms, such as stochastic gradient, clipped gradient, proximal point, and dual averaging methods, along with their accelerated and proximal variants. 
In realistic applications, deployment of a decision rule is often much more expensive than sampling. We show how to modify the aforementioned algorithms so as to  maintain their  sample efficiency while performing only logarithmically many deployments.


\end{abstract}


\section{Introduction}
Stochastic optimization plays a central role in statistical sciences and large-scale data-driven computing. 
The goal of stochastic optimization in these settings is to learn a decision rule (e.g. classifier)  from a limited data sample that
generalizes well to the entire population. In simplest circumstances, this task amounts to the optimization problem 
\begin{equation}\label{eqn:PS1}
\mathtt{St}(\cP):\qquad\qquad \min_x~ \ee_{z\sim \cP}\ell(x,z)+r(x).
\end{equation}
Here, $z$ encodes the population data, which is assumed to follow some fixed probability distribution $\cP$ that is   accessible only through sampling. The functions $\ell$ and $r$ play qualitatively different
roles. Typically, $\ell(x,z)$ evaluates the loss of the decision rule parametrized by $x$ on
a data point $z$. In contrast, the function $r\colon\R^d\to\R\cup\{\infty\}$ models constraints or promotes some low-dimensional structure in $x$, such as sparsity or low-rank. 
We refer to the problem~\eqref{eqn:PS1} as $\mathtt{St}(\cP)$ to emphasize its dependence on the distribution $\cP$, while  ``$\mathtt{St}$'' abbreviates ``static''---a term whose significance will become clear shortly.

Stochastic approximation algorithms are often the methods of choice for $\mathtt{St}(\cP)$. 
In every iteration $t=1,2,\ldots$, such methods draw a fresh sample $z_t\in \cP$ and update the iterate $x_t$ using the randomly selected function $\ell(\cdot,z_t)$. Among such algorithms, the stochastic proximal gradient method is the most popular; in each iteration, the method simply takes a step from $x_t$ in the direction opposite to the gradient $\nabla \ell(x_t,z_t)$, followed by a proximal operation of $r$. Other common stochastic algorithms include  proximal point, clipped gradient, and dual averaging methods, along with their inertial variants. 

Convergence guarantees of stochastic optimization algorithms crucially rely on the sampling distribution $\cP$ being fixed throughout the run of an algorithm. This assumption, however, is violated in applications where the distribution evolves along the iterations. 
There are two main sources of such distributional shifts. The first is temporal, where the distribution varies  in time due to reasons that are independent of the iterates~$x_t$. This setting has been extensively studied in the machine learning literature; see e.g. \cite{kuh1991learning,bartlett2000learning,bartlett1992learning,gama2014survey,besbes2015non,besbes2019optimal}. 
The second common source is due to a feedback mechanism, wherein the distribution generating the data in iteration $t$ may depend on, or react to, the current ``state'' $x_t$. For example, deployment of a classifier by a learning system, when made public, often causes the population to adapt their attributes in order to increase the likelihood of being positively labeled---a process called ``gaming''. Even when the population is agnostic to the classifier, the decisions made by the learning system (e.g. loan approval) may inadvertently alter the profile of the population (e.g. credit score).  
The goal of the learning system therefore is to find a classifier
that generalizes well under the response distribution. 
Recent research in strategic classification  \cite{hardt2016strategic,bruckner2012static,bechavod2020causal,dalvi2004adversarial} and performative prediction \cite{perdomo2020performative,mendler2020stochastic}  has highlighted the prevalence of this phenomenon.

\begin{table}[t]
	\renewcommand{\arraystretch}{1.5}
	\begin{center}
		\begin{tabular}{| l | c | }
			\hline
			Algorithms & Iterate update with $z_t\sim \cD(x_t)$ \\ 
			\hline
			Proximal point & $\displaystyle x_{t+1}=\argmin_x ~ \ell(x,z_t)+r(x)+\frac{1}{2\eta_t}\|x-x_t\|^2 $ \\  
			Prox-gradient & $\displaystyle x_{t+1}=\prox_{\eta_t r}\bigl(x_t-\eta_t\nabla \ell(x_t,z_t)\bigr)$   \\
			Accel.\ prox-grad. & 	
			$\left\{
			\begin{aligned}
			x_{t}&=\prox_{\eta_t r}\left(y_{t-1}-\eta_t \nabla \ell(y_{t-1},z_t')\right)\\
			y_t&=x_t+\beta_t(x_t-x_{t-1})
			\end{aligned}\right\}\quad \textrm{with }z_t'\sim \cD(y_{t-1})
			$   \\				Clipped gradient & $\displaystyle x_{t+1}=\argmin_x\,  \bigl(\ell(x_t,z_t)+\langle \nabla \ell(x_t,z_t),x-x_t\rangle\bigr)^+ +r(x)+\frac{1}{2\eta_t}\|x-x_t\|^2$   \\
			Dual averaging & $\displaystyle x_{t+1}=\argmin_x  \Big\langle \frac{1}{t}\sum_{i=1}^t\nabla \ell(x_t,z_t),x\Big\rangle+r(x)+\frac{1}{2\eta_t}\|x-x_0\|^2 $   \\		
			\hline
		\end{tabular}
	\end{center}
	\caption{Stochastic algorithms with state-dependent distributions.}
	\label{table:methods_intro}
\end{table}

\subsection{Problem setting}
The focus of this work is stochastic optimization under decision-dependent distributions. Our approach  to such problems builds on the framework of ``performative prediction'' proposed in \cite{perdomo2020performative,mendler2020stochastic}. Namely, we consider optimization
problems of the form 
\begin{equation}\label{eqn:hard_performativ_predict}
\min_x~ \ee_{z\sim \cD(x)}\ell(x,z)+r(x),
\end{equation}
where $\cD(x)$ is a distribution indexed by the decision variable $x\in\R^d$. In contrast to $\mathtt{St}(\cP)$ in~\eqref{eqn:PS1}, the data distribution is now decision-dependent. Thus, the quality of a decision rule $x$ is judged by its performance according to the induced distribution $\cD(x)$. 
A direct solution of~\eqref{eqn:hard_performativ_predict} is out of reach in general, even for convex loss functions. 
Nonetheless, an appealing and widely used heuristic for such problems is to  apply standard stochastic optimization algorithms to the static problem $\mathtt{St}(\cD(x_t))$ for one or more iterations, then update the distribution $\cD(x_{t+1})$ based on the generated iterates, and repeat.   
Motivated by the recent works \cite{perdomo2020performative,mendler2020stochastic}, we ask the following question:
\begin{quote}
	What can one expect from standard stochastic algorithms (Table~\ref{table:methods_intro}) when the sampling distribution $\cD(x_t)$ used in iteration $t$ depends on the iterate $x_t$?
\end{quote}
The answer we present relies on a certain distinguished point $\bar x$, highlighted in  \cite{perdomo2020performative,mendler2020stochastic}. A point $\bar x$ is at {\em equilibrium} for the family of distributions $\cD(x)$ if $\bar x$ solves the static problem $\mathtt{St}(\cP)$ with $\cP=\cD(\bar x)$. Thus $\bar x$ is at equilibrium  if it solves the static problem that the distribution $\cD(\bar x)$ induces. 
Equilibrium points are sure to exist under mild continuity and convexity assumptions \cite{perdomo2020performative}. Our main contribution can be  summarized as follows:
\begin{quote}
	{Under mild conditions, stochastic optimization algorithms that sample according to a state-dependent distribution $\cD(x_t)$ can be viewed as inexact analogues of the same algorithms applied to the  
		static problem $\mathtt{St}(\cD(\bar x))$.	The inexactness manifests in a bias/error that decays linearly with the distance to  $\bar x$.}
\end{quote}

Thus, stochastic optimization algorithms under decision-dependent distributions automatically search for the equilibrium point $\bar x$ by implicitly solving the problem $\mathtt{St}(\cD(\bar x))$. 
The strength of the distribution's dependence on the state---measured by the Lipschitz constant of the map $ \cD(\cdot)$---directly impacts the decay rate in the bias/error. We show therefore that if the state dependence is sufficiently weak,  standard stochastic algorithms (Table~\ref{table:methods_intro})  exhibit the same efficiency estimates as if they were directly applied to the static problem $\mathtt{St}(\cD(\bar x))$---an a priori impossible task, since the distribution $\cD(\bar x)$ is unknown and inaccessible.

\subsection{Application: strategic classification}
\label{sec:strategic-classification}
An important application arena for the developed techniques is the framework of \emph{strategic classification} \cite{hardt2016strategic,milli2019socialcost}. This problem class can  model a variety of settings with state-dependent distributions, such as fraud detection, traffic prediction, spam filtering, and  service recommendations.
See  \cite{zliobaite2016applications} and \cite[Appendix~B]{perdomo2020performative} for detailed applications.

Strategic classification is a two player game between an ``institution'' that deploys a classifier and a population of ``agents'' who can adapt their features in response in order to increase their likelihood of being positively classified. The game proceeds as an iterative process where the institution and the population take turns to adjust the classifier and features respectively.
Specifically, let $z=(a,b)$ denote the features-label pairs of the population that follow a base distribution $z\sim \cP$.
The institution begins the game by deploying a classifier $h_x$, parametrized by $x\in \R^d$, and learned from data sampled from $\cP$. 
Each agent responds to the classifier $h_x$ by greedily modifying their features $a$ to increase their chance of being favorably labeled:
\begin{equation}\label{eqn:strat_update}
\Delta(h_x, a) := \argmax_{a'} \bigl\{u(h_x, a')-c(a, a')\bigr\}.
\end{equation}
Here $u(\cdot,\cdot)$ is some utility function and $c(\cdot,\cdot)$ is the cost of altering the features.
Thus the samples available to the institution $(\Delta(h_x,a), b)$ in the next stage of the game follow a distribution that depends on~$x$, and which we denote by $(\Delta, b)\sim\cD(x)$. 
The goal of the institution is to find the decision variable~$x$ that minimizes the classification error with respect to the response distribution:
\[
\probs_{(\Delta, b)\sim\cD(x)}[h_x(\Delta)\neq b] ~=~ \ee_{(\Delta, b)\sim\cD(x)}[\mathbf{1}(h_x(\Delta)\neq b].
\]
In practice, one often replaces the $0/1$ loss $\mathbf{1}(\cdot)$ with a convex  surrogate  $\ell(x,z)$, such as the logistic loss, to facilitate large-scale optimization. Thus the problem of strategic classification is an instance of \eqref{eqn:hard_performativ_predict} under a specific family of distributions $\cD(x)$.

In practice, the agents are unlikely to actually play the best-response solutions when modifying their features or that the utility and cost functions are common to all agents. 
Moreover, in some cases, the dependence of the distribution on the state  can be passive;
for example, when a bank uses a classifier to approve loan applications, the credit scores of the population are automatically impacted for downstream tasks, even in absence of feature manipulation.
In line with the recent works \cite{perdomo2020performative,mendler2020stochastic}, we do not restrict to the model of strategic classification, and instead only require the distribution map $\cD(\cdot)$ to have a relatively small Lipschitz constant.

\subsection{Related work}
Our work is closely related to a number of research themes in optimization, statistics, and machine learning. We now highlight these relationships. 
\smallskip

{\em Performative prediction and distributional shift.} The two seminal papers on performative prediction \cite{perdomo2020performative,mendler2020stochastic} motivate and guide much of our  work. In particular, the problem setting and assumptions we use (Section~\ref{sec:stoch_alg_feed_intro}) are identical to that of \cite{perdomo2020performative,mendler2020stochastic}. ``Performative prediction'' is an evocative name for the problem class in machine learning settings since it nicely contrasts the problem class with supervised learning. To stay consistent with the stochastic optimization literature, however, we do not use this terminology here and instead refer to the qualifier ``decision-dependent distributions'' when needed.

The earlier paper \cite{perdomo2020performative} introduces the notion of an equilibrium point (therein called ``performatively stable'')  and identified regimes in which retraining and gradient descent algorithms converge linearly. The follow up paper \cite{mendler2020stochastic}, in turn, analyzes two variants of the projected stochastic gradient method (called greedy and lazy), and establishes convergence of the last iterate to the equilibrium. Our current work complements  \cite{perdomo2020performative,mendler2020stochastic}, aiming to provide a systematic and transparent treatment, based on controlling bias/error. The convergence guarantees we develop 
apply for a wide class of stochastic algorithms, including stochastic gradient, clipped gradient, proximal point, and dual averaging methods, along with their accelerated and proximal variants. Comments and comparisons with  \cite{perdomo2020performative,mendler2020stochastic} appear throughout the text. 
Aside from performative prediction, there is a long history of problems with distributional shift in machine learning, whether due to time drift (e.g. \cite{kuh1991learning,bartlett2000learning,bartlett1992learning,gama2014survey,besbes2015non,besbes2019optimal}) or deployment of classifiers (e.g. \cite{hardt2016strategic,bruckner2012static,bechavod2020causal,dalvi2004adversarial}.
We believe that the techniques developed here may be useful in these contexts as well.

{\em Stochastic programming.} Stochastic optimization problems with decision-dependent uncertainties have  appeared in the classical stochastic programming literature, such as \cite{jonsbraaten1998class,varaiya1988stochastic,dupacova2006optimization,ahmed2000strategic,rubinstein1993discrete}. We  refer the reader to the recent paper \cite{hellemo2018decision}, which discusses taxonomy and various models of decision dependent uncertainties. An important theme of these works is to utilize structural assumptions on how the decision variables impact the distributions. Consequently, these works sharply deviate from the framework explored in  \cite{perdomo2020performative,mendler2020stochastic} and  from  our paper. 

{\em Convex stochastic optimization.} Many of the techniques used here are rooted in convex optimization;  recent monographs on the subject include \cite{beck2017first,bubeck2014convex,nesterov2018lectures}. The arguments we present are most closely related to the literature on (accelerated) stochastic gradient methods \cite{kulunchakov2019estimate,ghadimi2012optimal}, dual averaging \cite{nesterov2009primaldual,xiao2010jmlr}, and model-based minimization \cite{asi2019stochastic,davis2019stochastic}. 

{\em Online convex optimization.}
Online convex optimization is a more general framework than stochastic convex optimization in the sense that the loss function at each iteration need not follow any probability distribution and can even be adversarial. On the other hand, online convex optimization requires stronger assumptions such as bounded domain and gradients;
see, e.g., \cite{zinkevich2003,shalevshwartz2012online,hazan2016intro}.
Under these conditions, any online convex optimization algorithm can be applied to static stochastic optimization problems. Similarly, we show by a simple reduction that any online algorithm can be  applied to stochastic optimization problems with state-dependent distributions.

{\em Error and bias in gradient oracles.}
Biased stochastic gradients play a central role in this work. There is a vast literature on errors/bias in gradient computations. For example, \cite{devolder2014first,dvurechensky2017universal,stonyakin2019inexact} analyze first-order methods under inexact oracles and discuss convergence rates. The bias that appears in the current work is fundamentally different, however, in that it decays linearly with the distance to the solution. The closest error model we are aware of is based on relative errors in the gradient (e.g. \cite[Section 1.2.1]{bertsekas1997nonlinear} and \cite{ajalloeian2020analysis}). The  bias we encounter is more restrictive still and therefore facilitates stronger guarantees.

\subsection{Outline of the paper}
Section~\ref{sec:distrib_shift} presents two fundamental lemmas that characterize the sensitivity of the expected loss function and its gradient to arbitrary distributional shift.
Section~\ref{sec:stoch_alg_feed_intro} formalizes the assumptions that we will use throughout the paper, in particular emphasizing Lipschitz continuity of the distributions  relative to variations in the decision variables. 
Section~\ref{sec:outline_main_results} outlines the main results and previews  technical contributions---formally developed in Sections~\ref{sec:prox_point_prox_grad}-\ref{sec:inexactRM}.


\paragraph{Notation.}
Throughout, we consider a Euclidean space, denoted for simplicity as $\R^d$. The symbol $\langle \cdot,\cdot\rangle$ will denote the inner product in $\R^d$, while $\|x\|=\sqrt{\langle x,x\rangle}$ will denote the induced norm. The {\em proximal map} of any function $f\colon\R^d\to\R\cup\{\infty\}$ is defined as 
$$\prox_{\eta f}(x)=\argmin_{y} \,\Bigl\{f(y)+\tfrac{1}{2\eta}\|y-x\|^2\Bigr\},$$
where $\eta>0$ is an arbitrary constant.

We will be interested in random variables taking values in a  metric space. Therefore, throughout the paper, we fix a metric space $\Z$ with metric  $d(\cdot,\cdot)$ and 
equip $\Z$ with the Borel $\sigma$-algebra. The symbol $\mathbb{P}$  will denote the set of Radon probability measures on $\Z$ with a finite first moment $\ee_{z\sim \cP} [d(z,z_0)]<\infty$ for some $z_0\in Z$. We  measure the deviation between two measures $\mu,\nu\in \mathbb{P}$ using the Wasserstein-1 distance:
\begin{equation*}\label{eqn:KR}
W_1(\mu,\nu)=\sup_{g\in {\rm Lip}_1}\,\bigl\{\EE_{X\sim \mu}[g(X)]-\EE_{Y\sim \nu}[g(Y)]\bigr\},
\end{equation*}	
where ${\rm Lip}_1$ denotes the set of $1$-Lipschitz continuous functions $g\colon\Z\to\R$.
The equivalence of this definition with the description of $W_1(\mu,\nu)$ using couplings is the Kantorovich-Rubinstein duality theorem \cite{kantorovich1958space}.

Throughout the paper, we fix an arbitrary function $r\colon\R^d\to\R\cup\{\infty\}$ and a loss function $\ell\colon \R^d\times Z\to\R$. The symbol $\nabla \ell(x,z)$ will always refer to the gradient of $\ell(x,z)$ in the variable $x$. Throughout the paper, we impose the following assumption.

\begin{assumption}[Smoothness]\label{assum:smoothness}
	The loss $\ell(x,z)$ is $C^1$-smooth in $x$ for all $z\in Z$,  and the map $z\mapsto\nabla \ell(x,z)$ is $\beta$-Lipschitz continuous for any $x\in\R^d$.
\end{assumption}

An important consequence (following from the dominated convergence theorem) is that for any measure $\mu\in \mathbb{P}$, the expected loss $\ee_{z\sim \mu}\ell(x,z)$ is differentiable in $x$ with gradient $\ee_{z\sim \mu}[\nabla \ell(x,z)]$.

\section{Sensitivity to distributional shift}\label{sec:distrib_shift}
Optimization algorithms that rely on state-dependent sampling, in essence, perform updates on a sequence of static problems that slowly vary  along the iterations. An appealing strategy for analyzing such algorithms---and the one we follow here---leverages the stability of the problem $\mathtt{St}(\cP)$ to perturbations in $\cP$. Formalizing this viewpoint,  define the expected loss $$f_{\mu}(x):=\ee_{z\sim \mu}\ell(x,z),$$
for any measure  $\mu\in \mathbb{P}$.  The main question we aim to answer in this section is how variations in $\mu\in\mathbb{P}$ impact the function $f_{\mu}$ and its gradient $\nabla f_{\mu}$. The  two Lemmas~\ref{lem:grads_close} and \ref{lem:func_gap_error} provide the answer that guides much of our development.

We begin with Lemma~\ref{lem:grads_close}, which shows that the  distance $W_1(\mu,\nu)$ between two measures $\mu,\nu\in \mathbb{P}$ uniformly bounds the deviation $\nabla f_{\mu}(x)-\nabla f_{\nu}(x)$. This result is well-known and widely used; we provide a short proof for completeness.  

\begin{lemma}[Gradient deviation]\label{lem:grads_close}
Under Assumption~\ref{assum:smoothness}, 
all measures $\mu,\nu\in \mathbb{P}$ satisfy:
\begin{equation}\label{eqn:grad_const}
\sup_{x\in \R^d}\|\nabla f_{\mu}(x)-\nabla f_{\nu}(x)\|\leq \beta\cdot W_1(\mu,\nu).
\end{equation}
\end{lemma}
\begin{proof}
	Fix a unit vector $v\in \R^d$ and define the function $g(z):=v^T\nabla \ell(x,z)$. Clearly by Assumption~\ref{assum:smoothness}, $g$ is $\beta$-Lipschitz continuous in $z$ and therefore  we deduce
$$v^T(\nabla f_{\mu}(x)-\nabla f_{\nu}(x))=\EE_{z\sim \mu}g(z)-\EE_{z\sim \nu}g(z)\leq\beta\cdot W_1(\mu,\nu).$$
Taking the supremum over unit vectors $v$ yields the result.
\end{proof}

Similarly, it is tempting to use the distance $W_1(\mu,\nu)$ to linearly bound the functional error $|f_{\mu}(x)-f_{\nu}(x)|$; such an estimate, however, is far too crude for analyzing algorithms. Instead, the key insight is that 
 convergence analysis of algorithms does not rely on function values in absolute terms, but rather on their differences. The following lemma provides a {\em multiplicative bound} on the error in the function gaps $\Delta_{\mu}(x,y):=f_{\mu}(x)-f_{\mu}(y)$.
 
\begin{lemma}[Function gap deviation]\label{lem:func_gap_error}
Under Assumption~\ref{assum:smoothness}, 
all points $x,y\in\R^d$ and measures $\mu,\nu\in\mathbb{P}$ satisfy:
 \begin{equation}\label{eqn:val_const}
 \Delta_{\mu}(x,y)-\Delta_{\nu}(x,y)\leq \beta \cdot\|y-x\|\cdot W_1(\mu,\nu).
 \end{equation}
\end{lemma}
 \begin{proof}
 Fix two points $x,y\in \R^d$. For any $s\in [0,1]$, set $x_s:=x+s(y-x)$. The fundamental theorem of calculus allows to write 
 \begin{align*}
 f_{\mu}(y)-f_{\mu}(x)=\int_{0}^1 \langle \nabla f_{\mu}(x_s),y-x\rangle~ds\quad \textrm{and}\quad f_{\nu}(y)-f_{\nu}(x)=\int_{0}^1 \langle \nabla f_{\nu}(x_s),y-x\rangle~ds.
 \end{align*}
 Subtracting the two estimates yields
 \begin{align}
 |[f_{\mu}(y)-f_{\mu}(x)]-[f_{\nu}(y)-f_{\nu}(x)]|&=\left|\int_{0}^1 \langle \nabla f_{\mu}(x_s)-\nabla f_{\nu}(x_s),y-x\rangle~ds \right|\notag\\
 &\leq \|y-x\|\cdot \int_0^1 \|\nabla f_{\mu}(x_s)-\nabla f_{\nu}(x_s)\|~ds\label{eqn:app_cs}\\
 &\leq \beta \cdot\|y-x\|\cdot W_1(\mu,\nu),\label{eqn:mvt_w1}
 \end{align}
 where \eqref{eqn:app_cs} follows from  Cauchy–Schwarz  and \eqref{eqn:mvt_w1} follows from Lemma~\ref{lem:grads_close}.
 \end{proof}

 Importantly, the right side of \eqref{eqn:val_const} is proportional to the {\em product} $\|y-x\|\cdot W_1(\mu,\nu)$. In particular, if the distance $W_1(\mu,\nu)$ is on the order of $\|y-x\|$---the setting we will pass to shortly---then the right-side is proportional to the quadratic error $\|y-x\|^2$. Both Lemmas~\ref{lem:grads_close} and \ref{lem:func_gap_error} will play a central role in the later sections. 

\section{Assumptions under state-dependent sampling}\label{sec:stoch_alg_feed_intro}
We next record the assumptions that we will use to analyze stochastic optimization under state-dependent distributions---the main content of the work. The imposed conditions are identical to those used in the seminal work \cite{perdomo2020performative}.

Consider a family of probability measures $\cD(x)\in \mathbb{P}$  indexed by $x\in\R^d$. It is instructive throughout the discussion to keep in mind  the corresponding problem \eqref{eqn:hard_performativ_predict}. 
The first assumption asserts Lipschitz control on the assignment $x\mapsto\cD(x)$.
\begin{assumption}[Lipschitz distributions]\label{assum:perm_pred}
	There exists $\gamma>0$ satisfying 	
 $$W_1(\cD(x),\cD(y))\leq \gamma\cdot\|x-y\|\qquad \textrm{ for all }x,y\in \R^d.$$
\end{assumption}

As an example, consider the framework of strategic classification in Section~\ref{sec:strategic-classification}.
If the agents employ a linear utility function $u(x,a)=\langle x, a\rangle$ and a quadratic cost for modifying the features $c(a,a')=\frac{1}{2\gamma}\|a-a'\|^2$, then $W_1(\cD(x), \cD(y))\leq\gamma\|x-y\|$; see \cite[Appendix~G]{perdomo2020performative}.

The Lipschitz condition on the distribution map $\cD(\cdot)$, recorded in Assumption~\ref{assum:perm_pred}, quantifies how far the map is away from being constant---the strength of the state-dependence. 
The guarantees of Lemmas~\ref{lem:grads_close} and \ref{lem:func_gap_error} become especially potent  under  Assumption~\ref{assum:perm_pred}. We record them in Corollary~\ref{cor:grad_func}, and will use it often.
To simplify the exposition, we introduce the notation
$$f_x(y):=\ee_{z\sim \cD(x)}\ell(y,z)\qquad \textrm{and}\qquad  \nabla f_x(y):=\ee_{z\sim \cD(x)}\nabla \ell(y,z).$$ 
Note that in light of Section~\ref{sec:distrib_shift}, we may equivalently write $f_x(y)=f_{\cD(x)}(y)$. Observe also that  $\nabla f_x(x)$ is  the gradient of the function $y\mapsto f_x(y)$ evaluated at $y=x$.

\begin{corollary}[Gradient and function gap deviations]\label{cor:grad_func}
	 Suppose Assumptions~\ref{assum:smoothness} and \ref{assum:perm_pred} hold. Then for all points $x,y,u,v\in\R^d$ the estimates hold:
\begin{align}
	\sup_{w\in \R^d}~\|\nabla f_{x}(w)-\nabla f_{y}(w)\|&\leq \gamma\beta\cdot\|x-y\|,\label{eqn:grad_const2}\\
	|\left(f_x(u)-f_x(v)\right)-\left(f_y(u)-f_y(v)\right)|&\leq \gamma\beta \cdot\|x-y\|\cdot\|u-v\|.\label{eqn:gap_comp}
\end{align}
\end{corollary}

We will see that a variety of algorithms with state-dependent sampling are implicitly solving a certain static problem that is at ``equilibrium''. A formal description of this phenomenon 
relies on the notion of an equilibrium point  from \cite{perdomo2020performative}. 
 
\begin{defn}[Equilibrium point]
 	{\rm
 A point $\bar x\in \R^d$ is at {\em equilibrium with respect to} $\cD(\cdot)$ if $\bar x$ solves the static problem $\mathtt{St}(\cD(\bar x))$, or equivalently
$$\bar x\in \argmin_x\, \bigl\{f_{\bar x}(x)+r(x)\bigr\}.$$
}
\end{defn}

Thus a point $\bar x$ is at equilibrium with respect to $\cD(\cdot)$ if $\bar x$ solves the static problem that it entails. 
In the context of strategic classification, the decision variable $\bar x$ is at equilibrium if the institution has no incentive to deploy another classifier based purely on the population's response to $\bar x$.
Equilibrium points are distinct from minimizers of \eqref{eqn:hard_performativ_predict}  in general, though the distance between the two can be bounded under strong convexity assumptions \cite[Theorem 4.3]{perdomo2020performative}.

Observe that equilibrium points are precisely the fixed points of the {\em repeated minimization} procedure 
\begin{equation}\label{eqn:rep:min}
x_{t+1}=\argmin_x\, \bigl\{f_{x_t}(x) +r(x)\bigr\}.
\end{equation}
This algorithm is largely conceptual  since it requires access to the entire data set (to evaluate the expectation) in every iteration. Nonetheless, such ``retraining heuristics'' are ubiquitous in practice for recovering from distributional shifts, regardless of their origin. Repeated minimization will play an important role in later sections.

Equilibrium points are sure to exist under fairly weak assumptions. The following two sufficient conditions  were proved in Theorem 3.5 and Proposition 4.1 of \cite{perdomo2020performative}, respectively.

\begin{proposition}[Existence of equilibrium points]\label{prop:exist_equilib}
	The family of distributions $\cD(\cdot)$ is sure to admit an equilibrium point under either of the two conditions:
	\begin{enumerate}[label=(\alph*)]
		\item\label{i1:exist_equilib} Assumptions~\ref{assum:smoothness} and \ref{assum:perm_pred} hold and the functions $f_x$ are $\alpha$-strongly convex for all $x\in \R^d$ with $\frac{\gamma\beta}{\alpha}<1$. Moreover, in this case, the equilibrium point is unique.
		\item The loss $\ell(x,z)$ is jointly continuous in $(x,z)$ and convex in $x$, the distribution map $\cD(\cdot)$ is continuous, and the domain of $r$ is compact.
	\end{enumerate}
\end{proposition}

In light of Proposition~\ref{prop:exist_equilib}, we make the blanket assumption throughout:
\begin{quote}
\emph{The family of distributions $\cD(\cdot)$ admits an equilibrium point, denoted by $\bar x$.}
\end{quote}
The final two ingredients  are assumptions on strong convexity and smoothness: we will use three variants of the former and two variants of the latter depending on context.

\begin{assumption}[Strong convexity]\label{ass:strong_conv_perm}
The regularizer $r$ is convex and there exists $\alpha>0$ satisfying one of the following three properties:
\begin{enumerate}[label=(\alph*)]
\item\label{it:str_conv1}  $f_{\bar x}$ is $\alpha$-strongly convex,
\item\label{it:str_conv_uniform}  $f_x$ is $\alpha$-strongly convex for all $x\in\R^d$, 
\item\label{it:str_conv_loss} the loss function $\ell(\cdot,z)$ is $\alpha$-strongly convex for all $z\in Z$.
\end{enumerate}
In this case, define the ratio $\rho:=\frac{\gamma\beta}{\alpha}$.
\end{assumption}

Clearly, the implications $\ref{it:str_conv_loss}\Rightarrow \ref{it:str_conv_uniform}\Rightarrow \ref{it:str_conv1}$ hold in Assumption~\ref{ass:strong_conv_perm}. We will see that the constant  $\rho=\frac{\gamma\beta}{\alpha}$ sharply characterizes the regime of convergence of basic stochastic algorithms with state-dependent sampling.

\begin{assumption}[Smoothness]\label{ass:smoothness}
There exists $L>0$ satisfying one of the following:
\begin{enumerate}[label=(\alph*)]
	\item\label{assump:smooth_center} $\nabla f_{\bar x}$ is $L$-Lipschitz continuous,
	\item\label{assump:smooth_center_2} $\nabla f_x$ is $L$-Lipschitz continuous for all $x\in \R^d$. 
\end{enumerate}
If in addition either of Assumptions~\ref{ass:strong_conv_perm}\ref{it:str_conv1}-\ref{it:str_conv_loss} hold, define the condition number $\kappa:=\frac{L}{\alpha}$.
\end{assumption}

\subsection{The interesting parameter regime $\rho<1$}
The proof of Proposition~\ref{prop:exist_equilib}\ref{i1:exist_equilib}, presented in \cite[Theorem 3.5]{perdomo2020performative}, is particularly instructive algorithmically. Namely, observe that repeated minimization \eqref{eqn:rep:min} is a fixed point iteration of the map $S(x):=\argmin_y \{f_x(y)+r(y)\}$. The authors of \cite{perdomo2020performative} show that $S$ is Lipschitz continuous with parameter $\rho$. Consequently,  in the regime $\rho<1$, the map $S$ is a contraction and has a fixed point, which by definition is at equilibrium relative to $\cD(\cdot)$. Conversely, they show in the same paper that repeated minimization can easily diverge if $\rho\geq 1$. In this sense, the parameter regime  $\rho<1$ is the natural setting for analyzing algorithms under state-dependent distributions. In this work,
we aim to show that in the setting $\rho<1$, typical stochastic algorithms (Table~\ref{table:methods_intro}) implicitly solve the static problem $\mathtt{St}(\cD(\bar x))$, where $\bar x$ is an equilibrium point. 
We end the section with an illuminating informal argument supporting this claim for the stochastic gradient method.

Suppose that Assumptions~\ref{assum:smoothness}, \ref{assum:perm_pred}, \ref{ass:strong_conv_perm}\ref{it:str_conv1}, \ref{ass:smoothness}\ref{assump:smooth_center} hold.
In each iteration,  the stochastic gradient method draws a sample $z_t\sim\cD(x_t)$ and uses $\nabla \ell(x_t,z_t)$ to advance. Clearly $\nabla \ell(x_t,z_t)$ is an unbiased estimator of $\nabla f_{x_t}(x_t)$ but is biased with respect to the true gradient $\nabla f_{\bar x}(x_t)$. The 
estimate \eqref{eqn:grad_const2} directly bounds the bias $\|\nabla f_{x}(x)-\nabla f_{\bar x}(x)\|$, which combined with strong convexity,
yields the relative error guarantee:
\begin{equation}\label{eqn:grad_rel_accur}
\|\nabla f_{x}(x)-\nabla f_{\bar x}(x)\|\leq \rho \cdot\|\nabla f_{\bar x}(x)\|.
\end{equation}
A simple consequence is that in the regime $\rho<1$, the two vectors $\nabla f_{x}(x)$ and $\nabla f_{\bar x}(x)$ are well-aligned in the sense that they span an angle with cosine 
 $\sqrt{1-\rho^2}$; see Appendix~\ref{sec:angle_cond} for a quick justification. Therefore, in this parameter regime, numerical methods that use unbiased estimators of $\nabla f_{ x}(x)$ are effectively solving the static problem $\mathtt{St}(\cD(\bar x))$ using biased gradients. 
A simple numerical example will illustrate this viewpoint.

\begin{example}[Illustration]\label{ex:llustr_dumb}
	{\rm
We describe now a synthetic two-dimensional example of mean estimation of a moving Gaussian. Specifically, fix a parameter $\rho\geq 0$ and consider the problem \eqref{eqn:hard_performativ_predict} with losses $\ell(x,z)=\tfrac{1}{2}\|x-z\|^2$, no regularizer $r=0$, and the Gaussian distribution $\cD(x_1,x_2)=N(\rho(x_2,x_1),I)$. A quick computation yields the expression
$$\nabla f_y(x)=x-\ee_{z\sim\cD(y)}(z)=\begin{bmatrix} x_1-\rho y_2\\ x_2-\rho y_1
\end{bmatrix}.$$
It is straightforward to see that  the origin $\bar x=\{0\}$ is the unique equilibrium point of $\mathcal{D}(\cdot)$, provided $\rho^2\neq 1$. A quick computation shows the equalities $\alpha=\beta=1$ and $\gamma=\rho$.
According to the estimate \eqref{eqn:grad_rel_accur}, the vector fields $x\mapsto-\nabla f_{\bar x}(x)$ and  $x\mapsto -\nabla f_x(x)$  span an acute angle pointwise for any $\rho\in(0,1)$. Figure~\ref{fig:allignment} illustrates that this is indeed the case for $\rho=\{0.25,0.5, 0,99\}$.
Moreover with these  parameters, the integral curves of the two vector fields converge to the origin. When $\rho=1.25$, the two vector fields span an obtuse angle in some regions and integral curves of $x\mapsto -\nabla f_x(x)$ may even diverge.

\begin{figure}[h!]
	\begin{center}
		\begin{subfigure}{.35\textwidth}
			\centering
			\includegraphics[width=\linewidth]{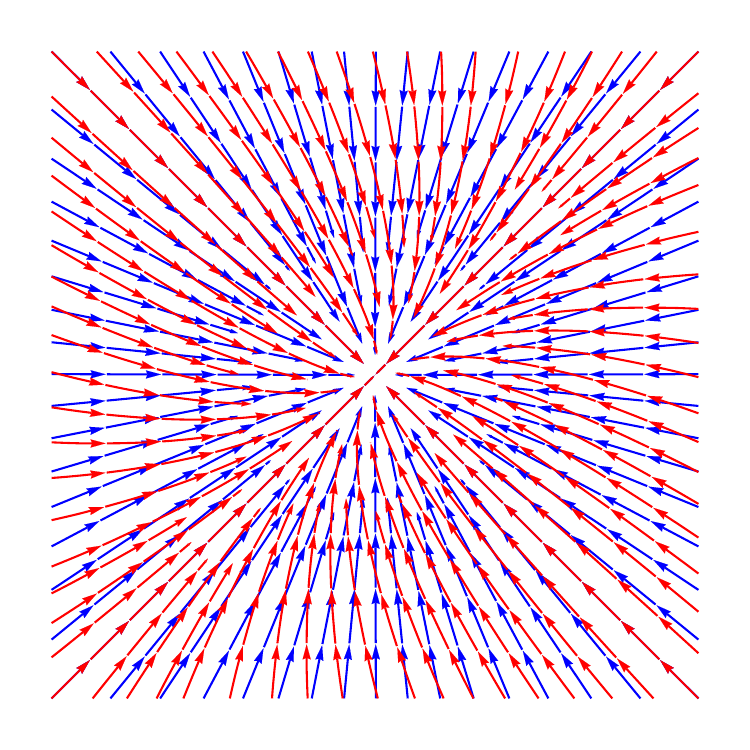}
			\caption{$\rho=0.25$}
		\end{subfigure}\quad
		\begin{subfigure}{.35\textwidth}
			\centering
			\includegraphics[width=\linewidth]{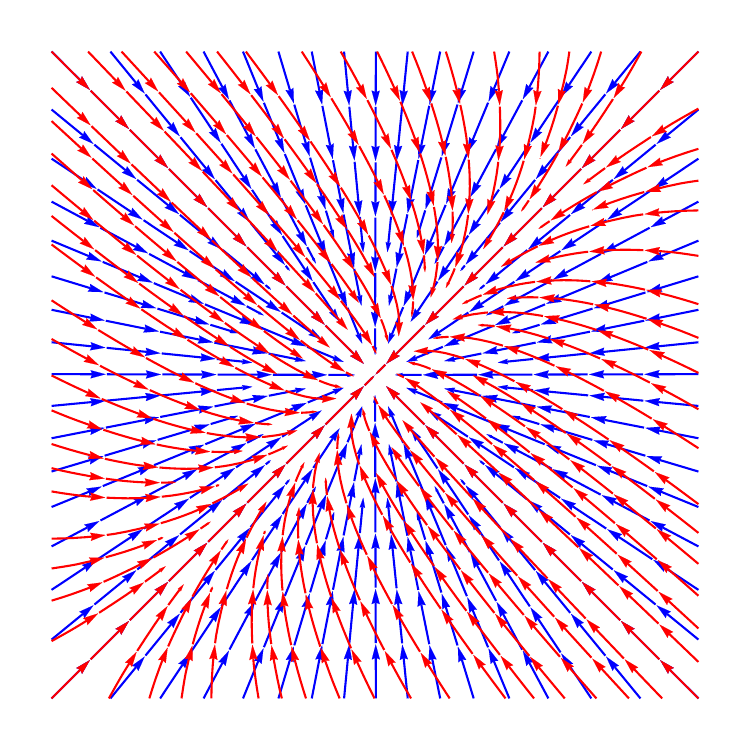}
			\caption{$\rho=0.5$}
		\end{subfigure}
		\begin{subfigure}{.35\textwidth}
			\centering
			\includegraphics[width=\linewidth]{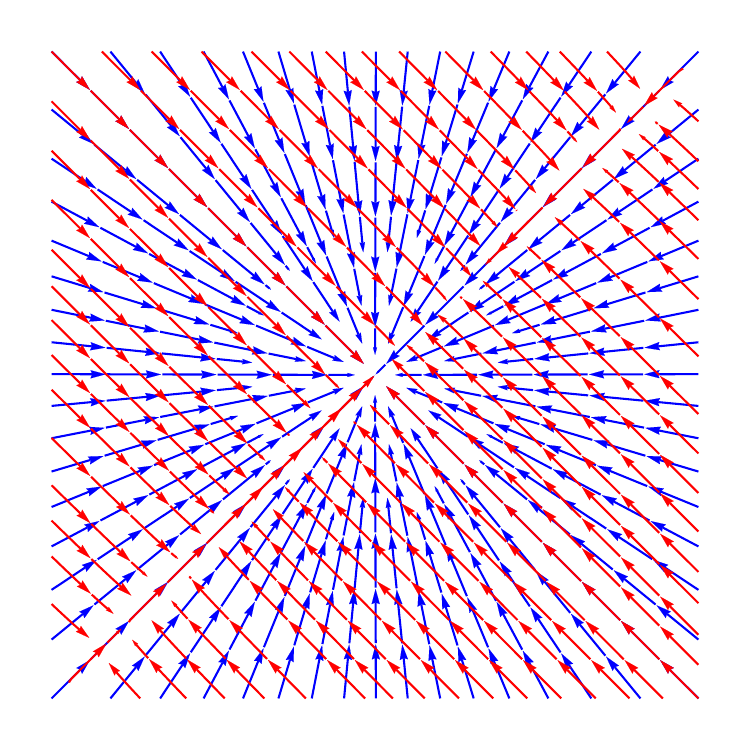}
			\caption{$\rho=0.99$}
		\end{subfigure}\quad
		\begin{subfigure}{.35\textwidth}
			\centering
			\includegraphics[width=\linewidth]{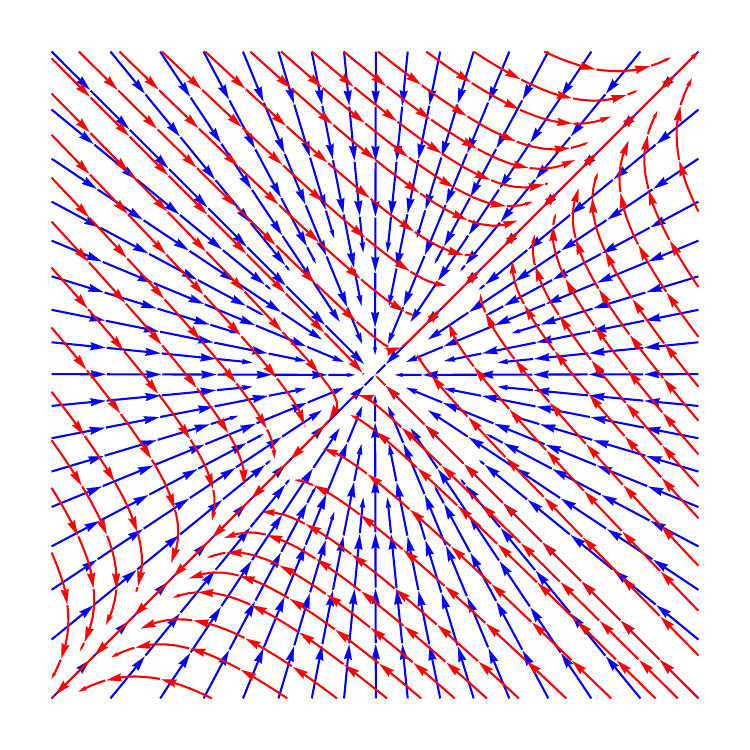}
			\caption{$\rho=1.25$}
		\end{subfigure}
	\end{center}
	\caption{Alignments of vector fields  $x\mapsto -\nabla f_x(x)$ (red) and $x\mapsto -\nabla f_{\bar x}(x)$ (blue). Each picture corresponds to a different choice of $\rho\in \{0.25,0.5,0.99,1.25\}$.
	}
	\label{fig:allignment}
\end{figure}
}
\end{example}

 The  intuition highlighted in Example~\ref{ex:llustr_dumb} can be made precise in the setting $r=0$. 
When regularization is present, the estimate \eqref{eqn:grad_rel_accur} is not useful because the right-hand side might be uniformly bounded away from zero. Instead, our arguments in Section~\ref{sec:biased_oracle}  will use~\eqref{eqn:grad_const2} directly.

\section{Outline of the main results}\label{sec:outline_main_results}
This section outlines the main results of this work. Throughout, the reader should keep in mind the running theme of the paper: a variety of stochastic algorithms under state-dependent distributions are implicitly solving the static problem $\texttt{St}(\cD(\bar x))$, where $\bar x$ is the equilibrium point. 
We impose Assumptions~\ref{assum:smoothness}, \ref{assum:perm_pred}, \ref{ass:strong_conv_perm}\ref{it:str_conv_loss}, \ref{ass:smoothness}\ref{assump:smooth_center_2} throughout the section. 
Define the objective function of the static problems $\texttt{St}(\cD(x))$ as
$$\min_{y}~ \varphi_x(y):=f_{x}(y)+r(y).$$
To shorten the notation, we drop the subscript $\bar x$ from $\varphi_{\bar x}$ and set $\varphi:=\varphi_{\bar x}$.  To better conceptualize the theoretical guarantees, we augment the discussion with numerical illustrations on Example~\ref{ex:llustr_dumb} and on a strategic classification  problem  explored in \cite{perdomo2020performative,mendler2020stochastic}. The implementation details for strategic classification appear in Appendix~\ref{sec:strat_class}.

\subsection{Calm and contractive algorithms (Section~\ref{sec:prox_point_prox_grad})}
We begin the algorithmic development by analyzing a number of conceptual algorithms under state-dependent distributions, the most important being repeated minimization \eqref{eqn:rep:min}. More generally, the classical proximal point and gradient methods extend to the state-dependent setting as follows:
\begin{center}
	\begin{tabular}{ l l} 
		{\bf (Proximal point)} &\qquad $\displaystyle x_{t+1}=\argmin_x\,\Bigl\{ f_{x_{t}}(x)+r(x)+\frac{1}{2\eta}\|x-x_{t}\|^2\Bigr\}$,  \\[2ex]
		{\bf (Proximal gradient)} & \qquad $\displaystyle x_{t+1}=\prox_{\eta r}(x_{t}-\eta\nabla f_{x_{t}}(x_{t}))$,\\[1ex]
	\end{tabular}
\end{center}
where $\eta>0$ is a user-specified step size.
Thus in each iteration $t$, the two algorithms simply take a proximal point step and a proximal gradient step, respectively, on the static problem $\texttt{St}(\cD(x_t))$. The proximal point method in the extreme case $\eta=\infty$ coincides with repeated minimization \eqref{eqn:rep:min}.
Setting notation, let us denote either of these updates as $x_{t+1}=S_{x_t}(x_t)$; more generally, the symbol $S_y(x)$ will denote the update of a point $x$ by an algorithm acting on the static problem $\mathtt{St}(\cD(y))$.

The paper \cite{perdomo2020performative} showed that when $r$ is the indicator function of a closed convex set and $\rho<1$, repeated minimization and the projected gradient method converge linearly to the equilibrium point $\bar x$. Section~\ref{sec:prox_point_prox_grad} provides an alternative and transparent explanation  based on stability to distributional shifts.
Namely, linear convergence 
is a direct consequence of the two independent phenomena:
\begin{enumerate}
	\item {\bf (Calmness to distribution)} The  updates $S_y(x)$ are {\em $\tau$-calm}  relative  to $\cD(\bar x)$, meaning there exists $\tau>0$ satisfying $\displaystyle\sup_x\|S_{y}(x)-S_{\bar x}(x)\|\leq \tau\cdot W_1(\cD(y),\cD(\bar x))$ for all $y\in\R^d$.
	\item {\bf (Contraction at equilibrium)} The updates $S_{\bar x}(\cdot)$ are {\em $q$-contractive} on the static problem $ \mathtt{St}(\cD(\bar x))$, meaning 
 $\|S_{\bar x}(x)-S_{\bar x}(y)\|\leq q\cdot\|x-y\|$ for all $x,y\in\R^d$.  
\end{enumerate}

The first property asserts control on how the update $S_{y}(x)$ varies with respect to the distribution $\cD(y)$, 
while the contraction property asserts that the update is contractive when applied to the static problem induced by $\cD(\bar x)$.
It is elementary to see that these two conditions imply that the update 
$x\mapsto S_x(x)$ contracts towards $\bar x$ with ratio $q+\gamma\tau$. By computing the calmness and contraction parameters for different algorithms (in terms of $\rho>0$), we obtain the following theorem.



\begin{theorem}[Informal]
Repeated minimization and proximal point methods converge linearly to $\bar x$ in the regime $\rho<1$, while the proximal gradient method converges linearly to $\bar x$ in the regime $\rho<\frac{1}{2}$.
(A more careful argument in Section~\ref{sec:biased_oracle} shows that  the proximal gradient method converges linearly in the optimal regime  $\rho<1$.)
\end{theorem}

In theory, repeated minimization and the proximal-point method work within the same optimal parameter regime $\rho\in (0,1)$. We have seen experimentally, however, that the proximal point method can succeed in a much wider parameter regime whereas repeated minimization can exhibit wild oscillatory behavior. As an example, Figure~\ref{fig:prox_point} depicts the performance of both algorithms with different proximal parameters on a problem of strategic classification. 
 It is an interesting question to identify the theoretic justification for this behavior.

\begin{figure}[h!]
	\begin{center}
			\centering
			\includegraphics[scale=0.5]{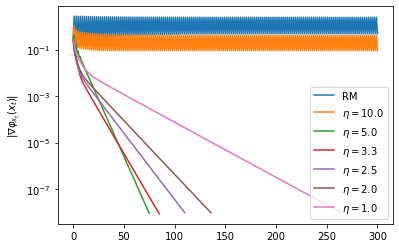}
			\caption{We implement repeated minimization and the proximal point method on a problem of strategic classification detailed in Appendix~\ref{sec:strat_class} (the problem parameters are $n=200$, $\alpha=10/n$, $\gamma=5$.) ``RM'' refers to repeated minimization while the rest of the curves refer to the proximal-point method with parameter $\eta$. Repeated minimization and the prox-point method with $\eta=10$ exhibit wild oscillatory behavior. The prox-point method with parameters $\eta\in\{1,2,2.5,3.3,5.0\}$ succeeds at finding the equilibrium point, with a rate that degrades with decreasing $\eta>0$.}	\label{fig:prox_point}

		\end{center}
\end{figure}

The three algorithms described so far (repeated minimization, prox-point, gradient descent) are largely conceptual since they require access to the entire data set in every iteration.
Implementable algorithms under state-dependent distributions fall into two categories---those that update the sampling distribution in every step and those that run multiple iterations of stochastic methods on the current static problem before updating the sampling distribution. The former algorithms are called ``greedy'' and the latter are called ``lazy'' in \cite{mendler2020stochastic}. As remarked in \cite{mendler2020stochastic}, lazy algorithms can be interpreted as performing inexact repeated minimization. Such algorithms can be advantageous in applications where  the action of updating the distribution is much costlier than sampling from the currently available distribution.
For example, this is the case in applications where it takes significant amount of time for the population to adjust to a newly unveiled learning rule.

We analyze both types of algorithms in this work. Most of the paper, however, focuses on greedy algorithms, while inexact repeated minimization is deferred to Section~\ref{sec:inexactRM}.
Though the two types of methods generate entirely different trajectories,  the convergence arguments we present are slight modifications of each other---thereby underscoring the utility of the developed techniques.

\subsection{Reduction to online convex optimization (Section~\ref{sec:red_online})}
We begin by showing that virtually any algorithm developed for ``online convex optimization'' can be used to find the equilibrium point $\bar x$ in the stochastic setting under state-dependent distributions.
The framework of online convex optimization can be interpreted as a repeated game (e.g., \cite{zinkevich2003,shalevshwartz2012online,hazan2016intro}).
At each round~$t$, the player chooses a point $x_t\in \dom r$, then a convex cost function $\ell_t$ is revealed and the player incurs the cost $\ell_t(x_t)$.
The goal of the player is to minimize the regret
\[
R_t := \sum_{i=1}^t \bigl(\ell_i(x_i)+r(x_i)\bigr) - \min_x \sum_{i=1}^t \bigl(\ell_i(x)+r(x)\bigr),
\]
which is simply the difference between the total regularized cost  incurred up to round~$t$ and the minimum regularized cost in hindsight. Typical algorithms for online convex optimization are the  proximal gradient \cite{DuchiSinger2009}, dual averaging  \cite{xiao2010jmlr}, and variants of FTRL (Follow-The-Regularized-Leader) methods \cite{McMahan2017}.
Under various Lipschitz and strong convexity assumptions, the regret $R_t$  scales as $\mathcal{O}(\log{t})$.
We establish the following reduction by leveraging the gap deviation inequality \eqref{eqn:gap_comp}.

\begin{theorem}[Informal]\label{thm:outline_online}
Suppose that we are in the regime $\rho<\frac{1}{2}$. Then any online algorithm with regret $R_t=\mathcal{O}(\log t)$  
can be used under state-dependent distributions by declaring $\ell_i(x_i)=\ell(x_i,z_i)$ with $z_i\sim \cD(x_i)$ in each iteration. Then the average iterate
$\hat{x}_t:=\frac{1}{t}\sum_{i=1}^t x_i$ satisfies
\[
\EE[\varphi(\hat{x}_t) - \varphi(\bar{x})] ~\leq~ \mathcal{O}\left(\frac{\log t}{(1-2\rho)\,t}\right), \qquad \textrm{for all } t\geq 0.
\]
\end{theorem}

In particular, known regret bounds for the proximal gradient, regularized dual averaging, and FTRL algorithms directly yield a converge rate $\mathcal{O}(\log(t)/(1-2\rho)t)$ under state-dependent sampling.  
Though the theorem is attractive in its generality,  it is far from satisfactory. Indeed, the framework of online convex optimization is far more general than stochastic optimization, since the loss function $\ell_t$ at each iteration may not follow any probability distribution and can even be adversarial.
As a consequence, online convex optimization requires stringent assumptions in order to have meaningful regret analysis, most notably that the encountered gradients of the loss functions and the domain $\dom r$ be bounded. Moreover smoothness of the loss function does not play a significant role. 

We will see that much finer convergence guarantees hold for stochastic optimization under state-dependent distributions. In particular, bounds on the gradient of the loss will be replaced by the finite variance assumption.

 \begin{assumption}[Finite variance]\label{assum:fin_var}
 	There is a constant $\sigma>0$ satisfying 
 	$$\EE_{z\sim \cD(x)}\|\nabla \ell(x,z)-\nabla f_x(x)\|^2\leq \sigma^2\qquad \forall x\in\dom r.$$
 \end{assumption}

\subsection{Stochastic gradient methods (Section~\ref{sec:biased_oracle})}
The simplest and most widely used stochastic algorithm is the stochastic gradient method. We begin by investigating its extension under decision-dependent distributions:
\begin{equation}\label{eqn:outline_prox_grad}
\mathrm{SG:}\qquad
\left\{\quad
\begin{aligned}
	&\textrm{Sample~} z_t\sim \cD(x_t)\textrm{ and set } g_t=\nabla \ell(x_t,z_t),\\
	& \textrm{Set~} x_{t+1}=\prox_{\eta_t r}(x_t-\eta_t g_t).
	\end{aligned}\quad\right\}
\end{equation}
Rather than reducing the method SG to online convex optimization, as mentioned previously, we will analyze it directly.
Observe that contrary to the static setting, the vector $g_t$ is a biased estimator for the true gradient $\nabla f_{\bar x}(x_t)$, where the bias is proportional to  $\|x_t-\bar x\|$; recall the gradient deviation inequality \eqref{eqn:grad_const2}. We will prove the following.

\begin{lemma}[Key recursion]\label{intor:key_recurs}
	Suppose that the step-size sequence satisfies $\eta_t<\frac{1}{2L}$.  Then the iterates $\{x_t\}$ generated by the SG Algorithm in~\eqref{eqn:outline_prox_grad} satisfy
	\begin{align*}
	2\eta_t\EE[\varphi(x_{t+1})-\varphi(\bar x)]&\leq \left(1-\alpha(1-2\rho)\eta_t+2 \gamma^2\beta^2\eta_t^2\right)\EE\|x_t-\bar x\|^2  -  \EE\|x_{t+1}-\bar x\|^2 +2\eta_t^2 \sigma^2.
	\end{align*}
\end{lemma}

Observe that the contraction factor multiplying $\EE\|x_t-\bar x\|^2$ for small $\eta$ scales as $1-\alpha(1-2\rho)\eta$. It follows that in the regime $\rho<\frac{1}{2}$, one can drive the gap 
$\EE[\varphi(x_{t+1})-\varphi(\bar x)]$ to zero at a controlled rate, with an appropriate choice of $\eta_t$. Moreover, a quick argument shows that the parameter regime of convergence becomes larger if we focus on the rate at which the square distance $\|x_{t+1}-\bar x\|^2$ decays. Indeed, lower-bounding the left side in Lemma~\ref{intor:key_recurs} using strong convexity and rearranging yields the one-step progress guarantee
	\begin{align*}
	 \EE\|x_{t+1}-\bar x\|^2&\leq \left(1-\alpha(1-\rho)\eta_t+\tfrac{1}{2}\eta_t^2 \gamma^2\beta^2\right)\EE\|x_t-\bar x\|^2  +\eta_t^2 \sigma^2.
	\end{align*}
For small $\eta_t$, the contraction factor multiplying $\|x_t-\bar x\|^2$ roughly scales as $1-\alpha(1-\rho)\eta$. Therefore in the regime $\rho<1$, one can drive the square distance 
$\EE[\|x_{t+1}-\bar x\|^2]$ to zero at a controlled rate, with an appropriate choice of $\eta_t$.

 With Lemma~\ref{intor:key_recurs} at hand, obtaining formal guarantees with various choices of parameters $\eta_t$ is standard. The following theorem presents one such guarantee. The first part of the theorem when $r$ is an indicator function of a closed convex set was  proved in \cite{mendler2020stochastic} using a different argument.

\begin{theorem}[Informal]\label{thm:intro_stoch_gradetc}
In the regime $\rho< 1$, the proximal stochastic gradient method with appropriate parameters  $\eta_t$ will generate a point $x$ satisfying $\EE\|x- \bar x\|^2\leq \varepsilon$ using 
$$\mathcal{O}\left(\left(\left(\frac{\rho}{1-\rho}\right)^2+\frac{\kappa}{1-\rho}\right)\cdot\log\left(\frac{\|x_0-\bar x\|^2}{\varepsilon}\right)+\frac{\sigma^2}{(1-\rho)^2\alpha^2	\varepsilon}\right)\qquad\textrm{samples}.$$
Moreover, in the regime $\rho< \frac{1}{2}$, the method will generate $x$ satisfying $\EE[\varphi(x)- \varphi(\bar x)]\leq \varepsilon$ using 
$$\mathcal{O}\left(\left(\left(\frac{\rho}{1-2\rho}\right)^2+\frac{\kappa}{1-2\rho}\right)\cdot\log\left(\frac{\varphi(x_0)-\varphi(\bar x)}{\varepsilon}\right)+\frac{\sigma^2}{(1-2\rho)\alpha	\varepsilon}\right)\qquad\textrm{samples}.$$
\end{theorem}

In the static setting $\gamma=0$ (and hence $\rho=0$), Theorem~\ref{thm:intro_stoch_gradetc} recovers the classical guarantees for the proximal stochastic gradient method \cite{ghadimi2012optimal,ghadimi2013optimal}. 
Figure~\ref{fig:stoch_grad_synth} illustrates the performance of SG on Example~\ref{ex:llustr_dumb}.


\begin{figure}[h!]
	\begin{center}
		\begin{subfigure}{.5\textwidth}
			\centering
			\includegraphics[width=\linewidth]{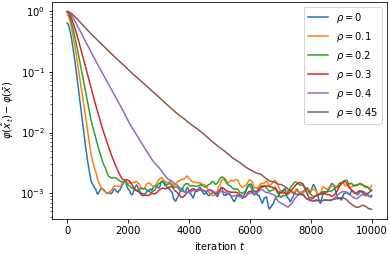}
			\caption{Functional value along the average.}
			\label{fig:stream_average}
		\end{subfigure}%
		\begin{subfigure}{.5\textwidth}
			\centering
			\includegraphics[width=\linewidth]{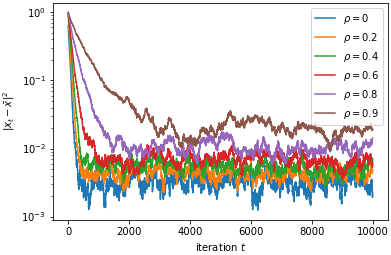}
			\caption{Square distance to equilibrium}
			\label{fig:stream_last_iterate}
		\end{subfigure}
		
	\end{center}
	\caption{Returning to Example~\ref{ex:llustr_dumb}, we implement a stochastic gradient method  with a constant parameter $\eta=0.01$ averaged over $20$ runs for various parameters $\rho$. Figure~\ref{fig:stream_average} shows the function gap $\varphi(\hat x_t)-\varphi(\bar x)$ along the ``average iterate'' on the equilibrium problem. As the theoretical results (Theorem~\ref{thm:rate_grad-desc}) suggest,  the gap tends linearly to a noise level controlled by $\eta$, with a linear rate that degrades as $\rho$ tends to $1/2$. Figure \ref{fig:stream_last_iterate} depicts the square distance of the current iterate $\|x_t-\bar x\|^2$, which also tends linearly to a noise level that now depends both on $\eta$ and on $\rho$. Indeed, Figure \ref{fig:stream_last_iterate} is fully justified by the results of \cite{mendler2020stochastic}, and we include it here only as an illustration.}
		\label{fig:stoch_grad_synth}
\end{figure}

Accelerated gradient methods, famously introduced by Nesterov \cite{nest_orig} and extended to the proximal setting by Beck and Teboulle \cite{beck}, play a central role in convex optimization. Such methods are best in class for smooth convex optimization.
Ghadimi and Lan \cite{lan2012optimal,ghadimi2013optimal} proposed an accelerated method for the stochastic setting, which is best in class for smooth convex stochastic problems. Continuing  the theme of the paper, we ask whether acceleration is possible with state-dependent distributions. We show an affirmative answer in the slightly suboptimal regime $\rho\lesssim \kappa^{-1/4}$. 

Rather than analyzing the original method of Ghadimi and Lan \cite{ghadimi2013optimal}, we focus on the more recent variant of Kulunchakov and Mairal \cite{kulunchakov2019estimate}. With decision-dependent distributions, it reads as:
\begin{equation}\label{eqn:asg}
\mathrm{ASG:}\qquad
\left\{\quad
	\begin{aligned}
	&\textrm{Sample~} z_t\sim \cD(y_{t-1})\textrm{ and set } g_t=\nabla \ell(y_{t-1},z_t),\\
	&\textrm{Set~} x_{t}=\prox_{\eta_t r}(y_{t-1}-\eta g_t), \\
	&\textrm{Set~} y_t=x_t+\tfrac{1-\sqrt{\eta\alpha (1 -2\rho)}}{1+\sqrt{\eta\alpha (1 -2\rho)}}(x_t-x_{t-1}).
	\end{aligned}\quad\right\}.
\end{equation}
In the deterministic and unregularized ($r\equiv 0$) setting, the method reduces to the classical procedure in \cite{intro_lect}, derived through estimate sequences.
The following theorem summarizes the convergence guarantees of the ASG algorithm under state dependent sampling.

\begin{theorem}[Informal]\label{thm:informal_accel_grad}
In the regime $\rho\lesssim \kappa^{-1/4}$, the stochastic ASG method \eqref{eqn:asg} with appropriate parameter choices will generate a point $x$ satisfying $\EE[\varphi(x)-\varphi(\bar x)]\leq \varepsilon$ using 
$$\mathcal{O}\left(\sqrt{\kappa}\cdot\log\left(\frac{\varphi(x_0)- \varphi(\bar x)}{\varepsilon}\right)+\frac{\sigma^2}{\alpha\varepsilon}\right) \qquad \textrm{samples}.$$
\end{theorem}

Notice that the regime  when the accelerated method is guaranteed to work $\rho\lesssim \kappa^{-1/4}$ is ``suboptimal'' by the small factor $\kappa^{-1/4}$. On the other hand, it is  surprising that there is any regime $\rho>0$ where the accelerated method works at all, since it is well known that accelerated methods suffer from error accumulation \cite{devolder2014first}. The reason there is no contradiction here is that the gradient bias that we encounter tends to zero linearly as one approaches the solution \eqref{eqn:grad_const}. It would be interesting to know whether the extra factor $\kappa^{-1/4}$ is really necessary or is an artifact of the proof.

The acceleration phenomenon is most prominent in the nearly noiseless setting $\sigma\approx 0$. As an illustration, Figure~\ref{fig:accele} compares the performance of the vanilla gradient method and the accelerated gradient method in the batch setting ($\sigma=0$) on a problem of strategic classification.  Experimentally, we see that acceleration leads to an impressive speedup even in very ill-conditioned settings, thereby suggesting that the parameter regime $\rho\lesssim \kappa^{-1/4}$ in  Theorem~\ref{thm:informal_accel_grad} may be loose.

\begin{figure}[h!]
	\begin{center}
		\begin{subfigure}{.4\textwidth}
			\centering
			\includegraphics[width=\linewidth]{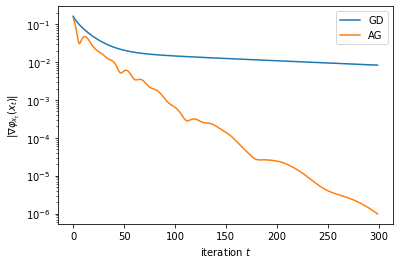}
			\caption{$\gamma=0$.}
		\end{subfigure}\quad
		\begin{subfigure}{.4\textwidth}
			\centering
			\includegraphics[width=\linewidth]{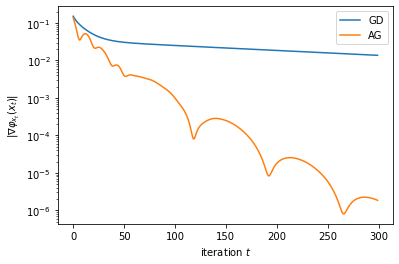}
			\caption{$\gamma=5$.}
		\end{subfigure}%
		
		\begin{subfigure}{.4\textwidth}
			\centering
			\includegraphics[width=\linewidth]{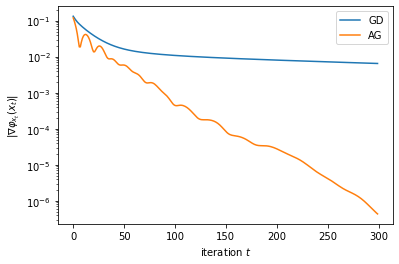}
			\caption{$\gamma=100$.}
		\end{subfigure}\quad
		\begin{subfigure}{.4\textwidth}
			\centering
			\includegraphics[width=\linewidth]{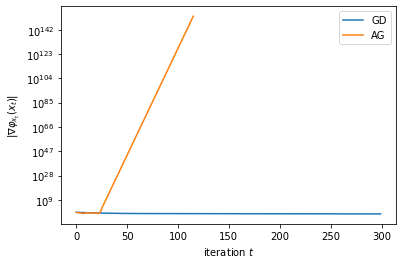}
			\caption{$\gamma=250$.}
		\end{subfigure}%
		
	\end{center}
	\caption{We implement accelerated gradient descent on strategic classification (Appendix~\ref{sec:strat_class}) with parameters
	$\alpha=10/n$ and $n=2000$. In this example $\beta=2$ and therefore the inclusion $\rho\in (0,1)$ holds as long as $\gamma\in (0,\frac{\alpha}{2})$. In each iteration, the method uses the full gradient $\nabla f_{y_{t}}(y_t)$ and we set $y_t=x_t+\tfrac{1-\sqrt{\alpha/L}}{1+\sqrt{\alpha/L}}(x_t-x_{t-1})$, thereby heuristically ignoring scaling by $1-2\rho$. Acceleration leads to an impressive speedup in settings far beyond the threshold $\gamma\in (0,\frac{\alpha}{2})$.}
	\label{fig:accele}
\end{figure}

\paragraph{Proof technique: gradient deviation as a measure of bias.}
The results in this section follow from the following transparent geometric reasoning. Recall that the gradient deviation inequality \eqref{eqn:grad_const2} shows that contrary to the static setting, the vector $\nabla \ell(x,z)$ with $z\sim \cD(x)$ is a biased estimator of the gradient $\nabla f_{\bar x}(x)$, with bias scaling as  $\|x-\bar x\|$. Nonetheless, a quick computation shows that the mean of the estimator $\nabla f_x(x)=\ee_{z\sim \cD(x)}[\nabla \ell(x,z)]$  furnishes a strong convexity inequality between $x$ and $\bar x$ given by
$$f_{\bar x}(\bar x)\geq f_{\bar x}(x)+\langle \nabla f_x(x),\bar x-x\rangle+\frac{\alpha(1-2\rho)}{2}\|x-\bar x\|^2.$$
This inequality suffices to establish convergence guarantees for the proximal stochastic gradient method; the reason is simply that strong convexity is used in the classical argument only to compare the function values along the iterates with the minimal value. Perhaps more surprisingly, the accelerated variant of the method can also be understood from this viewpoint. Section~\ref{sec:biased_oracle} analyzes the proximal stochastic gradient  method and its accelerated variant under a biased stochastic oracle model. This oracle model is broader than the setting of state-dependent sampling and  may be of independent interest.

\subsection{Model-based minimization: stochastic proximal point and clipped gradient methods (Section~\ref{sec:MBA_algos})}
Though the stochastic gradient method is popular in practice, it has well-documented deficiencies.
Notably, the method is highly sensitive to algorithmic parameters, with small misspecifications often drastically degrading performance. Recent works \cite{asi2019stochastic,ryu2014stochastic} have suggested that algorithms based on tighter models than linear may lead to more robust algorithms.  Following \cite{asi2019stochastic,davis2019stochastic}, we consider a class of  stochastic algorithms that proceed as follows. In each iteration $t$, the methods draw a sample $z_t\in \cD(x_t)$ and approximate the loss function $\ell(\cdot,z_t)$ by a simpler model $\ell_{x_t}(\cdot,z_t)$ formed at the basepoint $x_t$. The next iterate $x_{t+1}$ is then  the minimizer of the function $\ell_{x_t}(\cdot,z_t)+r+\frac{1}{2\eta_t}\|\cdot-x_t\|^2$. Thus, the model-based algorithm repeats the steps
	\begin{equation}
    \left\{\quad\begin{aligned}
	&\textrm{Sample~} z_t\sim \cD(x_t)\\
	& \textrm{Set~} x_{t+1}=\argmin_y\,\Bigl\{ \ell_{x_t}(y,z_t)+r(y)+\frac{1}{2\eta_t}\|y-x_t\|^2 \Bigr\}
	\end{aligned}\quad\right\}\tag{MBA}
	\end{equation}
For example, the stochastic proximal gradient method \eqref{eqn:outline_prox_grad} uses the linear model $\ell_x(y,z)=\ell(x,z)+\langle \nabla \ell(x,z),y-x\rangle$, while the stochastic proximal point method uses the loss function itself $\ell_x(y,z)=\ell(y,z)$. Often, the proximal point subproblem can be solved in closed form since it depends only on a single data point. Yet another interesting algorithm is the clipped gradient method, which uses the truncated  models $\ell_x(y,z)=\max\{\ell(x,z)+\langle \nabla \ell(x,z),y-x\rangle,0\}$ under the assumption that the losses  are nonnegative. See Fig.~\ref{fig:illustr_lower_model} for an illustration. Tighter models often lead to better performing algorithms.

\begin{figure}[t]\centering
\begin{subfigure}{0.3\textwidth}
	\begin{center}
		\begin{tikzpicture}[scale=0.65]
		\pgfplotsset{every tick label/.append style={font=\large}}
		\begin{axis}[%
		domain = -2:3,
    ymin=-1,
    ymax=3.5,
		samples = 200,
		axis x line = center,
		axis y line = center,
		xtick={1},
		ytick=\empty
		]		
		\addplot[black, ultra thick] {ln(1+exp(x))} [yshift=3pt] ;
		\addplot[purple,very thick] {ln(1+exp(x))} [yshift=3pt] ;

		\addplot [only marks,mark=*] coordinates { (1,1.31326) };
		\end{axis}		
		\end{tikzpicture}
	\end{center}
\end{subfigure}%
\begin{subfigure}{0.3\textwidth}
	\begin{center}
		\begin{tikzpicture}[scale=0.65]
		\pgfplotsset{every tick label/.append style={font=\large}}
		\begin{axis}[%
		domain = -2:3,
		    ymin=-1,
    ymax=3.5,
		samples = 200,
		axis x line = center,
		axis y line = center,
		xtick={1},
		ytick=\empty
		]		
		\addplot[black, ultra thick] {ln(1+exp(x))} [yshift=3pt] ;
		 {         \addplot[purple, very thick] {ln(1+exp(1))+exp(1)/(1+exp(1))*(x-1)};      }

		\addplot [only marks,mark=*] coordinates { (1,1.31326) };
		\end{axis}		
		\end{tikzpicture}
	\end{center}
\end{subfigure}%
\begin{subfigure}{0.3\textwidth}
	\begin{center}
		\begin{tikzpicture}[scale=0.65]
		\pgfplotsset{every tick label/.append style={font=\large}}
		\begin{axis}[%
		domain = -2:3,
		    ymin=-1,
    ymax=3.5,
		samples = 200,
		axis x line = center,
		axis y line = center,
		xtick={1},
		ytick=\empty
		]		
		\addplot[black, ultra thick] {ln(1+exp(x))} [yshift=3pt];
		 {         \addplot[purple, very thick] {max(ln(1+exp(1))+exp(1)/(1+exp(1))*(x-1),0)} [yshift=6pt] ;      }

		\addplot [only marks,mark=*] coordinates { (1,1.31326) };
		\end{axis}		
		\end{tikzpicture}
	\end{center}
	\end{subfigure}	
  \caption{Illustration of the three models for the function $\ell(y)=\ln(1+e^y)$;  black curve depicts the graph of $\ell$, the red curves depict the models $\ell_1(y)=\ell(y)$ (proximal point), $\ell_1(y)=\ell(1)+\ell'(1)(y-1)$ (gradient), $\ell_1(y)=\max\{\ell(1)+\ell'(1)(y-1),0\}$ (clipped gradient).}
  \label{fig:illustr_lower_model}
\end{figure}
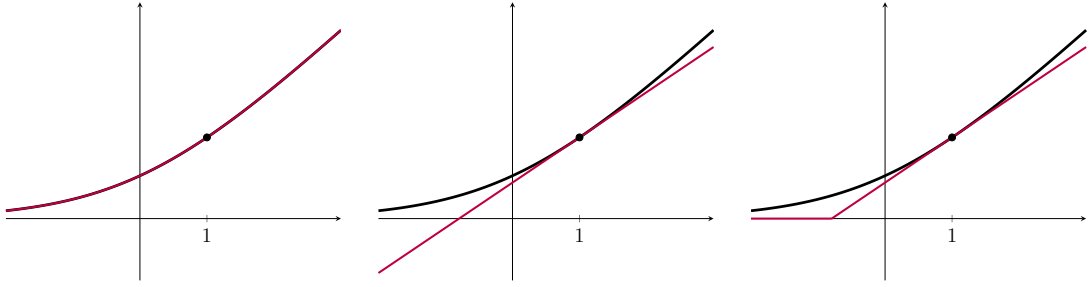

Section~\ref{sec:MBA_algos} presents convergence guarantees for the entire class of model-based algorithms. It will be important for the clipped stochastic gradient method to assume that 
 $r$ is $\mu$-strongly convex for some $\mu\geq 0$.  The efficiency guarantees specialized for the stochastic proximal gradient, proximal point, and clipped gradient methods read as follows. 

\begin{theorem}[Informal]\label{thm:inform_model}
Define $\tilde \alpha:=\alpha+\mu$ for stochastic proximal gradient and proximal point methods and set $\tilde \alpha:=\mu$ for the clipped gradient method; in addition, set $\tilde\rho=\frac{\gamma\beta}{\tilde\alpha}$. In the regime $\tilde\rho< 1$, the three methods with appropriate parameters  $\eta_t$ will find  $x$ satisfying $\EE\|x- \bar x\|^2\leq \varepsilon$ using 
$$\mathcal{O}\left(\frac{L+\tilde \alpha}{\tilde\alpha(1-\tilde\rho)}\cdot\log\left(\frac{\|x_0-\bar x\|^2}{\varepsilon}\right)+\frac{\sigma^2}{\tilde\alpha^2(1-\tilde\rho)^2	\varepsilon}\right)\qquad\textrm{samples}.$$
Moreover, in the regime $\tilde\rho< \frac{1}{2}$, the methods will generate  $x$ satisfying $\EE[\varphi(x)- \varphi(\bar x)]\leq \varepsilon$ using 
$$\mathcal{O}\left(\frac{L+\tilde\alpha}{\tilde\alpha(1-2\tilde\rho)}\cdot\log\left(\frac{\varphi(x_0)-\varphi(\bar x)}{\varepsilon}\right)+\frac{\sigma^2}{\tilde\alpha(1-2\tilde\rho)	\varepsilon}\right)\qquad\textrm{samples}.$$
\end{theorem}

As an illustration, Figure~\ref{fig:models} illustrates the performance of the stochastic gradient, clipped gradient, and proximal point methods on a problem of strategic classification. We use a sublinearly decaying stepsize-sequence $\eta_t=\frac{2}{\alpha(t+1)}$. The three methods perform similarly asymptotically. In the initial stage, however, the subgradient method generates iterates that are highly suboptimal due to a large initial step-size. Consequently, the clipped gradient and proximal-point methods may be preferable. The drastically different performance in the early stages between the subgradient method and  the clipped gradient/proximal point methods  in the static setting was investigated in \cite{asi2019stochastic}. Similar guarantees likely extend to the setting of decision-dependent distributions, though we do not pursue this line of work here.

\begin{figure}[h!]
	\begin{center}
		\begin{subfigure}{.5\textwidth}
			\centering
			\includegraphics[width=\linewidth]{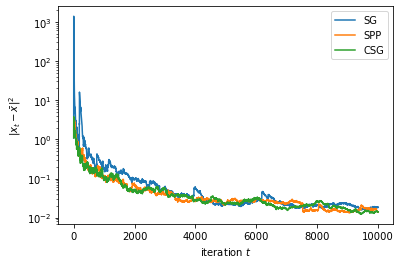}
			\caption{$\gamma=0.1$.}
		\end{subfigure}%
		\begin{subfigure}{.5\textwidth}
			\centering
			\includegraphics[width=\linewidth]{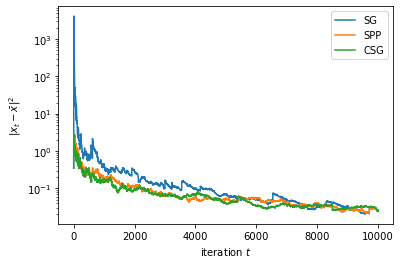}
			\caption{$\gamma=0.25$.}
		\end{subfigure}%

	\end{center}
	\caption{Stochastic gradient (SG), clipped gradient (CSG), proximal point (SPP) methods on strategic classification (Appendix~\ref{sec:strat_class}) with parameters $\alpha=100/n$, $n=2000$, and $\eta_t=\frac{2}{\alpha(t+1)}$.}
	\label{fig:models}
\end{figure}

\paragraph{Proof technique: function gap deviation \eqref{eqn:gap_comp} \&  Lyapunov analysis.}
The proofs of the outlined results rely on the interplay between the function gap inequality~\eqref{eqn:gap_comp} and typical Lyapunov arguments used in stochastic optimization. Namely, classical convergence arguments for stochastic methods on a static problem $\min_{x} \psi(x)$ rely on one-step improvement bounds of the form: 
\begin{equation}\label{eqn:one_step_improvea}
\eta_t\EE_t[\psi(x_{t+1})-\psi(x)]\leq \frac{1-c_1\eta_t}{2}\|x_t-x\|^2-\frac{1+c_2\eta_t}{2}\| x_{t+1}- x\|^2+c_3\eta_t^2\qquad \forall x.
\end{equation}
Here, $c_1,c_2,c_3\in\R$ are some constants and $\eta_t>0$ is a user-specified sequence. 
 As long as the sum $c_1+c_2$ is positive, one may drive the gap  $\EE[\psi(x_{t+1})-\psi(x)]$ below any fixed tolerance by choosing $\eta_t>0$ appropriately. The condition $c_1+c_2>0$ typically holds under strong convexity assumptions. Standard efficiency estimates follow by setting $x$ to be the minimizer of $\psi$; for our purposes, however, it is important that $x$ can be  arbitrary.

Returning to stochastic optimization with state-dependent distributions, recall that the losses themselves are assumed to be smooth and strongly convex. Therefore, stochastic methods on the static problem $\min \varphi_x$ (for any $x$) likely enjoys the estimate~\eqref{eqn:one_step_improvea} with $c_1+c_2>0$.
Imagine now that given a current iterate $x_t$ we take a single step of such an algorithm on the problem $\texttt{St}(\cD(x_t))$. Setting $x=\bar x$ (with $\psi\equiv\varphi_{x_t}$) and applying the function gap inequality implies the one-step improvement:
$$\eta_t\EE_t[\varphi(x_{t+1})-\varphi(\bar x)]\leq \frac{1-(c_1-\gamma\beta)\eta_t}{2}\|x_t-\bar x\|^2-\frac{1+(c_2-\gamma\beta)\eta_t}{2}\| x_{t+1}- \bar x\|^2+c_3\eta_t^2.$$
Therefore in the regime $c_1+c_2-2\gamma\beta>0$, we can drive the gap $\EE_t[\varphi(x_{t+1})-\varphi(\bar x)]$ to zero with an appropriate choice of $\eta_t>0$. Moreover, if $\varphi$ is $\tilde{\alpha}$-strongly convex (for some $\tilde\alpha>0$), then we may lower bound the left side by $(\eta_t\tilde\alpha/2)\EE\|x_{t+1}-\bar x\|^2$. Elementary algebraic manipulations then show that the parameter regime of convergence in $\EE\|x_{t+1}-\bar x\|^2$ improves to $c_1+c_2+\tilde\alpha-2\gamma\beta>0$.
In summary, the function gap inequality~\eqref{eqn:gap_comp} allows to translate one-step improvements on static problems $\min \varphi_{x_t}$ into one-step improvements   on the target problem $\min \varphi$. We will show that model based algorithms on well-conditioned static problems satisfy a one-step improvement bound of the form \eqref{eqn:one_step_improvea} with $c_1+c_2>0$, and then apply the outlined argument  using  \eqref{eqn:gap_comp}.

\subsection{Inexact repeated minimization  (Section~\ref{sec:inexactRM})}

All the aforementioned stochastic algorithms draw a single sample in between every change in distribution. In practice, however, modifying the sampling distribution may be much more expensive than drawing a sample from the current distribution $\cD(x_t)$. Following \cite{perdomo2020performative}, we call the process of modifying the distribution {\em deployment}. It is an interesting question if it is possible to maintain the sample efficiencies of the aforementioned algorithms for (implicitly) solving $\texttt{St}(\cD(\bar x))$ while decreasing the number of deployments. An answer to this question for the projected stochastic gradient method appears in \cite[Theorem 3.3]{perdomo2020performative}. Namely, consider running inexact repeated minimization for iterations $t=0,1,\ldots$ with a projected stochastic gradient method applied for $J_t$ iterations on each subproblem. The authors show that setting $J_t=\mathcal{O}(t^{1.1p})$ ensures that iterates $x_t$ converge to $\bar x$ at the rate $O(t^{-p})$. Here $p>0$ is a tuning parameter that controls the tradeoff between sampling and deployment.
We prove a closely related result for all the algorithms analyzed in the previous sections; concisely, we show that the number of deployments can be reduced to be logarithmic in problem parameters without sacrificing sample efficiency. 

We begin with the following theorem, which provides guarantees on the efficiency of repeated minimization with model-based algorithms used as inexact subsolvers.

\begin{theorem}[Informal]\label{thm:inex_model}
Define $\tilde \alpha=\alpha+\mu$ for stochastic proximal gradient and proximal point methods and set $\tilde \alpha=\mu$ for the clipped gradient method; in addition, set $\tilde\rho=\frac{\gamma\beta}{\tilde\alpha}$. In the regime $\tilde\rho< 1$,  the three methods may be used as inexact solvers within repeated minimization. The resulting methods will generate a point $x$ satisfying  $\EE[\varphi(x)- \varphi(\bar x)]\leq \varepsilon$ using 
\begin{equation*}
\mathcal{O}\left(\frac{1}{1-\tilde\rho}\cdot\left(\log\left(\frac{\varphi(x_0)-\varphi(\bar x)}{\varepsilon}\right)+\log\left(\frac{\sigma^2}{1-\tilde{\rho}}\cdot \frac{1}{L\varepsilon}\right)\right)\right)\qquad \textrm{deployments},
\end{equation*}	
and 
\begin{equation*}
\mathcal{O}\left(\frac{\log((1-\tilde\rho)^{-1})}{1-\tilde\rho}\cdot\left(\frac{L}{\tilde\alpha(1-\tilde\rho)}\log\left(\frac{\varphi(x_0)-\varphi(\bar x)}{\varepsilon}\right)+\frac{\sigma^2}{\tilde\alpha(1-\tilde{\rho})^2\varepsilon}\right)\right)\qquad\textrm{samples}.
\end{equation*}	
\end{theorem}

Observe that the sample complexity in Theorems~\ref{thm:inform_model} and \ref{thm:inex_model} are essentially the same. The key difference is that the number of deployments in the latter is only logarithmic in the problem parameters. As a simple illustration, 
Figure~\ref{fig:repeated_algos} depicts the performance of inexact repeated minimization on Example~\ref{ex:llustr_dumb} with the stochastic gradient method for approximately solving the inner problems. We cap the number of deployments at $20$ and adjust the number of inner iterations according to the schedule in Corollary~\ref{cor:delayed-complexities}. Experimentally, we see that online algorithms and those based on repeated minimization perform similarly on this example.

\begin{figure}[h!]
	\begin{center}
		\begin{subfigure}{.5\textwidth}
			\centering
			\includegraphics[width=\linewidth]{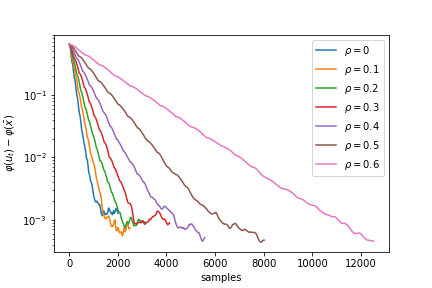}
			\caption{Functional value.}
		\end{subfigure}%
		\begin{subfigure}{.5\textwidth}
			\centering
			\includegraphics[width=\linewidth]{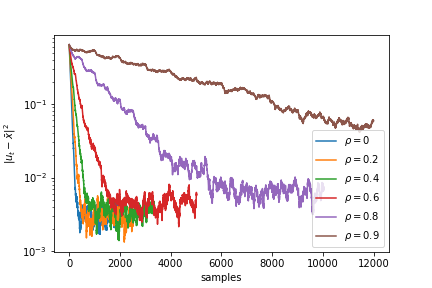}
			\caption{Square distance to equilibrium.}
		\end{subfigure}
		
	\end{center}
	\caption{Inexact repeated minimization on Example~\ref{ex:llustr_dumb}.}
	\label{fig:repeated_algos}
\end{figure}

The following theorem provides guarantees on the efficiency of repeated minimization with an accelerated stochastic gradient method used as an inexact solver.

\begin{theorem}[Informal]\label{thm:inex_accel}
 In the regime $\rho< \frac{1}{2}$,  the accelerated stochastic proximal gradient method may be used as an inexact solver within repeated minimization. The resulting method will generate a point $x$ satisfying  $\EE[\varphi(x)- \varphi(\bar x)]\leq \varepsilon$ using  
\begin{equation*}
 \mathcal{O}\left(\left(1-\tfrac{1}{2(1-\rho)}\right)^{-1}\left(\log\left(\frac{\varphi(x_0)-\varphi(\bar x)}{\varepsilon}\right)+ \log\left(\frac{\sigma^2}{(1-\frac{\rho}{1-\rho})\varepsilon\sqrt{\alpha L}}\right)\right)\right)\qquad \textrm{deployments},
\end{equation*}	
and 
\begin{equation*}
 \mathcal{O}\left(\left(1-\tfrac{1}{2(1-\rho)}\right)^{-1}\cdot\left(\sqrt{\kappa}\cdot\log\left(\frac{\varphi(x_0)-\varphi(\bar x)}{\varepsilon}\right)+\frac{\sigma^2}{(1-\tfrac{\rho}{1-\rho})\alpha\varepsilon}\right)\right)\qquad\textrm{samples}.
\end{equation*}	
	\end{theorem}

Comparing Theorems~\ref{thm:informal_accel_grad} and \ref{thm:inex_accel}, we see that the latter takes hold in the nearly optimal parameter regime $\rho<\frac{1}{2}$ without the extra factor of $\kappa^{-1/4}$. In addition, the number of deployments  is only logarithmic in the problem parameters. The proof strategy for Theorems~\ref{thm:inex_model} and \ref{thm:inex_accel} 
again relies heavily on using the function gap inequality \eqref{eqn:gap_comp} to perturb Lyapunov type arguments.

\bigskip

Sections~\ref{sec:prox_point_prox_grad}-\ref{sec:inexactRM} formally justify the results outlined in this section.

\section{Calm and contractive methods}\label{sec:prox_point_prox_grad}
In this section, we analyze two conceptual algorithms with state-dependent distributions:
\begin{center}
	\begin{tabular}{ l l} 
		{\bf (Proximal point)} &\qquad $\displaystyle x_{t+1}=\argmin_x \,\left\{ f_{x_{t}}(x)+r(x)+\frac{1}{2\eta}\|x-x_{t}\|^2 \right\}$,  \\
		\\
		{\bf (Proximal gradient)} & \qquad $\displaystyle x_{t+1}=\prox_{\eta r}(x_{t}-\eta\nabla f_{x_{t}}(x_{t}))$.  \\ 	
	\end{tabular}
\end{center}
Thus in each iteration $t$, the two methods simply take a proximal point and proximal gradient steps, respectively, on the static problem $\texttt{St}(\cD(x_t))$. The proximal point method in the extreme case $\eta=\infty$ is called {\em repeated minimization} in \cite{perdomo2020performative}.
The paper \cite{perdomo2020performative} showed that when $r$ is the indicator function of a closed convex set and $\rho<1$, repeated minimization and the proximal gradient method converge linearly to $\bar x$. In this section, we provide a different and complimentary viewpoint based on stability to distributional shifts.

\subsection{An interlude: calmness of algorithms for $\mathtt{St}(\nu)$}
We begin with an interlude quantifying the stability of algorithmic updates on the parametric family of problems $\mathtt{St}(\nu)$ with $\nu\in \mathbb{P}$.
To this end,  let  $S_{\nu}(x)$ denote an update of a point $x$ by an algorithm on the static problem $\mathtt{St}(\nu)$. Recall the notation $f_{\nu}(x)=\ee_{z\sim \nu} \ell(x,z)$ from Section~\ref{sec:distrib_shift}. Table~\ref{table:stab_algos} lists  three basic examples that are worth keeping in mind: full minimization, proximal-point, and proximal gradient updates. A desirable property of an algorithm is that for any fixed $x$, the update map $S_{\nu}(x)$ is Lipschitz continuous with respect to variations in $\nu$. 

The following definition summarizes this stability property, relative to perturbations of a fixed distribution $\mu$, which will later correspond to the equilibrium distribution $\cD(\bar x)$.

\begin{defn}[Calmness]
	{\rm
		Fix a distribution $\mu\in \mathbb{P}$. We say that a map $S_{\nu}(x)$ is {\em $\tau$-calm relative to} $\mu$ if the estimate holds:
		$$\|S_{\nu}(x)-S_{ \mu}(x)\|\leq \tau\cdot W_1(\nu,\mu),\qquad \forall \nu\in \mathbb{P}.$$ 	 }
\end{defn}

Without further assumptions, standard algorithms can easily fail to be calm even if the loss function is convex.\footnote{Consider the univariate function $l(x,z)=\tfrac{1}{4}x^4+zx$ and define $\mu$ and $\nu_k$ to be point masses at zero and $\frac{1}{k}$, respectively. Set $x_{\mu}=\argmin f_{\mu}$ and 
	$x_{\nu_k}=\argmin f_{\nu_k}$.
	A quick computation shows $|x_{\mu}-x_{\nu_k}|/W_1(\mu,\nu_k)=k^{2/3}\to\infty$ as $k\to\infty$.}
As a remedy, we impose strong convexity assumptions on the problem data.

\begin{assumption}[Strong convexity]\label{assump:2}
	Suppose that $r$ is closed and convex and that  $f_{\mu}$ is $\alpha$-strongly convex for some distribution $\mu\in\mathbb{P}$ and some constant $\alpha>0$.
\end{assumption}

Theorem~\ref{thm:stab_minppprox} verifies that under Assumptions~\ref{assum:smoothness} and \ref{assump:2}, the three basic updates in Table~\ref{table:stab_algos} are indeed calm relative to $\mu$. The fourth column in Table~\ref{table:stab_algos} also lists the well known contraction factors of the updates  $x\mapsto S_{\mu}(x)$.

\begin{table}[t]
\renewcommand{\arraystretch}{1.3}
	\begin{center}
		\begin{tabular}{| l | c| c |  c |}
			\hline
			Updates & $\displaystyle S_{\nu}(x)=\argmin_y~ \{\ldots\}$ & $\begin{aligned} & {\rm Calmness}\\ 
			&~~\tau>0\end{aligned}$ & $\begin{aligned} &{\rm Contraction}\\ &~~q\in [0,1)\end{aligned}$ \\ 
			\hline
			Minimization  & $f_{\nu}(y)+r(y)$ & $\beta/\alpha$ & $0$ \\  
			Prox-point & $f_{\nu}(y)+r(y)+\frac{1}{2\eta}\|y-x\|^2$  & $\frac{\eta\beta}{1+\eta\alpha}$ &  $\frac{1}{1+\eta\alpha}$  \\
			Prox-gradient  & $\langle \nabla f_{\nu}(x),y\rangle+r(y)+\frac{1}{2\eta}\|y-x\|^2$ & $\eta\beta$ & $\sqrt{1-\alpha\eta}$ ~~($\eta\leq 1/L$)  \\
			\hline
		\end{tabular}
	\end{center}
	\caption{Calm and contractive updates under Assumption~\ref{assum:smoothness}, \ref{assump:2}: first column lists the names of the algorithms, second column specifies the updates, third column lists the calmness constants, the fourth columns lists the contraction factor of the map $x\mapsto S_{\mu}(x)$.}\label{table:stab_algos}
\end{table}

\begin{theorem}[Calmness of the updates]\label{thm:stab_minppprox}
	Suppose that Assumptions~\ref{assum:smoothness} and \ref{assump:2} hold. Then the updates of the repeated minimization, proximal point, and proximal gradient methods are $\tau$-stable for the constants $\tau>0$ appearing in Table~\ref{table:stab_algos}.
\end{theorem}
\begin{proof}
	Our goal is to establish an upper bound $\sup_{x}\|S_{\nu}(x)-S_{\mu}(x)\|\leq \tau W_1(\mu,\nu)$ for the three algorithms. Consequently, let us fix a point $x$ throughout the proof and	
	define the updates corresponding to repeated minimization, proximal point, and proximal gradient updates, respectively:
	\begin{align*}
	S^{1}_{\nu}(x)&=\argmin_y\,\bigl\{f_{\nu}(y)+r(y)\bigr\},\\
	S^{2}_{\nu}(x)&=\argmin_y\,\bigl\{ f_{\nu}(y)+r(y)+\tfrac{1}{2\eta}\|y-x\|^2\bigr\},\\ 
    S^{3}_{\nu}(x)&=\prox_{\eta r}\bigl(x-\eta\nabla f_{\nu}(x)\bigr).
	\end{align*}
	We first verify the calmness constant for repeated minimization, $i=1$. To this end, define the function $\varphi_{\nu}(y):=f_{\nu}(y)+r(y)$ and its minimizer $y_{\nu}:=\argmin \varphi_{\nu}$. First-order optimality conditions guarantee the inclusions $0\in \partial \varphi_{\mu}(y_{\mu})$ and $0\in\partial \varphi_{\nu}(y_{\nu})=\nabla f_{\nu}(y_{\nu})+ \partial r(y_{\nu})$.  
In other words, there exist $\xi\in\partial r(y_\nu)$ such that $\nabla f_{\nu}(y_{\nu})+\xi=0$.

On the other hand, strong convexity of $\varphi_{\mu}$ guarantees $\alpha\|y-y'\|\leq \|w-w'\|$ for all $w\in \partial \varphi_{\mu}(y)$ and $w'\in \partial \varphi_{\mu}(y')$. We set 
$$y=y_{\mu}, \quad w=0\in\partial\varphi_{\mu}(y_{\mu}),\quad y'=y_{\nu}, \quad w'=\nabla f_{\mu}(y_{\nu})+\xi\in\partial\varphi_{\mu}(y_{\nu}).$$ 
Since $\nabla f_{\nu}(y_{\nu})+\xi=0$, we have $w'=\nabla f_{\mu}(y_\nu)-\nabla f_{\nu}(y_\nu)$ and thus can deduce 
	$$\alpha\|y_{\mu}-y_{\nu}\|\leq 
	\|\nabla f_{\mu}(y_\nu)-\nabla f_{\nu}(y_\nu)\|\leq\beta\cdot W_{1}(\mu,\nu),$$
where the last inequality follows from Lemma~\ref{lem:grads_close}. Noticing that $S_{\nu}(x)=y_{\nu}$ and $S_{\mu}(x)=y_{\mu}$, we thus arrive at the claimed estimate $\tau=\beta/\alpha$ for the case $i=1$. 

Calmness of the proximal point ($i=2$) and proximal gradient ($i=3$) updates follow  by applying what we have already proved (case $i=1$) but with the different loss functions 
\begin{align*}
\varphi^{2}_{\nu}(y)&=f_{\nu}(y) + r(y) +\frac{1}{2\eta}\|y-x\|^2,\\
\varphi^{3}_{\nu}(y)&=\langle \nabla f_{\nu}(x),y\rangle+r(y)+\frac{1}{2\eta}\|y-x\|^2,
\end{align*}
and recognizing $S^{i}_{\nu}(x)$ as the minimizers
$S^{i}_{\nu}(x)=\argmin_{y}\,\varphi^{i}_{\nu}(y)$.
\end{proof}

\subsection{Linear convergence of conceptual algorithms}
We next pass to the setting where the distribution governing the data is state-dependent.
To this end, suppose that Assumptions~\ref{assum:smoothness}, \ref{assum:perm_pred}, \ref{ass:strong_conv_perm}\ref{it:str_conv1}, \ref{ass:smoothness}\ref{assump:smooth_center} hold.
We will show that linear convergence of repeated minimization, proximal point, and proximal gradient methods is a direct consequence of the two independent phenomenon:
\begin{enumerate}
	\item {\bf (Calm)} The updates are $\tau$-calm (Theorem~\ref{thm:stab_minppprox}).
	\item {\bf (Contractive)} The algorithms are $q$-contractive on the static problem $ \mathtt{St}(\cD(\bar x))$. That is,
 $\|S_{\bar x}(x)-S_{\bar x}(y)\|\leq q\|x-y\|$ for all $x,y\in\dom r$.  
\end{enumerate}

It is elementary to see that an algorithm satisfying these two properties is automatically  $(q+\tau\gamma)$-contractive under state dependent sampling. This is the content of the following lemma.

\begin{lemma}[Calm and contractive]\label{lem:sta_contr}
	Fix a map $S_{\nu}(x)$ that is $\tau$-calm and such that  the map $x\mapsto S_{\cD(\bar x)}(x)$ is q-contractive with $\bar x$ as its fixed point. Then the estimate holds:
	$$\|S_{\cD(x)}(x)-\bar x\|\leq (q+\tau\gamma)\|x-\bar x\|\qquad \qquad\forall x\in \R^d.$$

\end{lemma}
\begin{proof}
	Abusing  notation slightly and setting $S_{x}(y):=S_{\cD(x)}(y)$, we compute
	\begin{align}
	\|S_{x}(x)-\bar x\|=\|S_{x}(x)-S_{\bar x}(\bar x)\|&\leq \|S_{x}(x)-S_{\bar x}(x)\|+\|S_{\bar x}(x)-S_{\bar x}(\bar x)\|\label{eqn:triangle_ab}\\
	&\leq \tau\cdot W_1(\cD(x),\cD(\bar x))+q\|x-\bar x\|\label{eqn:tauq}\\
	&\leq (q+\tau\gamma) \|x-\bar x\|.\label{eqn:eps_assum}
	\end{align}
	Here \eqref{eqn:triangle_ab} uses the triangle inequality, \eqref{eqn:tauq} follows from calmness and contractiveness, while \eqref{eqn:eps_assum} follows from Assumption~\ref{assum:perm_pred}.
	The proof is complete.
\end{proof}

Combining Lemma~\ref{lem:sta_contr} with the constants $\tau$ and $q$ specified in Table~\ref{table:stab_algos} immediately yields linear convergence guarantees for the three conceptual algorithms.

\begin{corollary}[Repeated minimization, proximal point, and proximal gradient methods]\label{cor:lin_conv}
	 Suppose that Assumptions~\ref{assum:smoothness}, \ref{assum:perm_pred}, \ref{ass:strong_conv_perm}\ref{it:str_conv1}, \ref{ass:smoothness}\ref{assump:smooth_center} hold.
Then for any point $x$, the following estimates hold:
	$$
	\frac{\|x^+-\bar x\|}{\|x-\bar x\|}\leq \begin{cases} \frac{\gamma\beta}{\alpha} &\mbox{if } x^+=\argmin_y\,\bigl\{f_{x}(y)+r(y)\bigr\}, \\[1ex]
	\frac{1+\gamma\eta\beta}{1+\eta\alpha} & \mbox{if } x^+=\argmin_y\, \bigl\{f_x(y)+r(y)+\frac{1}{2\eta}\|y-x\|^2 \bigr\},\\[1ex]
	\sqrt{1-\eta\alpha}+\gamma\eta\beta & \mbox{if } x^+=\prox_{\eta r}(x-\eta\nabla f_{x}(x))~ \textrm{and}~ \eta\leq \frac{1}{L}. \end{cases}
	$$
	Thus iterated minimization and the prox-point methods converge linearly to $\bar x$ in the regime $\rho<1$, while the prox-gradient method converges linearly to $\bar x$ in the regime $\rho<\frac{1}{2}$.
\end{corollary}

Corollary~\ref{cor:lin_conv} shows that iterated minimization and  proximal point methods  converge linearly in the parameter regime $\rho <1$. Moreover, as shown in \cite{perdomo2020performative}, repeated minimization can easily diverge outside this parameter regime. 
The regime of convergence  $\rho<\frac{1}{2}$ for the proximal gradient method therefore appears slightly suboptimal.
This regime can be trivially enlarged when no regularization is present, i.e., when $r=0$. Indeed, with the choice $\eta=\frac{2}{\alpha+L}$, the gradient method is $\frac{L-\alpha}{L+\alpha}$ contractive on the static problem $\mathtt{St}(\cD(\bar x))$ (e.g. \cite[Theorem 3.12]{bubeck2014convex}). Lemma~\ref{lem:sta_contr} therefore guarantees that the update $x^+$ satisfies 
$$
\frac{\|x^+-\bar x\|}{\|x-\bar x\|}\leq \frac{L-\alpha}{L+\alpha}+\frac{2\gamma\beta}{L+\alpha}.$$
Clearly, the right side is smaller than one if and only if $\rho<1$. A similar guarantee based on a different argument appears in \cite[Proposition 2.5]{mendler2020stochastic}.

In general regularized settings, the regime of convergence of the proximal gradient method can be enlarged to $\rho<1$ through a  different argument. See the forthcoming guarantees for the proximal stochastic gradient method (Theorem~\ref{thm:rate_grad-desc} with  $\sigma^2=0$).

\section{Reduction to online convex optimization}\label{sec:red_online}
This section shows that virtually any algorithm designed for online convex optimization can be applied under state-dependent sampling. We begin with a short summary of online convex optimization, and refer the reader to the surveys \cite{hazan2016intro,shalevshwartz2012online} for further  details.

\subsection{Review of online convex optimization}
The framework of online convex optimization can be interpreted as a repeated game. At each iteration~$t$, the player chooses a point $x_t$ from a convex set $K\subset\R^d$, then a convex cost function $\ell_t:K\to\R$ is revealed and the player incurs the cost $\ell_t(x_t)$.
The goal of the player is to minimize the {\em regret}, defined as the difference between the total cost $\sum_{i=1}^t\ell_i(x_i)$ incurred up to round~$t$ and the minimum cost of any fixed decision from hindsight, $\min_{x\in K}\sum_{i=1}^t \ell_i(x)$.
Here we consider a regularized version of online convex optimization and define the regret up to round~$t$ as  
\[
    R_t := \sum_{i=1}^t \bigl(\ell_i(x_i)+r(x_i)\bigr) - \min_x \sum_{i=1}^t \bigl(\ell_i(x)+r(x)\bigr) ,
\]
where $r$ is a convex function that represents the indicator function of the set~$K$ or a more general regularization. 
In this setting, we can use the online proximal gradient method,
\begin{equation}\label{eqn:online-prox-grad}
    x_{t+1} = \prox_{\eta_t r} (x_t - \eta_t \nabla\ell_t(x_t)),
\end{equation}
as proposed in \cite{DuchiSinger2009}.
Other suitable algorithms include the regularized dual average method \cite{xiao2010jmlr} and variants of FTRL (Follow-The-Regularized-Leader) method \cite{McMahan2017}.

The framework of online convex optimization is more general than stochastic optimization, since the loss function $\ell_t$ at each iteration may not follow any probability distribution and can even be adversarial.
On the other hand, it requires more strict assumptions in order to have meaningful regret analysis.
In particular, the set~$K$ needs be bounded, say with diameter~$D$, and the gradient of the cost functions are also bounded by a constant~$G$, i.e., 
$\|\nabla\ell_t(x)\|\leq G$ for all $x\in K$ and all~$t$.
Under these assumptions, the methods mentioned above have bounded regret against any reference point $x\in\dom r$:
\begin{equation}\label{eqn:uni_regret_bound}
    R_t(x) := \sum_{i=1}^t \left( \bigl(\ell_i(x_i)+r(x_i)\bigr) -  \bigl(\ell_i(x)+r(x)\bigr)\right)\leq U_t,
\end{equation}
where $U_t=\cO(GD\sqrt{t})$ for convex losses with $\eta_t=D/(G\sqrt{t})$ and $U_t=\cO((G^2/\alpha)\ln(t))$ for $\alpha$-strongly convex losses with $\eta_t=1/(\alpha t)$.
See \cite{DuchiSinger2009}, \cite{xiao2010jmlr} and \cite{McMahan2017} for the details.

\subsection{Reduction}
We claim that one can directly apply online optimization algorithms in the setting of state-dependent sampling, and derive their convergence rate from the  regret bounds \eqref{eqn:uni_regret_bound}. As usual, we let $\bar x$ be an equilibrium point with respect to $\cD(\cdot)$, set  $\varphi_{x}(y)=\EE_{z\sim\cD(x)}[\ell(y,z)]+r(y)$, and use the shorthand  $\varphi=\varphi_{\bar x}$.

\begin{theorem}[Reduction]\label{thm:reduction}
	Suppose that  Assumptions~\ref{assum:smoothness}, \ref{assum:perm_pred}, \ref{ass:strong_conv_perm}\ref{it:str_conv1} hold.
Consider an online algorithm that in each iteration $t\geq 1$ encounters the loss function  $\ell_t(x_t)=\ell(x_t, z_t)$ where $z_t\sim \mathcal{D}(x_t)$. Suppose that the uniform regret bound \eqref{eqn:uni_regret_bound} holds for each $t\geq 1$. Then the average iterate  $\hat{x}_t:=\frac{1}{t}\sum_{i=1}^t x_i$ satisfies 
\[
\EE[\varphi(\hat{x}_t)] - \varphi(\bar{x}) ~\leq~ \frac{U_t}{(1-2\rho)\,t}\qquad \forall t\geq 1.
\]
\end{theorem}
\begin{proof}
Taking the expectation in \eqref{eqn:uni_regret_bound} and setting $x=\bar x$ yields
\[
    \EE\left[R_t(\bar x)\right] = \sum_{i=1}^t \EE\bigl[\varphi_{x_i}(x_i) - \varphi_{x_i}(\bar x)\bigr] \leq U_t.
\]
The function gap inequality \eqref{eqn:gap_comp} in turn implies
\[
    \varphi(x_i)-\varphi_{\bar x}(\bar{x}) \leq \varphi_{x_i}(x_i) - \varphi_{x_i}(\bar{x}) + \gamma\beta\cdot\|x_i-\bar{x}\|^2.
\]
Strong convexity implies $\|x_i-\bar{x}\|^2 \leq \frac{2}{\alpha}(\varphi(x_i)-\varphi(\bar{x}))$ and hence    
\[
    (1-2\rho) \bigl(\varphi(x_i)-\varphi(\bar{x})\bigr) \leq \varphi_{x_i}(x_i) - \varphi_{x_i}(\bar{x}).
\]
Therefore, in the setting $\rho<1/2$, we have
\[
\EE\biggl[ \sum_{i=1}^t \bigl(\varphi(x_i) - \varphi(\bar x)\bigr) \biggr]
~\leq~ \frac{1}{1-2\rho} \EE[R_t(\bar{x})] ~\leq~ \frac{1}{1-2\rho} U_t .
\]
Dividing both sides by~$t$ and using convexity of $\varphi$ completes the proof.
\end{proof}

Known regret bounds for the proximal gradient \cite{DuchiSinger2009}, regularized dual averaging \cite{xiao2010jmlr}, and FTRL  \cite{McMahan2017} algorithms  directly yield a converge rate $O(\log(t)/(1-2\rho)t)$ under state-dependent sampling and under the assumptions used in the aforementioned papers.  

Though Theorem~\ref{thm:reduction} is attractive in its generality, it is far from satisfactory. Indeed, the framework of online convex optimization is far more general than stochastic optimization, since the loss function $\ell_t$ at each iteration may not follow any probability distribution and can even be adversarial.
As a consequence, online convex optimization requires stringent assumptions in order to have meaningful regret analysis, most notably boundedness of the encountered gradients of the loss functions. Moreover smoothness of the loss function does not play a significant role. 
We will see in the next section that much finer convergence guarantees hold for stochastic optimization under state-dependent distributions.

\section{Stochastic gradient methods}
\label{sec:biased_oracle}
In this section, we directly analyze the (accelerated) stochastic proximal gradient method, without relying on regret bounds. The only consequence of state-dependent distribution that will be relevant is that the stochastic estimator of the gradient is biased and the bias is bounded as in~\eqref{eqn:grad_const2}. 
As a result, we work with a broader model with a biased stochastic gradient oracle, formalized next in Assumption~\ref{ass:bias_var}.

\begin{assumption}[Biased stochastic gradient oracle]\label{ass:bias_var}
Consider the optimization problem 
$$\min_{x} ~\varphi(x):=f(x)+r(x),$$
and denote its minimizer by  $\bar x$.
We suppose that there exist constants $L,B,\alpha>0$ satisfying the following.
\begin{enumerate}[label=(\alph*)]
\item The function $f\colon\R^d\to\R$ is $\alpha$-strongly convex  and differentiable with $L$-Lipschitz gradient.
\item The function $r\colon\R^d\to\R\cup\{\infty\}$ is closed and convex. 
\item\label{it:oracle3} For every point $x$, we may draw a realization of a random vector  $g(x)\in \R^d$ satisfying the bias/variance bounds:
$$\|\EE[g(x)]-\nabla f(x)\|\leq B \|x-\bar x\|\qquad \textrm{and}\qquad \EE\|g(x)-\EE[g(x)]\|^2\leq \sigma^2.$$
\end{enumerate}
Throughout, we define the condition number $\kappa:=\frac{L}{\alpha}$ and set $\rho:=\frac{B}{\alpha}$.
\end{assumption}

It is clear that Assumptions~\ref{assum:smoothness}, \ref{assum:perm_pred}, \ref{ass:strong_conv_perm}\ref{it:str_conv1} and \ref{ass:smoothness}\ref{assump:smooth_center} under the setting of state-dependent distributions imply Assumption~\ref{ass:bias_var} with $g(x):=\nabla \ell(x,z)$ where $z\sim\cD(x)$, and $B:=\gamma\beta$. Consequently, all convergence guarantees presented in this section can be interpreted under the setting of state-dependent distributions by simply setting $B=\gamma\beta$.

The main consequence of the particular form of the bias in Assumption~\ref{ass:bias_var} \ref{it:oracle3} is summarized in the following lemma. In essence, the lemma states that a strong convexity type inequality holds between $x$ and $\bar x$, with $\nabla f(x)$ replaced by $\EE[g(x)]$.

\begin{lemma}[Approximate subgradient inequality]\label{lem:approx_subgrad_ineq}
Under Assumption~\ref{ass:bias_var}, it holds that
	$$f(\bar x)\geq f(x)+\langle \EE[g(x)],\bar x-x\rangle+\frac{\alpha(1-2\rho)}{2}\|x-\bar x\|^2\qquad \forall x\in\R^d.$$
\end{lemma}
\begin{proof}
	Strong convexity of $f$ guarantees 
	\begin{align*}
	f(\bar x)&\geq f(x)+\langle \nabla f(x),\bar x-x\rangle+\frac{\alpha}{2}\|x-\bar x\|^2\\
	&=f(x)+\langle \EE[g(x)],\bar x-x\rangle+\frac{\alpha}{2}\|x-\bar x\|^2 +\langle \nabla f(x)-\EE[g(x)],\bar x-x\rangle.
	\end{align*}
The Cauchy-Schwarz inequality in turn yields the estimate 
	$$\langle \nabla f(x)- \EE[g(x)],\bar x-x\rangle\geq -\|\EE[g(x)]-\nabla f(x)\|\cdot \|\bar x-x\|\geq -B \|x-\bar x\|^2.$$
Combining the two inequalities above yields the desired result.
\end{proof}

\subsection{Stochastic proximal gradient method}
We are now ready to analyze the stochastic proximal gradient method, summarized in Algorithm~\ref{alg:sgd}, under Assumption~\ref{ass:bias_var}.

\begin{algorithm}[H]
{\bf Input:} initial $x_0$ and sequence $\{\eta_t\}_{t=0}^T\subset (0,\infty)$.

{\bf Step} $t=0,1,\ldots, T$: 
\begin{equation*}
\begin{aligned}
&\textrm{Sample~} g_t=g(x_t)\\
& \textrm{Set~} x_{t+1}=\prox_{\eta_t r}(x_t-\eta_t g_t)
\end{aligned}
\end{equation*}

 \caption{Stochastic gradient method}\label{alg:sgd}
\end{algorithm}

Throughout, we let $\EE_t[\cdot]$ be the conditional expectation given the iterates $x_0,\ldots,x_{t}$. The following lemma provides a key recursion quantifying the one-step progress of the algorithm. The argument closely parallels the proof in the unbiased setting \cite{ghadimi2012optimal}.

\begin{lemma}[Key recursion]\label{eqn:lem_key_rec_sgd}
Suppose that Assumption~\ref{ass:bias_var} holds and the step-size sequence satisfies $\eta_t<\frac{1}{L}$.  Then the iterates $\{x_t\}$ generated by Algorithm~\ref{alg:sgd} satisfy
\begin{align*}
2\eta_t\EE_t[\varphi(x_{t+1})-\varphi(\bar x)]&\leq \left(1-\alpha\eta_t(1-2\rho)+\frac{\eta_t^2 B^2}{1-\eta_t L}\right)\|x_t-\bar x\|^2  -  \EE_t\|x_{t+1}-\bar x\|^2 +\frac{\eta_t^2 \sigma^2}{1-\eta_t L}.
\end{align*}
\end{lemma}
\begin{proof}
Since by Assumption~\ref{ass:bias_var} the gradient $\nabla f$ is $L$-Lipschitz, we have
\begin{align*}
\varphi(x_{t+1}) &=f(x_{t+1}) + r(x_{t+1}) \notag\\
&\leq f(x_{t}) + \langle \nabla f(x_t), x_{t+1}- x_{t}\rangle + r(x_{t+1}) + \frac{L}{2} \|x_{t+1} - x_t\|^2 \\
&=  f(x_{t}) + \langle g_t, x_{t+1} - x_t\rangle + r(x_{t+1}) + \frac{L}{2} \|x_{t+1} - x_t\|^2 - \langle g_t - \nabla f(x_t),  x_{t+1}-x_t\rangle.
\end{align*}
Let $\delta_t>0$ be an arbitrary positive sequence.
We use Young's inequality to bound the last inner-product term in the above inequality, which results in
\begin{align}
\varphi(x_{t+1}) 
&\leq  f(x_{t}) + \langle g_t, x_{t+1} - x_t\rangle + r(x_{t+1}) + \frac{L}{2} \|x_{t+1} - x_t\|^2  \notag\\
&\quad + \frac{\delta_t}{2}\| g_t - \nabla f(x_t)\|^2 + \frac{1}{2\delta_t}\| x_{t+1} - x_t\|^2 \notag\\
&=  f(x_{t}) + \langle g_t, x_{t+1} - x_t\rangle + r(x_{t+1}) + \frac{1}{2\eta_t} \|x_{t+1} - x_t\|^2  \notag\\
&\quad + \frac{\delta_t}{2}\| g_t - \nabla f(x_t)\|^2 + \frac{\delta_t^{-1}-\eta^{-1}_t+L}{2}\| x_{t+1} - x_t\|^2\notag\\
&\leq f(x_t) +\langle g_t,\bar x-x_t\rangle + r(\bar x) + \frac{1}{2\eta_t}\|\bar x-x_t\|^2  - \frac{1}{2\eta_t} \|x_{t+1} - \bar x\|^2\notag  \\
&\quad + \frac{\delta_t}{2}\| g_t - \nabla f(x_t)\|^2 + \frac{\delta_t^{-1}-\eta^{-1}_t+L}{2}\| x_{t+1} - x_t\|^2, \label{eqn:three_point_ineq}
\end{align}
where the last inequality~\eqref{eqn:three_point_ineq} follows from the fact that $x_{t+1}$ is by construction the minimizer of the $\eta^{-1}_t$-strongly convex function $f(x_t) +\langle g_t,\cdot-x_t\rangle + r + \frac{1}{2\eta_t}\|\cdot-x_t\|^2$.

Next set $\delta_t:= \frac{\eta_t}{1-\eta_t L}$ to make the last term of \eqref{eqn:three_point_ineq} zero.
Taking conditional expectations and using Lemma~\ref{lem:approx_subgrad_ineq}, we deduce 
\begin{align}
\EE_t[\varphi(x_{t+1})]
&\leq  \varphi(\bar x)  -\frac{\alpha(1-2\rho)}{2}\|x_t-\bar x\|^2 \notag \\
&\quad + \frac{1}{2\eta_t}\|\bar x-x_t\|^2  - \frac{1}{2\eta_t} \EE_t\|\bar x-x_{t+1}\|^2 +\frac{\delta_t}{2}\EE_t\| g_t - \nabla f(x_t)\|^2.\label{eqn:getit_var}
\end{align}
Next, we upper bound the last term in \eqref{eqn:getit_var} as follows:
$$\EE_t\| g_t - \nabla f(x_t)\|^2=\EE\|g_t-\EE[g_t]\|^2+\|\EE[g_t]-\nabla f(x_t)\|^2\leq \sigma^2+B^2\|x_t-\bar x\|^2.$$
Combining this estimate with \eqref{eqn:getit_var} and grouping like terms completes the proof.
\end{proof}

With Lemma~\ref{eqn:lem_key_rec_sgd} at hand, obtaining convergence guarantees under various choices of the control sequence $\eta_t>0$ is completely standard. We highlight one such result based on using a constant sequence.

\begin{theorem}[Constant step size]\label{thm:rate_grad-desc}
Suppose Assumption~\ref{ass:bias_var} holds.
Let $x_t$ be the iterates generated by Algorithm~\ref{alg:sgd} with a fixed parameter $\eta>0$. Then the following are true.
\begin{enumerate}
\item Suppose we are in the regime $\rho<1$ and set $\hat \alpha:=\alpha-B$. Then with the parameter $\eta\leq \frac{1}{B^2/\hat\alpha+L}$, the estimate holds:
\begin{equation}\label{eqn:distance_decay_prox0}
\EE\|x_{t}-\bar x\|^2\leq \left(1-\frac{2\eta\hat\alpha}{3}\right)^{t}\EE\|x_0-\bar x\|^2 +\frac{2\sigma^2 \eta}{\hat\alpha}.
\end{equation}

\item  Suppose we are in the regime $\rho<\frac{1}{2}$ and set $\hat \alpha:=\alpha-2B$. Then with the parameter $\eta\leq \frac{1}{2(B^2/\hat\alpha+L)}$,  the estimate holds:
 \begin{equation}\label{eqn:sgd_const}
\EE [ \varphi\left(\hat x_t\right)-\varphi(\bar x)]\leq  2\left(1- \frac{\eta\hat\alpha}{2}\right)^t\left( \varphi(x_0)-\varphi(\bar x)\right)+\sigma^2\eta.
\end{equation}
	where the average iterate is defined recursively by $\hat x_t=(1-\frac{\hat\alpha\eta}{2})\hat x_{t-1}+\frac{\hat\alpha\eta}{2} x_t$. 	
\end{enumerate}
\end{theorem}
\begin{proof}
Observe first that the assumption $\eta\leq\frac{1}{2L}$ implies $\frac{\eta^2\sigma^2}{1-\eta L}\leq 2\eta^2\sigma^2$. Therefore Lemma~\ref{eqn:lem_key_rec_sgd} yields the one step improvement:
\begin{equation}\label{eqn:,ain_ineq_we need}
2\eta\EE[\varphi(x_{t+1})-\varphi(\bar x)]\leq \left(1-\alpha\eta(1-2\rho)+\frac{\eta^2 B^2}{1-\eta L}\right)\EE\|x_t-\bar x\|^2  -  \EE\|x_{t+1}-\bar x\|^2 +2\eta^2 \sigma^2.
\end{equation}
{{\em Proof of claim 1:}} Lower bounding the left side of \eqref{eqn:,ain_ineq_we need} using strong convexity, $\varphi(x_{t+1})-\varphi(\bar x)\geq \frac{\alpha}{2}\|x_{t+1}-\bar x\|^2$, and rearranging yields the guarantee:
	$$\EE\|x_{t+1}-\bar x\|^2\leq \left(1-\frac{2\alpha\eta(1-\rho)-\frac{\eta^2 B^2}{1-\eta L}}{1+\eta_t\alpha}\right)\EE\|x_t-\bar x\|^2  +\frac{2\eta^2 \sigma^2}{1+\eta\alpha}.$$
The assumption $\eta<\frac{1}{B^2/\hat\alpha +L}$ directly implies $\frac{\eta^2 B^2}{1-\eta L}\leq \alpha\eta(1-\rho)$, and  therefore we deduce
$$\EE\|x_{t+1}-\bar x\|^2\leq \left(1-\frac{\alpha\eta(1-\rho)}{1+\eta\alpha}\right)\EE\|x_t-\bar x\|^2  +\frac{2\eta^2 \sigma^2}{1+\eta\alpha}.$$
Unrolling the recursion and using the estimate $\eta\alpha\leq \frac{1}{2}$ completes the proof of \eqref{eqn:distance_decay_prox0}.

\noindent {\em Proof of Claim 2:}
	The assumption $\eta< \frac{1}{2(B^2/\hat \alpha+L)}$ by construction guarantees 
	$\frac{\eta^2 B^2}{1-\eta L}\leq \frac{\hat \alpha\eta}{2}.$ Therefore, the inequality \eqref{eqn:,ain_ineq_we need} implies
\begin{align*}
2\eta\EE_t[\varphi(x_{t+1})-\varphi(\bar x)]&\leq \left(1-\frac{\hat\alpha\eta}{2}\right)\|\bar x-x_t\|^2  -  \EE_t\|\bar x-x_{t+1}\|^2 +2\eta^2 \sigma^2.
\end{align*}
	Applying Corollary~\ref{cor:const_param} for the recursion 
	completes the proof of \eqref{eqn:sgd_const}.
\end{proof}

Let us now translate Theorem~\ref{thm:rate_grad-desc} into an efficiency guarantee for finding an approximate minimizer of $\varphi$. There are two standard restarting techniques that achieve this goal.
The first is based on restarting the constant step algorithm with exponentially increasing mini-batches of gradients $\frac{1}{m}\sum_{i=1}^m g_t^i$ in order to decrease the variance. The second approach restarts the constant step algorithm with geometrically decreasing parameters~$\eta$.
Both techniques are standard and are detailed in Appendix~\ref{sec:restart_improv_eff}. The following corollary focuses on the latter strategy for simplicity, though both strategies enjoy the same sample complexity guarantees.

\begin{corollary}[Efficiency of Algorithm~\ref{alg:sgd} with geometrically decaying schedule of~$\eta$]\label{cor:main_eff_result} {\hfill \\ }
Suppose Assumption~\ref{ass:bias_var} holds. Then the following statements hold.
\begin{enumerate}
\item {\bf (Distance)} Suppose $\rho<1$ and that we have available an estimate $\Delta\geq \|x_0-\bar x\|^2$. Define the modified convexity parameter $\hat \alpha=\alpha-B$. Then Algorithm~\ref{alg:sgd} may be augmented with the geometric decay schedule in Algorithm~\ref{alg:geo_decay} (Appendix~\ref{sec:restart_improv_eff}) under the identification 
$$c=\frac{2\hat\alpha}{3}, \quad C=1,\quad h(x)=\|x-\bar x\|^2,\quad \delta_0=\frac{1}{B^2/\hat\alpha+L},\quad D=\frac{2\sigma^2}{\hat\alpha}.$$ 
The resulting procedure will generate a point $x$ satisfying $\EE[\|x-\bar x\|^2]\leq \varepsilon$ using 
$$\mathcal{O}\left(\left(\frac{B^2}{\hat\alpha^2}+\frac{L}{\hat \alpha}\right)\log\left(\frac{\Delta}{\varepsilon}\right)+\frac{\sigma^2}{\hat\alpha^2\varepsilon}\right)\qquad\textrm{samples}.$$
\item {\bf (Function value)} Suppose $\rho<\frac{1}{2}$ and that we have available an estimate $\Delta\geq \varphi(x_0)-\varphi(\bar x)$. Define the modified convexity parameter $\hat \alpha=\alpha-2B$. Then Algorithm~\ref{alg:sgd} may be augmented with the geometric decay schedule in Algorithm~\ref{alg:geo_decay} (Appendix~\ref{sec:restart_improv_eff}) under the identification $$c=\frac{\hat\alpha}{2}, \quad C=2,\quad h(x)=\varphi(x)-\varphi(\bar x),\quad \delta_0=\frac{1}{2(B^2/\hat\alpha+L)},\quad D=\sigma^2.$$

 The resulting procedure will generate a point $x$ satisfying $\EE[\varphi(x)-\varphi(\bar x)]\leq \varepsilon$ using 
$$\mathcal{O}\left(\left(\frac{B^2}{\hat\alpha^2}+\frac{L}{\hat \alpha}\right)\log\left(\frac{\Delta}{\varepsilon}\right)+\frac{\sigma^2}{\hat\alpha\varepsilon}\right)\qquad \textrm{samples}.$$
\end{enumerate}
\end{corollary}

\subsection{Accelerated Stochastic proximal gradient method}
We next discuss the accelerated stochastic proximal gradient method in the oracle model (Assumption~\ref{ass:bias_var}). To state the algorithm, we require a few auxiliary quantities. Setting the stage, define $\hat \alpha=\alpha-2B$ and choose an arbitrary $\gamma_0\geq\hat \alpha$ and stepsize parameters $\eta_t>0$. Define the two auxiliary sequences 
\begin{equation}\label{eqn:acc_gamma_delta}
\begin{aligned}
\delta_t&=\sqrt{\eta_t \gamma_t}&\textrm{for all }t\geq 0.\\
\gamma_t&=(1-\delta_t)\gamma_{t-1}+\delta_t\hat \alpha& \textrm{for all }t\geq 1.
\end{aligned}
\end{equation}
Algorithm~\ref{alg:sgd_acc} summarizes the accelerated stochastic gradient method, proposed in \cite{kulunchakov2019estimate}.

\begin{algorithm}[H]
{\bf Input:} Initial $x_0\in \R^d$,  $\gamma_0> 0$,  sequence $\{\eta_t\}_{t\geq 0}\in [0,\infty)$,  count $T\in \mathbb{N}$.

{\bf Step} $t=0,1,\ldots, T$: Set
\begin{equation*}
\begin{aligned}
	x_{t}&=y_{t-1}-\eta_t g_t, &\\
	y_t&=x_t+\beta_t(x_t-x_{t-1})& \qquad\textrm{where}\qquad \beta_t=\frac{\delta_t(1-\delta_t)\eta_{t+1}}{\eta_t\delta_{t+1}+\eta_{t+1}\delta_t^2}.
	\end{aligned}
\vspace{-1ex}
\end{equation*}
	\caption{Accelerated stochastic gradient method}\label{alg:sgd_acc}
\end{algorithm}

Our main result of analyzing Algorithm~\ref{alg:sgd_acc} is the following theorem. 
Notice that here the sequence $\beta_t$ is not related to $\beta$ in Assumption~\ref{assum:smoothness}.

\begin{theorem}[Constant step accelerated method]\label{thm:accel}
	Suppose that we are in the regime $\rho\leq \frac{1/2}{1+\sqrt{32+64\sqrt{3\kappa}}}$. Set $\gamma_0=\hat\alpha$ and $\eta_t=\frac{1}{4L}$ for all  $t\geq 0$. Then the iterates generated by Algorithm~\ref{alg:sgd_acc} satisfy 
	$$\EE\left[\varphi(x_{t})-\varphi(\bar x)\right]\leq 2\left(1-\sqrt{\frac{\hat\alpha}{4L}}\right)^t\left(\varphi(x_0)-\varphi(\bar x)\right)+\frac{9\sigma^2}{16\sqrt{L\hat\alpha}}.$$
\end{theorem}

Theorem~\ref{thm:accel} immediately yields an accelerated rate of convergence in the setting  $\rho=\frac{B}{\alpha}\leq \frac{1/2}{1+\sqrt{32+64\sqrt{3\kappa}}}$.
Indeed,  Lemma~\ref{lem:minbatch_restart} trivially implies that combining Algorithm~\ref{alg:sgd_acc} with a minibatch restart strategy (Algorithm~\ref{alg:restart}) yields a procedure that will find a point $x$ satisfying $\EE[\varphi(x)-\min \varphi]\leq \varepsilon$ using 
$$\mathcal{O}\left(\sqrt{\kappa}\cdot\log\left(\frac{\Delta}{\varepsilon}\right)+\frac{\sigma^2}{\alpha	\varepsilon}\right)$$
stochastic gradient samples, where $\Delta\geq \varphi(x_0)-\varphi(\bar x)$ is user specified. 

The proof of Theorem~\ref{thm:accel} is quite long and technical, relying on the machinery of stochastic estimate sequences \cite{kulunchakov2019estimate}. Therefore, we have placed it in Appendix~\ref{sec:accel_tech}. The high level idea of the argument is as follows. Existing arguments (in the unbiased setting) based on estimate sequences $d_t(x)=d_t^*+\frac{\gamma_t}{2}\|x-v_t\|^2$ rely on lower bounding  the error $\EE[d_t^*-\varphi(x_t)]$; see e.g.  \cite{intro_lect,kulunchakov2019estimate}. The usual path is through a sequence of clever algebraic manipulations. During these manipulations, there is a term $\|y_t-v_t\|^2$ that appears, which is lower-bounded by zero and ignored. In contrast, we show that this term balances the incurred bias of the stochastic gradients.

\section{Model-based algorithms}\label{sec:MBA_algos}
This section presents convergence guarantees for a wide class of algorithms of a ``proximal point type'', introduced in the two recent papers \cite{asi2019stochastic,davis2019stochastic} in the static setting.

\subsection{Model-based algorithms for a static problem}
We begin by reviewing model-based algorithms for static problems and refine the available convergence guarantees. Namely, the convergence guarantees developed in the papers \cite{asi2019stochastic,davis2019stochastic} require the second moment of subgradients to be bounded. We will show that when the loss function is smooth, we may instead assume a bound on the variance. 

Setting the stage, consider the static optimization problem
$$\min_x~ F(x)=f(x)+r(x)\qquad \textrm{with}\qquad f(x)=\ee_{z\sim \cP} \ell(x,z),$$
where $\cP$ is some probability measure, the loss $\ell(x,z)$ is differentiable in $x$, and $r\colon\R^d\to\R\cup\{\infty\}$ is a closed function.  
In each iteration $t$, the algorithms we consider draw an i.i.d set of $m$ samples $S_t \stackrel{i.i.d}{\sim}  \cP$ and approximate the loss function $\ell(\cdot,z_t)$ by a simpler model $\ell_{x_t}(\cdot,S_t)$ formed at the basepoint $x_t$. The next iterate $x_{t+1}$ is then  declared to be the minimizer of the function $\ell_{x_t}(\cdot,S_t)+r+\frac{1}{2\eta_t}\|\cdot-x_t\|^2$. The formal procedure is stated in Algorithm~\ref{alg:MBA_stat}.

\begin{algorithm}[H]
	{\bf Input:} initial $x_0$ and sequence $\{\eta_t\}_{t=0}^T\subset (0,\infty)$.
	
	{\bf Step} $t=0,1,\ldots, T$: 
	\begin{equation*}
	\begin{aligned}
	&\textrm{Sample~} S_t\stackrel{i.i.d}{\sim} \cP\\
	& \textrm{Set~} x_{t+1}=\argmin_y\,\biggl\{ \ell_{x_t}(y,S_t)+r(y)+\frac{1}{2\eta_t}\|y-x_t\|^2 \biggr\}
	\end{aligned}
	\end{equation*}
	
	\caption{Model-based algorithm for static problems}\label{alg:MBA_stat}
\end{algorithm}

Clearly, convergence guarantees of  Algorithm~\ref{alg:MBA} must depend both on the regularity of the 
models $\ell_{x}(y,S)$ individually and on how well the models approximate $f$. The following assumption formalizes these two properties. It will be useful to keep track of two parameters $\alpha_1,\alpha_2\geq 0$, which measure ``strong convexity'' type properties, along with a variance bound $\sigma_0>0$.

\begin{assumption}[Models for static problems]\label{assum:stoch_model_stat}
	There exist constants $\alpha_1,\alpha_2,\sigma_0\geq 0$ such that the following properties hold for all $x,y\in \R^d$ and for almost all samples $S\stackrel{i.i.d}{\sim} \cP$:
	\begin{enumerate}[label=(\alph*)]
		\item\label{it:assum:stoch_model_stat1} {\bf (Convexity)} The model $\ell_x(\cdot,S)$ is convex and the sum $\ell_x(\cdot,S)+r$ is $\alpha_1$-strongly convex.
		\item\label{it:assum:stoch_model_stat2} {\bf (Bias/variance)} The model $\ell_x(\cdot,S)$ is differentiable at $x$ and satisfies 
		$$\ee_{S} [\nabla \ell_x(x,S)]=\nabla f(x)\qquad \textrm{and }\qquad \ee_S\|\nabla \ell_x(x,S)-\nabla f(x)\|^2\leq \sigma^2_0.$$
		\item\label{it:assum:stoch_model_stat3} {\bf (Accuracy)}  The estimate holds: 
		$$\ee_{S}[\ell_x(x,S)-\ell_x(y,S)]\geq f(x)-f(y)+\frac{\alpha_2}{2}\|x-y\|^2.$$
	\end{enumerate} 	
\end{assumption}

The convexity assumption~\ref{it:assum:stoch_model_stat1} is self-explanatory. The Bias/variance property \ref{it:assum:stoch_model_stat2} asserts that $\nabla \ell_x(x,S)$ is an unbiased estimator of $\nabla f(x)$ and has finite variance $\sigma_0^2$.
The accuracy assumption \ref{it:assum:stoch_model_stat3} simply states that the gap $\ell_x(x,S)-\ell_x(y,S)$ is lower bounded in expectation by the true gap $f(x)-f(y)$. Note that  assumption \ref{it:assum:stoch_model_stat3}  is trivially implied by the two intuitive conditions:
$$\ee_{S}[\ell_x(x,S)]=f(x)\qquad \textrm{and}\qquad \ee_{S}[\ell_x(y,S)]+\frac{\alpha_2}{2}\|x-y\|^2\leq f(y),$$
holding for all $x,y\in\R^d$.
The first simply says that the model $\ell_x(\cdot,S)$ evaluated at the basepoint $x$ coincides with $f(x)$ in expectation, while the second asserts that the model $\ell_x(\cdot,S)$ lower bounds $f$ in expectation.

An important consequence of Assumption~\ref{assum:stoch_model_stat} is that $F$ is strongly convex with parameter $\alpha_1+\alpha_2$.

\begin{lemma}\label{lem:quad_growth}
	The function $F$ is $(\alpha_1+\alpha_2)$-strongly convex.
\end{lemma}
\begin{proof}
We will show that $F$ satisfies a strong form of a subgradient inequality. To this end, fix a sample set $S\stackrel{i.i.d}{\sim} \cP$, and a point $x\in\R^d$ along with a subgradient $v\in \partial F(x)$. Notice that the inclusion $v-\nabla f(x)\in\partial r(x)$ holds. Therefore, the vector	$v-\nabla f(x)+\nabla \ell_x(x,S)$ is a subgradient of $\ell_x(\cdot,S)+r$ at $x$. 
We  compute
	\begin{align}
	F(y)-F(x)&\geq \ee_{S}[\ell_{x}(y,S)+r(y)-\ell_{x}( x,S)-r( x)]+\frac{\alpha_2}{2}\|y-\bar x\|^2\label{eqn:compute1}\\
	&\geq \ee_{S}[\langle v-\nabla f(x)+\nabla \ell_x(x,S),y- x\rangle]+\frac{\alpha_1+\alpha_2}{2}\|y- x\|^2\label{eqn:compute2}\\
	&=\langle v,y-x\rangle+\frac{\alpha_1+\alpha_2}{2}\|y- x\|^2,\label{eqn:compute3}
	\end{align}
	where  \eqref{eqn:compute1}, \eqref{eqn:compute2}, and \eqref{eqn:compute3} follow from conditions \ref{it:assum:stoch_model_stat3}, \ref{it:assum:stoch_model_stat1}, and \ref{it:assum:stoch_model_stat2} of Assumption~\ref{assum:stoch_model_stat}, respectively.  It follows immediately that $\partial F$ is $(\alpha_1+\alpha_2)$-strongly monotone. Therefore \cite[Theorem 12.17]{RW98} directly implies that $F$ is $(\alpha_1+\alpha_2)$-strongly convex, as claimed.
\end{proof}

\subsubsection{Examples: proximal point, proximal gradient, and clipped gradient}
Let us look at three examples of models satisfying Assumption~\ref{assum:stoch_model_stat} and the corresponding algorithms; see the accompanying Figure~\ref{fig:illustr_lower_model} and Table~\ref{table:methods}. Specifically, for the three examples, we assume the following two conditions on the loss and the regularizer:
\begin{enumerate}
	\item {\bf (Variance)} There exists $\sigma>0$ satisfying 
	$$\ee_{z\sim\cP}\|\nabla \ell(x,z)-\nabla f(x)\|^2\leq \sigma^2\qquad \textrm{for all } x\in\R^d.$$
	\item {\bf (Strong convexity)} There exist  $\alpha,\mu\geq 0$ such that $r$ is $\mu$-strongly convex and for every $z\in Z$, the loss $\ell(\cdot,z)$ is $\alpha$-strongly convex. 
\end{enumerate}

Algorithm~\ref{alg:MBA_stat} clearly treats $\ell$ and $r$ differently, and therefore $\alpha$ and $\mu$ play distinct roles. Notice that one can always modify the constants $\alpha$ and $\mu$ while maintaining the sum $\alpha+\mu$, simply by adding/subtracting a fixed quadratic $\frac{\lambda}{2}\|\cdot\|^2$ from $\ell$ and $r$.

\begin{table}
	\begin{center}
		\begin{tabular}{| l | c |  c |}
			\hline
			Method & Model $\ell_{x}(y,z)$ & $(\alpha_1,\alpha_2,\sigma_0)$ \\ 
			\hline\hline
			Prox-point & $\ell(y,z)$ & $(\alpha+\mu,0,\sigma)$ \\  
			Gradient & $\ell(x,z)+\langle \nabla \ell(x,z),y-x\rangle$ & $(\mu,\alpha,\sigma)$   \\
			Clipped-grad & $\max\{\ell(x,z)+\langle \nabla \ell(x,z),y-x\rangle,0\}$ & $(\mu,0,\sigma)$  \\
			\hline
		\end{tabular}
	\end{center}
	\caption{Stochastic proximal point, gradient, and clipped gradient methods.}\label{table:methods}
\end{table}

\begin{example}[Stochastic proximal point]
	{\rm
		The simplest stochastic model of a loss is the loss itself:
		$$\ell_x(y,z)=\ell(y,z).$$
		Algorithm~\ref{alg:MBA_stat} equipped with these models is the proximal point method.
		Assumption~\ref{assum:stoch_model_stat} trivially holds with parameters $(\alpha_1,\alpha_2,\sigma_0)=(\alpha+\mu,0,\sigma)$. More generally, given a sample set $S\stackrel{i.i.d}{\sim} \cP$, we may use the average model $\ell_x(y,S)=\frac{1}{|S|}\sum_{z\in S}\ell(y,z)$ whose gradient has the smaller  variance  $\sigma^2_0=\sigma^2/|S|$. 
	}
\end{example}

\begin{example}[Stochastic proximal gradient]
	{\rm
		Another class of models is induced by linearizations
		$$\ell_x(y,z)=\ell(x,z)+\langle \nabla\ell(x,z),y-x\rangle.$$	
		Algorithm~\ref{alg:MBA} equipped with these models reduces to the stochastic proximal gradient method.
		Assumption~\ref{assum:stoch_model_stat}\ref{it:assum:stoch_model_stat1} clearly holds with $\alpha_1=\mu$. Assumption~\ref{assum:stoch_model_stat}\ref{it:assum:stoch_model_stat2} holds trivially. Assumption~\ref{assum:stoch_model_stat}\ref{it:assum:stoch_model_stat3} with $\alpha_2=\alpha$ follows from the expression
		$\ee_{z\sim\cD(x)}\ell_x(x,z)= f_x(x)$ and strong convexity of $f_x$. More generally, given a sample set $S\stackrel{i.i.d}{\sim} \cP$, we may declare $\ell_x(y,S)=\frac{1}{|S|}\sum_{z\in S}\ell_x(y,z)$ thereby decreasing the variance  $\sigma^2_0=\sigma^2/|S|$. 
	}
\end{example}

\begin{example}[Stochastic clipped proximal gradient]
	{\rm
		An interesting middle ground between proximal point and gradient models was proposed in \cite{asi2019stochastic}. Namely, suppose that each loss $\ell(\cdot,z)$ is lower bounded by some known constant, which we may without loss of generality assume is zero. This assumption is  completely innocuous in data scientific contexts.
		Then we may use the clipped linear models
		$$\ell_x(y,z)=\max\{\ell(x,z)+\langle \nabla\ell(x,z),y-x\rangle,0\}.$$
		Algorithm~\ref{alg:MBA_stat} equipped with these models is the proximal clipped stochastic gradient method.	A quick computation shows that Assumption~\ref{assum:stoch_model_stat} holds with $(\alpha_1,\alpha_2,\sigma_0)=(\mu,0,\sigma)$. Given a sample set $S\stackrel{i.i.d}{\sim} \cP$, we can form the two models,
		$$\ell_x(x,S)=\frac{1}{|S|}\sum_{z\in S} \ell(x,z) \quad\textrm{and}\quad \ell_x(y,z)=\max\biggl\{\frac{1}{|S|}\sum_{z}\ell(x,z)+\langle \nabla\ell(x,z),y-x\rangle,0\biggr\}.$$
		It is straightforward to verify that both models $\ell_x(\cdot,S)$ are differentiable at $x$ with gradient $\nabla \ell_x(x,S)=\frac{1}{|S|}\sum_{z\in S}\nabla \ell(x,z)$. Thus the variance of the gradient improves to $\sigma^2_0=\sigma^2/|S|$.	
	}
\end{example}

\subsubsection{Convergence guarantees.}
We are now ready to analyze Algorithm~\ref{alg:MBA_stat}. The main idea of the convergence proof is to establish a one-step improvement guarantee on the function $F$.
Henceforth, we let $\EE_t[\cdot]$ denote the expectation conditioned on $x_t$.

\begin{lemma}[One step improvement]\label{lem:one-step_improv_stat}
	Suppose that Assumption~\ref{assum:stoch_model_stat} holds and that $\nabla f$ is $L$-Lipschitz continuous.
	Fix a sequence $\eta_t<\frac{1}{L}$ for all $t\geq 0$. Then the iterates generated by Algorithm~\ref{alg:MBA_stat} satisfy:
	\begin{equation}\label{eqn:one_step_inside}
	\eta_t\EE_t[F(x_{t+1})-F(y)] \leq \frac{1- \alpha_2\eta_t}{2}\|x_{t}- y\|^2-\frac{1+\alpha_1\eta_t}{2}\EE_t\|x_{t+1}- y\|^2+\tfrac{\sigma^2_0\eta_t^2}{2(1-\eta_t L)},
	\end{equation}
	for all indices $t\geq 0$ and all $y\in \R^d$.
\end{lemma}
\begin{proof} Observe that the inequality we wish to prove is conditioned on $x_t$. Therefore to simplify notation, set $x:=x_t$, $S:=S_t$, and $x^+:=x_{t+1}$.	
	Since $x^+$ is the minimizer of a $(\eta_t^{-1}+\alpha_1)$-strongly convex function $\ell_{x}(\cdot,S)+r+\frac{1}{2\eta_t}\|\cdot-x\|^2$, we deduce 
	\begin{align}
	\tfrac{\eta_t^{-1}+\alpha_1}{2}\|x^+-y\|^2\leq \left(\ell_{x}(y,S)+r(y)+\tfrac{1}{2\eta_t}\|y-x\|^2\right)-\left(\ell_{x}(x^+,S)+r(x^+)+\tfrac{1}{2\eta_t}\|x^+-x\|^2\right).\label{eqn:three_point_last}
	\end{align}	
	We next lower bound $\ell_{x}(x^+,S)$ using convexity: Assumption~\ref{assum:stoch_model_stat} guarantees	
\begin{align*}
	\ell_{x}(x^+,S)&\geq \ell_{x}(x,S)+\langle \nabla \ell_x(x,S),x^+-x\rangle.
\end{align*}
	Combining this estimate with \eqref{eqn:three_point_last}, multiplying through by $2\eta_t$, and taking the expectation yields  
	\begin{align}
	(1+\alpha_1\eta_t)\EE_S\|x^+-y\|^2
	&\leq \|x-y\|^2-2\eta_t\EE_S[\ell_{x}(x,S)- \ell_{x}(y,S)]- \EE_S\|x^+-x\|^2 \notag\\
	& \qquad -2\eta_t\EE_S[\langle \nabla \ell_x(x,S),x^+-x\rangle]-2\eta_t\EE_S [r(x^+)-r(y)]\label{eqn:main_ineq_prox-mod}.
	\end{align}
	Assumption~\ref{assum:stoch_model_stat}\ref{it:assum:stoch_model_stat3} implies 
	\begin{align*}
	\EE_S[\ell_{x}(x,S)- \ell_{x}(y,S)]&\geq f(x)-f(y)+\frac{\alpha_2}{2}\|x-y\|^2.
	\end{align*}	
	Combining this estimate with \eqref{eqn:main_ineq_prox-mod} yields
	\begin{align}
	(1+\alpha_1\eta_t)\EE_S\|x^+-y\|^2 &\leq  (1-\alpha_2\eta_t)\|x-y\|^2 - 2\eta_t\EE_S[f(x)+r(x^+)-f(y)-r(y)]\notag\\
	&~~~- \EE_S\|x^+-x\|^2
	-2\eta_t\EE_S[\langle \nabla \ell_x(x,S),x^+-x\rangle].\label{eqn:continue}
	\end{align}
	Next,  smoothness of $f$ guarantees
	\begin{equation}\label{eqn:cond_exp_smooth}
	f(x)\geq f(x^+)-\langle \nabla f(x),x^+-x\rangle-\frac{L}{2}\|x^+-x\|^2.
	\end{equation}
	Taking expectations in \eqref{eqn:cond_exp_smooth} and combining the estimate with \eqref{eqn:continue} yields 
	\begin{align*}
	(1+\alpha_1\eta_t)\EE_S\|x^+-y\|^2
	&\leq (1-\alpha_2\eta_t)\|x-y\|^2 - 2\eta_t\EE_S[F(x^+)-F(y)]\\
	&  ~~ +2\eta_t\EE_S[\langle \nabla f(x)-\nabla \ell_x (x,S),x^+-x\rangle]- (1-\eta_t L)\EE_S\|x^+-x\|^2 .
	\end{align*}
	Young's inequality in turn guarantees
	$$2\eta_t\langle \nabla f(x)-\nabla\ell_x(x,S),x^+-x\rangle \leq \frac{\eta_t^2\|\nabla f(x)-\nabla\ell_x(x,S)\|^2}{1-\eta_t L}+(1-\eta_t L)\|x^+-x\|^2.$$
	Combining the last two inequalities with the finite variance assumption \ref{assum:stoch_model_stat}\ref{it:assum:stoch_model_stat2} yields
	\begin{align*}
	(1+\alpha_1\eta_t)\EE_S\|x^+-y\|^2
	&\leq  \left(1-\alpha_2\eta_t\right)\|x-y\|^2 - 2\eta_t\EE_S[F(x^+)-F(y)]+\frac{\sigma^2_0\eta_t^2}{1-\eta_t L}.
	\end{align*}
	Rearranging yields~\eqref{eqn:one_step_inside} as claimed. 
\end{proof}

Using a constant parameter $\eta>0$ yields the following guarantee.
\begin{theorem}[Constant step]
	Suppose that Assumption~\ref{assum:stoch_model_stat} holds and that $\nabla f$ is $L$-Lipschitz continuous.
	Fix a sequence $\eta<\min\{\frac{1}{2L},\frac{1}{\alpha_1},\frac{1}{\alpha_2}\}$. Then for all points $y\in \R^d$, the estimate holds
$$\EE[F(\hat x_t)-F(y)]\leq 2\left(1-\frac{(\alpha_1+\alpha_2)\eta}{2}\right)\left( F(x_0)-F(y)\right)+\sigma_0^2\eta,$$
where we recursively define $\hat{x}_t=\left(1-\frac{\eta(\alpha_1+\alpha_2)}{1+c_2\eta}\right)\hat x_{t-1}+\frac{\eta(\alpha_1+\alpha_2)}{1+c_2\eta}x_t$.
\end{theorem}
\begin{proof}
The assumption $\eta<\frac{1}{2L}$ guarantees $\tfrac{\sigma^2_0\eta^2}{2(1-\eta L)}\leq \sigma_0^2\eta$, and therefore Lemma~\ref{lem:one-step_improv_stat} implies 
	$$\eta \EE_t[F(x_{t+1})-F(y)] \leq \frac{1- \alpha_2\eta}{2}\|x_{t}- y\|^2-\frac{1+\eta\alpha_1}{2}\EE_t\|x_{t+1}- y\|^2+\sigma^2_0\eta.$$
Lemma \ref{cor:const_param} directly implies $$\EE[F(\hat x_t)-F(y)]\leq \left(1-\frac{(\alpha_1+\alpha_2)\eta}{1+\alpha_1\eta}\right)\left( F(x_0)-F(y)+ \frac{\alpha_1+\alpha_2}{2}\|x_0-y\|^2\right)+\sigma_0^2\eta.$$
Taking into account $\alpha_1\eta\leq 1$ and Lemma~\ref{lem:quad_growth} completes the proof.
\end{proof}

Combining Algorithm~\ref{alg:MBA_stat}  with the geometric decay schedule (Algorithm~\ref{alg:geo_decay}) yields an algorithm with the following efficiency guarantee, which is immediate from Lemma~\ref{lem:geo_restart}.
\begin{corollary}[Efficiency]
Suppose that Assumption~\ref{assum:stoch_model_stat} holds and that $\nabla f$ is $L$-Lipschitz continuous. Fix an estimate $\Delta\geq F(x_0)-\min F$. Then Algorithm~\ref{alg:MBA_stat} may be augmented with the geometric decay schedule (Algorithm~\ref{alg:geo_decay}) under the identification
$$c=\frac{\alpha_1+\alpha_2}{2}, \quad C=2,\quad h(x)=F(x)-\min F,\quad \delta_0=\min\left\{\frac{1}{2L},\frac{1}{\alpha_1},\frac{1}{\alpha_2}\right\},\quad D=\sigma_0^2.$$ The resulting procedure will generate a point $x$ satisfying $\EE[F(x)-\min F]\leq \varepsilon$ after 
$$\mathcal{O}\left(\frac{L}{\alpha_1+\alpha_2}\log\left(\frac{\Delta}{\varepsilon}\right)+\frac{\sigma_0^2}{(\alpha_1+\alpha_2)\varepsilon}\right)\qquad \textrm{iterations}.$$
\end{corollary}

\subsection{Model-based algorithms under state-dependent distributions}

Model-based algorithms easily adapt to the setting with state-dependent distributions by allowing the sampling distribution to vary along the iterations. 
The formal procedure is stated in Algorithm~\ref{alg:MBA}.

\begin{algorithm}[H]
	{\bf Input:} initial $x_0$ and sequence $\{\eta_t\}_{t=0}^T\subset (0,\infty)$.
	
	{\bf Step} $t=0,1,\ldots, T$: 
	\begin{equation*}
	\begin{aligned}
	&\textrm{Sample~} S_t\stackrel{i.i.d}{\sim} \cD(x_t)\\
	& \textrm{Set~} x_{t+1}=\argmin_y\,\biggl\{ \ell_{x_t}(y,S_t)+r(y)+\frac{1}{2\eta_t}\|y-x_t\|^2 \biggr\}
	\end{aligned}
	\end{equation*}
	
	\caption{Model-based algorithm under state-dependent sampling}\label{alg:MBA}
\end{algorithm}

Naturally, we will impose Assumption~\ref{assum:stoch_model_stat} for all static problems $\eqref{eqn:PS1}$ with $\cP=\cD(x)$.

\begin{assumption}[Stochastic models under state feedback]\label{assum:stoch_model}
Suppose that Assumption~\ref{assum:stoch_model_stat} holds for the static problem $\mathtt{St}(\cD(x))$ for all $x\in \R^d$.
\end{assumption}

	For the rest of the section, we suppose that Assumptions~\ref{assum:perm_pred}, \ref{ass:smoothness}\ref{assump:smooth_center_2}, and \ref{assum:stoch_model} hold.
We are now ready to analyze Algorithm~\ref{alg:MBA}. 
The main idea of the convergence proof is to combine the one-step improvement guarantee   \eqref{eqn:one_step_inside} for the intermediate function $\varphi_{x_t}$ with the function gap  inequality~\eqref{eqn:gap_comp} to obtain a one-step improvement on $\varphi$.

\begin{corollary}[Key recursion]\label{cor:recur_prox}
	Suppose the inequality $\eta_t<\frac{1}{L}$ holds for all $t\geq 0$. Then the iterates generated by Algorithm~\ref{alg:MBA} satisfy the estimate:
	\begin{equation}\label{eqn:one_step_prox_grad_last}
\eta_t \EE_t[\varphi(x_{t+1})-\varphi(\bar x)]\leq \tfrac{1-(\alpha_2-\gamma\beta)\eta_t}{2}\|x_{t}- \bar x\|^2- \tfrac{1+(\alpha_1-\gamma\beta)\eta_t}{2}\|x_{t+1}- \bar x\|^2+\tfrac{\sigma^2_0\eta_t^2}{2(1-\eta_t L)}.
\end{equation}
\end{corollary}
\begin{proof}
The function gap  inequality~\eqref{eqn:gap_comp} directly yields
	\begin{align}
\varphi_{x_t}(x_{t+1})-\varphi_{x_t}(\bar x)&\geq \varphi_{\bar x}(x_{t+1})-\varphi_{\bar x}(\bar x)-\gamma\beta\|x_{t+1}-\bar x\|\cdot\|x_t-\bar x\|\notag\\
&\geq \varphi_{\bar x}(x_{t+1})-\varphi_{\bar x}(\bar x)-\frac{\gamma\beta}{2}\|x_{t+1}-\bar x\|^2-\frac{\gamma\beta}{2}\|x_t-\bar x\|^2, \notag
\end{align}
where the last inequality follows from Young's inequality. Using this estimate to lower bound the left side of \eqref{eqn:one_step_inside} and rearranging completes the proof.
\end{proof}

The following theorem summarizes the convergence guarantees of the method with a constant parameter $\eta>0$.

\begin{theorem}[Convergence guarantees]\label{thm:conv_guarant}
Let $x_t$ be the iterates generated by Algorithm~\ref{alg:MBA} with a fixed parameter $\eta>0$. Then the following are true.
\begin{enumerate}
\item Suppose, we are in the regime $\frac{\gamma\beta}{\alpha_1+\alpha_2}<1$ and set $\hat \alpha:=\alpha_1+\alpha_2-\gamma\beta$. Then with the parameter 
 $\eta\leq \min\{\frac{1}{2L},\frac{1}{\alpha_1},\frac{1}{\alpha_2}\}$,  the estimate holds:
\begin{equation}\label{eqn:distance_decay_prox}
\EE\|x_{t}- \bar x\|^2\leq \left(1-\frac{\hat \alpha\eta}{2}\right)^t\|x_{0}- \bar x\|^2+\frac{\sigma^2_0\eta}{2\hat \alpha}.
\end{equation}
\item Suppose we are in the regime $\frac{\gamma\beta}{\alpha_1+\alpha_2}<\frac{1}{2}$ and set $\hat \alpha:=\alpha_1+\alpha_2-2\gamma\beta$. Then with the parameter
 $\eta\leq \min\{\frac{1}{2L},\frac{1}{(\alpha_1-\gamma\beta)^+},\frac{1}{(\alpha_2-\gamma\beta)^+}\}$,  the estimate holds:
 $$\EE [ \varphi\left(\hat x_t\right)-\varphi(\bar x)]\leq  2\left(1-\frac{\hat \alpha\eta}{2}\right)^t\left( \varphi(x_0)-\varphi(\bar x)\right)+\sigma^2_0\eta,$$
 where we recursively define $\hat x_t=\left(1-\frac{\hat \alpha\eta }{1+(\alpha_1-\gamma\beta)\eta)}\right)\hat x_{t-1}+\frac{\hat \alpha\eta }{1+(\alpha_1-\gamma\beta)\eta)}x_t$.
\end{enumerate}
\end{theorem}
\begin{proof}
Notice that the assumption $\eta\leq \frac{1}{2L}$ implies $\tfrac{\sigma^2_0\eta^2}{2(1-\eta L)}\leq \sigma_0^2\eta^2$. Combining this estimate with Corollary~\ref{cor:recur_prox} yields:
	\begin{equation}\label{eqn:one_step_prox_grad_last_inproof}
\eta \EE[\varphi(x_{t+1})-\varphi(\bar x)]\leq \tfrac{1-(\alpha_2-\gamma\beta)\eta}{2}\EE\|x_{t}- \bar x\|^2- \tfrac{1+(\alpha_1-\gamma\beta)\eta}{2}\EE\|x_{t+1}- \bar x\|^2+\sigma^2_0\eta^2.
\end{equation}
{\em Proof of Claim 1:} Lower-bounding the left-side using Lemma~\ref{lem:quad_growth} and rearranging yields 
	\begin{equation}
\frac{1+(2\alpha_1+\alpha_2-\gamma\beta)\eta}{2}\EE\|x_{t+1}-\bar x\|^2\leq \frac{1-(\alpha_2-\gamma\beta)\eta}{2}\EE\|x_{t}- \bar x\|^2+\sigma^2_0\eta^2.
\end{equation}
The assumption $\alpha_1+\alpha_2-\gamma\beta>0$ ensures that  the left side is positive, and therefore
$$\EE\|x_{t+1}-\bar x\|^2\leq  \left(1-\frac{2(\alpha_1+\alpha_2-\gamma\beta)\eta}{1+(2\alpha_1+\alpha_2-\gamma\beta)\eta}\right)\EE\|x_{t}- \bar x\|^2+\frac{\sigma^2_0\eta^2}{1+(2\alpha_1+\alpha_2-\gamma\beta)\eta}.$$
The assumption $\eta\leq \frac{1}{\alpha_2}$ ensures that the coefficient multiplying $\EE\|x_{t}- \bar x\|^2$ lies in $(0,1)$. Iterating the recursion and taking into account $(\hat \alpha+\alpha_1)\eta\leq 3$ directly yields \eqref{eqn:distance_decay_prox}.

{\em Proof of Claim 2:}
In light of the expression \eqref{eqn:one_step_prox_grad_last_inproof}, we aim to apply Corollary~\ref{cor:const_param} with $h(x):=\varphi(x)-\varphi(\bar x)$, $D_t:=\frac{1}{2}\EE\|x_t-\bar x\|^2$, $c_1:=\alpha_2-\gamma\beta$, and $c_2:=\alpha_1-\gamma\beta$. The assumption $\eta\leq \min\{\frac{1}{(\alpha_1-\gamma\beta)^+},\frac{1}{(\alpha_2-\gamma\beta)^+}\}$ directly implies $1-c_1\eta>0$ and $1+c_2\eta>0$.
An application of Corollary~\ref{cor:const_param} yields the estimate:
 $$\EE [ \varphi\left(\hat x_t\right)-\varphi(\bar x)]\leq  \left(1-\frac{\hat \alpha\eta}{1+(\alpha_1-\gamma\beta)\eta}\right)^t\left( \varphi(x_0)-\varphi(\bar x)+ \frac{\hat \alpha}{2}\|x_0-\bar x\|^2\right)+\sigma^2_0\eta.$$
Lemma~\ref{lem:quad_growth} in turn implies $\frac{\hat \alpha}{2}\|x_0-\bar x\|^2\leq \varphi(x_0)+\varphi(\bar x)$. Taking into account $(\alpha_1-\gamma\beta)\eta\leq 1$ completes the proof.
\end{proof}

The following corollary obtains efficiency guarantees by combining Algorithm~\ref{alg:MBA} with the geometric decay schedule (Algorithm~\ref{alg:geo_decay}).

\begin{corollary}[Efficiency of Algorithm~\ref{alg:MBA} with geometrically decaying schedule]\label{cor:main_eff_result_dependent} {\hfill \\ }
Suppose that Assumptions~\ref{assum:perm_pred}, \ref{ass:smoothness}\ref{assump:smooth_center_2}, and \ref{assum:stoch_model} hold. Then the following are true.
\begin{enumerate}
\item {\bf (Distance)} Suppose we are in the regime $\frac{\gamma\beta}{\alpha_1+\alpha_2}<1$ and that we have available an estimate $\Delta\geq \|x_0-\bar x\|^2$. Define the augmented strong convexity parameter $\hat \alpha=\alpha_1+\alpha_2-\gamma\beta$. Then Algorithm~\ref{alg:MBA} may be augmented with the geometric decay schedule (Algorithm~\ref{alg:geo_decay})
under the identification
$$c=\frac{\hat\alpha}{2}, \quad C=1,\quad h(x)=\|x-\bar x\|^2,\quad \delta_0=\min\left\{\tfrac{1}{2L},\tfrac{1}{\alpha_1},\tfrac{1}{\alpha_2}\right\},\quad D=\frac{\sigma_0^2}{2\hat\alpha}.$$ 
The resulting procedure will generate a point $x$ satisfying $\EE[\|x-\bar x\|^2]\leq \varepsilon$ using 
$$\mathcal{O}\left(\left(\frac{L+\alpha_1+\alpha_2}{\hat\alpha}\right)\log\left(\frac{\Delta}{\varepsilon}\right)+\frac{\sigma_0^2}{\hat\alpha^2\varepsilon}\right)\qquad\textrm{samples}.$$
\item {\bf (Function value)} Suppose we are in the regime $\frac{\gamma\beta}{\alpha_1+\alpha_2}<\frac{1}{2}$ and that we have available an estimate $\Delta\geq \varphi(x_0)-\varphi(\bar x)$. Define the augmented strong convexity parameter $\hat \alpha=\alpha_1+\alpha_2-2\gamma\beta$. Then Algorithm~\ref{alg:MBA} may be augmented with the geometric decay schedule (Algorithm~\ref{alg:geo_decay}) under the identification
$$c=\frac{\hat\alpha}{2}, \quad C=2,\quad h(x)=\varphi(x)-\varphi(\bar x),\quad \delta_0=\min\left\{\tfrac{1}{2L},\tfrac{1}{(\alpha_1-\gamma\beta)^+},\tfrac{1}{(\alpha_2-\gamma\beta)^+}\right\},\quad D=\sigma_0^2.$$ 
The resulting procedure will generate a point $x$ satisfying $\varphi(x)-\varphi(\bar x)\leq \varepsilon$ using 
$$\mathcal{O}\left(\frac{L+\alpha_1+\alpha_2}{\hat\alpha}\log\left(\frac{\Delta}{\varepsilon}\right)+\frac{\sigma_0^2}{\hat\alpha\varepsilon}\right)\qquad \textrm{samples}.$$
\end{enumerate}
\end{corollary}
\begin{proof}
This is immediate from Theorem~\ref{thm:conv_guarant} and Lemma~\ref{lem:geo_restart}.
\end{proof}

Efficiency guarantees of the stochastic proximal point, proximal gradient, and clipped gradient methods under Assumptions~\ref{assum:smoothness}, \ref{assum:perm_pred}, \ref{ass:strong_conv_perm}\ref{it:str_conv_loss}, \ref{ass:smoothness}\ref{assump:smooth_center_2}  follow directly from Corollary~\ref{cor:main_eff_result_dependent} and the values of $\alpha_1$ and $\alpha_2$ in Table~\ref{table:methods}.

\section{Inexact repeated minimization}\label{sec:inexactRM}
All stochastic algorithms we have discussed so far in each iteration $t$ apply a single step of a standard algorithm on the static problem $\mathtt{St}(\cD(x_t))$. An alternative strategy is to apply an algorithm to the static problem $\mathtt{St}(\cD(x_t))$ for a moderate number of iterations and switch the distribution at the end of the run. Algorithms based on this principle can be interpreted as inexact repeated minimization. Such algorithms can be superior in the typical settings where changing the distribution governing the static problem may be much costlier than drawing a sample. We will informally call the process of changing the distribution from $\cD(x)$ to $\cD(y)$ as  {\em deployment}. In this section, we will show that it is possible to maintain the sample complexity derived in previous sections, while using a number of deployments that is only logarithmic in the problem parameters.

\subsection{Model-based algorithms}
We begin by analyzing  Algorithm~\ref{alg:MBA_stage}, which effectively performs repeated minimization with each subproblem solved approximately by a model-based algorithm (Algorithm~\ref{alg:MBA}). We describe two versions of the procedure. The first version uses the last iterate of each run to warmstart the next run, while the second version instead uses the running average of the iterates. Throughout the section, we  suppose that Assumptions~\ref{assum:perm_pred}, \ref{ass:smoothness}\ref{assump:smooth_center_2}, and \ref{assum:stoch_model} hold.

\begin{algorithm}[H]
	{\bf Input:} initial $u_0$, sequences $\{\eta_j\}_{j=0}^J\subset (0,\infty)$ and $\{J_t\}_{t=0}^T\subset\mathbb{N}$.

	\For{$t=0,1,\ldots,T$}{
		Set $x_0=u_t$
		
		\For{$j=1,\ldots, J_t$}
		{
			\begin{equation*}
			\begin{aligned}
			&\textrm{Sample~} S_j\stackrel{i.i.d}{\sim} \cD(x_0)\\
			& \textrm{Set~} x_{j}=\argmin_x\, \biggl\{ \ell_{x_0}(x,S_j)+r(x)+\frac{1}{2\eta_{j}}\|x-x_{j-1}\|^2 \biggr\}
			\end{aligned}
			\end{equation*}
		}
		
		Set 
		\begin{displaymath}
		u_{t+1} = \left\{
		\begin{array}{ll}
		x_{J_t} &, \textrm{ if Version I}\\
		\displaystyle\hat\Gamma_{J_t} x_0+\sum_{j=1}^{J_t} \frac{\hat\Gamma_{J_t}\hat \delta_j}{\hat\Gamma_j}x_j & , \textrm{ if Version II}
		\end{array}
		\right.
		\end{displaymath} 
		where we define $\displaystyle \hat \delta_j=\tfrac{(\alpha_1+\alpha_2-\gamma\beta)\eta_j}{1+(\alpha_1-\gamma\beta)\eta_j}$ and $\displaystyle\hat\Gamma_j=\prod_{i=1}^j (1-\hat \delta_i)$.
	}
	
	\caption{Stagewise based Minimization}\label{alg:MBA_stage}
\end{algorithm}

The following theorem shows a one-step improvement guarantee for Algorithm~\ref{alg:MBA_stage}, which explicitly balances the impact of the number of inner iterations $J_t$ and the step-size parameters $\eta_j$ on the overall performance.

\begin{theorem}[One epoch improvement]
	Suppose we are in the regime $\frac{\gamma\beta}{\alpha_1+\alpha_2}<1$. 
	\begin{enumerate}
		\item {\bf (Version I)} Choose a sequence $\eta_j\leq \frac{1}{2L}$. Then the iterates $u_t$ generated by Version I of Algorithm~\ref{alg:MBA_stage} satisfy
		\begin{equation}\label{eqn:stage_wise_dist}
		\EE\|u_{t+1}-\bar x\|^2\leq \left(\Gamma_{J_t}+\tfrac{\gamma\beta}{2\alpha_1+2\alpha_2-\gamma\beta}\right) \EE\|u_t-\bar x\|^2+\tfrac{2\sigma^2_0}{2\alpha_1+2\alpha_2-\gamma\beta}\sum_{j=1}^{J_t}\frac{\eta_j\delta_j\Gamma_{J_t}}{\Gamma_j},
		\end{equation}
		where we set $\delta_j=\frac{(2\alpha_1+2\alpha_2-\gamma\beta)\eta_j}{1+\eta_j(2\alpha_1+\alpha_2-\gamma\beta)}$ and $\Gamma_j=\prod_{i=1}^{j} (1-\delta_i)$.
		\item {\bf (Version II)} Choose a sequence $\eta_j\leq \min\{\frac{1}{2L}, \frac{1}{(\gamma\beta-\alpha_1)^+},\frac{1}{\alpha_2}\}$. Then the iterates $u_t$ generated by Version II of Algorithm~\ref{alg:MBA_stage} satisfy
		\begin{equation}\label{eqn:stage_wise_func}
		\EE[\varphi(u_{t+1})-\varphi(\bar x)]\leq \left(2\hat\Gamma_{J_t}+\tfrac{\gamma\beta}{\alpha_1+\alpha_2}\right)\EE[\varphi(u_t)-\varphi(\bar x)] +\sigma^2_0\sum_{j=1}^{J_t}\frac{\eta_j\hat \delta_j\hat \Gamma_{J_t}}{\hat \Gamma_j}.
		\end{equation}
	\end{enumerate}
\end{theorem}

\begin{proof}
	Fix an index $t\in\{0,\dots, T\}$. We will consider the outcome of the inner loop $j=1,\ldots, J_t$. To simplify notation therefore let us drop the subscript $t$ from $u_t$. 
	Setting $y=\bar x$ in Lemma~\ref{lem:one-step_improv_stat} yields the guarantee:
	\begin{equation}\label{eqn:inproof_onestep}
	\eta_{j}\EE[\varphi_{u}(x_{j})-\varphi_{u}(\bar x)]\leq \frac{1-\alpha_2\eta_j}{2}\EE\|x_{j-1}-\bar x\|^2-\frac{1-\alpha_1\eta_{j}}{2}\EE\|x_{j}-\bar x\|^2+\sigma^2_0\eta^2_{j},
	\end{equation}
	for all $j=1,\ldots,J_t$. The function gap inequality \eqref{eqn:gap_comp} in turn implies
	\begin{align*}
	\varphi_{u}(x_{j})-\varphi_u(\bar x)&\geq \varphi(x_{j})-\varphi(\bar x)-\gamma\beta\|x_{j}-\bar x\|\cdot\|y-\bar x\|\\
	&\geq \varphi(x_{j})-\varphi(\bar x)-\frac{\gamma\beta}{2}\|x_{j}-\bar x\|^2-\frac{\gamma\beta}{2}\|u-\bar x\|^2,
	\end{align*}
	where the last estimate follows from  Young's inequality.
	Combining with \eqref{eqn:inproof_onestep} and rearranging, we conclude
	\begin{equation}\label{eqn:key_perturb_est}
	\eta_j\EE[\varphi(x_{j})-\varphi(\bar x)]\leq \tfrac{1-\alpha_2\eta_j}{2}\EE\|x_{j-1}-\bar x\|^2-\tfrac{1+(\alpha_1-\gamma\beta)\eta_j}{2}\EE\|x_{j}-\bar x\|^2+\tfrac{2\sigma^2_0\eta^2_j+\gamma\beta\eta_j\EE\|u-\bar x\|^2}{2}.
	\end{equation}
	
	Suppose now that we  are running version I of Algorithm~\ref{alg:MBA_stage}. Lower-bounding the left-side of  \eqref{eqn:key_perturb_est} using strong convexity, $\varphi(x_{j})-\varphi(\bar x)\geq \tfrac{\alpha_1+\alpha_2}{2}\|x_j-\bar x\|^2$ (Lemma~\ref{lem:quad_growth}), and rearranging yields
	$$\EE\|x_{j}-\bar x\|^2\leq \left(\frac{1-\alpha_2\eta_j}{1+\eta_j(2\alpha_1+\alpha_2-\gamma\beta)}\right)\EE\|x_{j-1}-\bar x\|^2+\frac{2\sigma^2_0\eta^2_j+\eta_j\gamma\beta\EE\|u-\bar x\|^2}{1+\eta_j(2\alpha_1+\alpha_2-\gamma\beta)}.$$
	The coefficient of $\EE\|x_{j-1}-\bar x\|^2$ is precisely $1-\delta_j$. Unrolling the recursion, we conclude
	\begin{equation}\label{eqn:first_recourse}
	\EE\|x_{J_t}-\bar x\|^2\leq \Gamma_{J_t} \EE\|x_0-\bar x\|^2+\sum_{j=1}^{J_t} \frac{2\sigma^2_0\eta^2_j+\eta_j\gamma\beta\EE\|u-\bar x\|^2}{1+\eta_j(2\alpha_1+\alpha_2-\gamma\beta)}\cdot \frac{\Gamma_{J_t}}{\Gamma_j}.
	\end{equation}
	Next to simplify the sum, algebraic manipulations yield
	\begin{equation}\label{eqn:first_recourse2}
	\sum_{j=1}^{J_t} \frac{\eta_j}{1+\eta_j(2\alpha_1+\alpha_2-\gamma\beta)}\cdot\frac{\Gamma_{J_t}}{\Gamma_j}=\frac{1}{2\alpha_1+2\alpha_2-\gamma\beta}\sum_{j=1}^{J_t}\frac{\delta_j\Gamma_{J_t}}{\Gamma_j}=\frac{1-\Gamma_{J_t}}{2\alpha_1+2\alpha_2-\gamma\beta},
	\end{equation}
	where the last equality uses Lemma~\ref{lem:gamma_sum}. Note also the expression
	\begin{equation}\label{eqn:first_recourse3}
	\sum_{j=1}^{J_t} \frac{\eta_j^2}{1+\eta_j(2\alpha_1+\alpha_2-\gamma\beta)}\cdot\frac{\Gamma_{J_t}}{\Gamma_j}=\frac{1}{2\alpha_1+2\alpha_2-\gamma\beta}\sum_{j=1}^{J_t}\frac{\eta_j\delta_j\Gamma_{J_t}}{\Gamma_j}.
	\end{equation}
	Combining \eqref{eqn:first_recourse2}, \eqref{eqn:first_recourse3}, \eqref{eqn:first_recourse}, and the equality $x_0=u$ completes the proof of \eqref{eqn:stage_wise_dist}.

	Suppose now that we  are running version II of Algorithm~\ref{alg:MBA_stage}. Then applying Lemma~\ref{lem:average2} to the recursion \eqref{eqn:key_perturb_est} directly implies 
	$$\EE[\varphi(u_{t+1})-\varphi(\bar x)]\leq \hat \Gamma_{J_t}\left(\EE[\varphi(x_0)-\varphi(\bar x)]+\tfrac{\alpha_1+\alpha_2-\gamma\beta}{2}\|x_0-\bar x\|^2+\sum_{j=1}^{J_t}\tfrac{(\alpha_1+\alpha_2-\gamma\beta)(2\sigma^2_0\eta_j^2+\gamma\beta\eta_j\|u-\bar x\|^2)}{2\hat\Gamma_j(1+(\alpha_1-\gamma\beta)\eta_j)}]\right).$$
	To simplify the right side, observe first $\tfrac{\alpha_1+\alpha_2-\gamma\beta}{2}\|x_0-\bar x\|^2\leq \varphi(x_{0})-\varphi(\bar x)$ (Lemma~\ref{lem:quad_growth}). Next the same argument that justified \eqref{eqn:first_recourse2} and \eqref{eqn:first_recourse3} implies the expressions:
	$$\sum_{j=1}^{J_t} \frac{(\alpha_1-\alpha_2-\gamma\beta)\eta_j}{1+(\alpha_1-\gamma\beta)\eta_j}\cdot\frac{\hat\Gamma_{J_t}}{\hat\Gamma_j}=1-\hat\Gamma_{J_t}\qquad \textrm{and}\qquad \sum_{j=1}^{J_t} \frac{(\alpha_1+\alpha_2-\gamma\beta)\eta_j^2}{1+(\alpha_1-\gamma\beta)\eta_j}\cdot\frac{\hat \Gamma_{J_t}}{\hat \Gamma_j}=\sum_{j=1}^{J_t}\frac{\eta_j\hat \delta_j\hat \Gamma_{J_t}}{\hat \Gamma_j}.$$
	The claimed estimate \eqref{eqn:stage_wise_func} follows immediately.
\end{proof}

Looking at the estimate \eqref{eqn:stage_wise_func}, it is natural to choose $J_t$ such that  the estimate,
$2\hat\Gamma_{J_t}+\tfrac{\gamma\beta}{\alpha_1+\alpha_2}\leq \frac{1}{2}(1+\tfrac{\gamma\beta}{\alpha_1+\alpha_2})$, holds. The following result summarizes the guarantees of such a  procedure when the stepsizes $\eta_j$ are constant.

\begin{corollary}
	Suppose we are in the regime $\frac{\gamma\beta}{\alpha_1+\alpha_2}<1$.
	\begin{enumerate}
		\item {\bf (Version I)} Choose a constant parameter $\eta\leq \frac{1}{2L}$ and set  
		$$J_t=\left\lceil\left(1+\frac{1}{(2\alpha_1+2\alpha_2-\gamma\beta)\eta}\right)\log\left(\frac{2\alpha_1+2\alpha_2-\gamma\beta}{\alpha_1+\alpha_2-\gamma\beta}\right)\right\rceil\qquad \forall t\geq 0.$$
		Then the iterates $u_t$ generated by Version I of Algorithm~\ref{alg:MBA_stage} satisfy
		\begin{equation}\label{eqn:stage_wise_dist2}
		\EE\|u_{t}-\bar x\|^2\leq \left(\frac{1}{2}(1+\tfrac{\gamma\beta}{2\alpha_1+2\alpha_2-\gamma\beta})\right)^t \EE\|u_0-\bar x\|^2+\tfrac{2\sigma^2_0\eta}{\alpha_1+\alpha_2-\gamma\beta}.
		\end{equation}
		\item {\bf (Version II)} Choose $\eta\leq \min\{\frac{1}{2L}, \frac{1}{(\gamma\beta-\alpha_1)^+},\frac{1}{\alpha_2}\}$  and set  $$J_t=\left\lceil\left(1+\frac{1}{(\alpha_1+\alpha_2-\gamma\beta)\eta}\right)\log\left(\frac{2(\alpha_1+\alpha_2)}{\alpha_1+\alpha_2-\gamma\beta}\right)\right\rceil\qquad \forall t\geq 0.$$
		Then the iterates $u_t$ generated by Version II of Algorithm~\ref{alg:MBA_stage} satisfy
		\begin{equation}\label{eqn:stage_wise_func2}
		\EE[\varphi(u_{t})-\varphi(\bar x)]\leq \left(\frac{1}{2}(1+\tfrac{\gamma\beta}{\alpha_1+\alpha_2})\right)^t\EE[\varphi(u_0)-\varphi(\bar x)] +\frac{2\sigma^2_0\eta}{1-\tfrac{\gamma\beta}{\alpha_1+\alpha_2}}.
		\end{equation}
	\end{enumerate}
\end{corollary}
\begin{proof}
	With the constant parameter $\eta$, the estimate \eqref{eqn:stage_wise_dist} becomes
	$$\EE\|u_{t}-\bar x\|^2\leq \left(\left(1-\tfrac{(2\alpha_1+2\alpha_2-\gamma\beta)\eta}{1+\eta(2\alpha_1+\alpha_2-\gamma\beta)}\right)^{j_{t-1}}+\tfrac{\gamma\beta}{2\alpha_1+2\alpha_2-\gamma\beta}\right) \EE\|u_{t-1}-\bar x\|^2+\tfrac{2\sigma^2_0\eta}{2\alpha_1+2\alpha_2-\gamma\beta},$$
	where we used the equality $\sum_{j=1}^{j_{t-1}}\frac{\delta_j\Gamma_{j_{t-1}}}{\Gamma_j}=1-\Gamma_{j_{t-1}}$ (Lemma~\ref{lem:gamma_sum}). The definition of $j_{t-1}$ ensures that the coefficient of $\EE \|u_{t-1}-\bar x\|^2$ is at most $\frac{1}{2}(1+\tfrac{\gamma\beta}{2\alpha_1+2\alpha_2-\gamma\beta})$ and therefore
	$$\EE\|u_{t}-\bar x\|^2\leq \frac{1}{2}(1+\tfrac{\gamma\beta}{2\alpha_1+2\alpha_2-\gamma\beta}) \EE\|u_{t-1}-\bar x\|^2+\tfrac{2\sigma^2_0\eta}{2\alpha_1+2\alpha_2-\gamma\beta},$$
	Unrolling the recursion in $t$ completes the proof of \eqref{eqn:stage_wise_dist2}. Similarly, with the constant parameter $\eta$, the estimate \eqref{eqn:stage_wise_func} becomes
	$$\EE[\varphi(u_{t})-\varphi(\bar x)]\leq \left(2\left(1-\tfrac{\eta(\alpha_1+\alpha_2-\gamma\beta)}{1+(\alpha_1-\gamma\beta)\eta}\right)^{j_{t-1}}+\tfrac{\gamma\beta}{\alpha_1+\alpha_2}\right)\EE[\varphi(u_{t-1})-\varphi(\bar x)] +\sigma^2_0\eta.$$
	The definition of $j_{t-1}$ ensures that the coefficient of $\EE[\varphi(u_{t-1})-\varphi(\bar x)]$ is at most $\frac{1}{2}(1+\tfrac{\gamma\beta}{\alpha_1+\alpha_2})$.
	Iterating the recursion in $t$ completes the proof of \eqref{eqn:stage_wise_func2}.
\end{proof}

Consequently, we may equip Algorithm~\ref{alg:MBA_stage} with a geometrically decaying stepsize. The efficiency estimates of the resulting procedure are recorded in the following corollary.

\begin{corollary}\label{cor:delayed-complexities}
	Suppose that Assumptions~\ref{assum:perm_pred}, \ref{ass:smoothness}\ref{assump:smooth_center_2}, and \ref{assum:stoch_model} hold, and that we are in the regime $\frac{\gamma\beta}{\alpha_1+\alpha_2}<1$.
	\begin{enumerate}
		\item {\bf (Version I)} 
		Suppose we have available an estimate $\Delta\geq \|x_0-\bar x\|^2$. Define the augmented strong convexity parameter $\hat \alpha=\alpha_1+\alpha_2-\gamma\beta$. Then Algorithm~\ref{alg:MBA_stage} (Version I) may be augmented with the geometric decay schedule (Algorithm~\ref{alg:geo_decay}) with parameters:
		$$\psi(\delta)\equiv\frac{1}{2}\left(1-\tfrac{\gamma\beta}{2\alpha_1+2\alpha_2-\gamma\beta}\right),\quad C=1, \quad h=\|\cdot-\bar x\|^2,\quad D=\frac{2\sigma^2_0}{\alpha_1+\alpha_2-\gamma\beta}, \quad \delta_0=\frac{1}{2L}.$$
		The resulting procedure will generate a point $x$ satisfying $\EE[\|x-\bar x\|^2]\leq \varepsilon$ using 
		
		\begin{equation}\label{eqn:deployment1}
		\mathcal{O}\left((1-\tfrac{\gamma\beta}{2\alpha_1+2\alpha_2-\gamma\beta})^{-1}\cdot\left(\log\left(\frac{\Delta}{\varepsilon}\right)+\log\left(\frac{\sigma^2_0}{\hat\alpha L\varepsilon}\right)\right)\right)
		\end{equation}	
		deployments and 
		\begin{equation}\label{eqn:samples1}	
		\mathcal{O}\left((1-\tfrac{\gamma\beta}{2\alpha_1+2\alpha_2-\gamma\beta})^{-1}\log\left((1-\tfrac{\gamma\beta}{2\alpha_1+2\alpha_2-\gamma\beta})^{-1}\right)\cdot\left(\frac{L}{\hat\alpha}\log\left(\frac{\Delta}{\varepsilon}\right)+\frac{\sigma_0^2}{\hat\alpha^2\varepsilon}\right)\right) 
		\end{equation}
		samples. 
		\item {\bf (Version II)}	 Suppose we have available an estimate $\Delta\geq \varphi(x_0)-\min \varphi$. Define the augmented strong convexity parameter $\hat \alpha=\alpha_1+\alpha_2-\gamma\beta$. Then Algorithm~\ref{alg:MBA_stage} (Version II) may be augmented with the geometric decay schedule (Algorithm~\ref{alg:geo_decay}) with parameters:
		$$\psi(\delta)\equiv\frac{1}{2}\left(1-\tfrac{\gamma\beta}{\alpha_1+\alpha_2}\right),~ C=1, ~ h(\cdot)=\varphi(\cdot)-\varphi(\bar x),~ D=\frac{2\sigma^2_0}{1-\tfrac{\gamma\beta}{\alpha_1+\alpha_2}}, ~ \delta_0=\min\{\tfrac{1}{2L}, \tfrac{1}{(\gamma\beta-\alpha_1)^+},\tfrac{1}{\alpha_2}\}.$$
		The resulting procedure will generate a point $x$ satisfying $\EE[\varphi(x)-\min \varphi]\leq \varepsilon$ using \begin{equation*}
		\mathcal{O}\left((1-\tfrac{\gamma\beta}{\alpha_1+\alpha_2})^{-1}\cdot\left(\log\left(\frac{\Delta}{\varepsilon}\right)+\log\left(\frac{\sigma^2_0}{1-\frac{\gamma\beta}{\alpha_1+\alpha_2}}\cdot \frac{1}{L\varepsilon}\right)\right)\right)
		\end{equation*}	
		deployments and 
		\begin{equation*}
		\mathcal{O}\left(\left((1-\tfrac{\gamma\beta}{\alpha_1+\alpha_2})^{-1}\right)\cdot\log((1-\tfrac{\gamma\beta}{\alpha_1+\alpha_2})^{-1})\cdot\left(\frac{1}{\hat\alpha \delta_0}\log\left(\frac{\Delta}{\varepsilon}\right)+\frac{\sigma_0^2}{(1-\tfrac{\gamma\beta}{\alpha_1+\alpha_2})\hat\alpha\varepsilon}\right)\right)
		\end{equation*}	
		samples.
	\end{enumerate}
\end{corollary}

\begin{proof}
	We prove the efficiency of Version I of  Algorithm~\ref{alg:MBA_stage}; the proof for version II is completely analogoues. The procedure invokes the algorithm in $k=0,1,\ldots, K$ stages, where $K=\lceil1+\log_2(\frac{D\eta_0}{\varepsilon})\rceil$.
	Applying Lemma~\ref{lem:geo_restart}, the number of deployments coincides with the sum 
	$$\sum_{k=0}^K T_k=\left\lceil\frac{1}{\psi(\delta_0)}\cdot\log\left(\frac{2C\Delta}{\varepsilon}\right)\right\rceil+\sum_{k=1}^K \left\lceil\frac{\log(4C)}{\psi(2^{-k}\delta_0)}\right\rceil.$$
	Simplifying yields \eqref{eqn:deployment1}. The total number of samples used in stage $k$ is simply 
	\begin{equation}\label{eqn:sum_samples1}
	T_k\cdot\left\lceil\left(1+\frac{1}{(2\alpha_1+2\alpha_2-\gamma\beta)\eta_k}\right)\log\left(\frac{2\alpha_1+2\alpha_2-\gamma\beta}{\alpha_1+\alpha_2-\gamma\beta}\right)\right\rceil
	\end{equation}
	where $T_0=\left\lceil\frac{1}{\psi(\delta_0)}\cdot\log\left(\frac{2C\Delta}{\varepsilon}\right)\right\rceil$, $T_k=\left\lceil\frac{1}{\psi(\delta_k)}\cdot\log(4C)\right\rceil$ for $k\geq 1$, and $\eta_k=2^{-k}\eta_0$. Observe
	$\sum_{k=1}^K \frac{1}{\eta_k}=\sum_{k=1}^K 2^{k+1} L\leq 2^{K+2}L\leq 16\frac{D\eta_0}{\varepsilon}\cdot L=\frac{16\sigma^2_0}{\hat\alpha\varepsilon}$. Summing the expressions \eqref{eqn:sum_samples1} across $k=0,1,\ldots, K$ yields \eqref{eqn:samples1}.
\end{proof}

\subsection{Accelerated method}
We next explore applying an accelerated stochastic gradient method within inexact repeated minimization. Throughout the section, we suppose that Assumptions~\ref{assum:smoothness}, \ref{assum:perm_pred}, \ref{ass:strong_conv_perm}\ref{it:str_conv_uniform}, \ref{ass:smoothness}\ref{assump:smooth_center_2} hold. The  accelerated method is summarized as  Algorithm~\ref{alg:accel_stage}.

\begin{algorithm}[H]
	{\bf Input:} initial $u_0$, integers  $J,T\in\mathbb{N}$.
	
	\For{$t=0,1,\ldots,T$}{	
		Set $x_0=u_t$
		
		\For{$j=1,\ldots, J$}
		{
			\begin{equation*}
			\begin{aligned}
			&\textrm{Sample }z_t\sim \cD(u_t)\\
			&x_{j}=\prox_{r/L}\left(y_{j-1}-\eta \nabla \ell(y_{j-1},z_t)\right) &\\
			&y_j=x_j+\frac{1-\sqrt{\alpha/L}}{1+\sqrt{\alpha/L}}(x_t-x_{t-1})
			\end{aligned}
			\end{equation*}   
		}
		Set $u_{t+1}=x_J$
	}		
	\caption{Stagewise Stochastic Accelerated Gradient}\label{alg:accel_stage}
\end{algorithm}

The idea of the argument is as follows. The convergence guarantees for the static stochastic accelerated gradient method \cite[Corollary~13]{kulunchakov2019generic} directly imply
$$\EE_t[\varphi_x(x_{J})-\varphi_x(y)]\leq \left(1-\sqrt{\frac{\alpha}{L}}\right)^J\cdot \left(\varphi_x(x_0)-\varphi_{ x}(y)+\frac{\alpha}{2}\|x_{J}-y\|^2\right)+\frac{\sigma^2}{\sqrt{\alpha L}},$$
for all points $y$, where $\EE_t$ denotes the conditional expectation on $u_t$. 
The following general lemma shows how  to translate efficiency estimates of this type for $\varphi_x$ to a similar estimate for $\varphi$.

\begin{theorem}[Reduction]\label{thm:transf_guarantee}
	Suppose we are in the parameter regime $\rho<\frac{1}{2}$. Fix a point $x$ and counter $j\in \mathbb{N}$ and suppose that an algorithm applied to the static problem $\mathtt{St}(\cD(x))$
	generates a random point $x_j$ satisfying 
	\begin{equation}\label{eqn:assumed_improv}
	\EE[\varphi_x( x_j)-\varphi_x(y)]\leq C(1-q)^j \left(\varphi_x(x)-\varphi_x(y)+\frac{\alpha}{2}\| x_j-y\|^2\right)+B\qquad \forall y\in \R^d,
	\end{equation}
	where $B,C\geq 0$ and $q\in (0,1)$ are some constants.
	Then the estimate holds:
	$$\EE[\varphi( x_j)-\varphi(\bar x)]\leq \frac{\frac{\gamma\beta}{\alpha}+2C(1-q)^j}{{1-\frac{\gamma\beta}{\alpha}-C(1-q))^j}}\cdot(\varphi(x)-\varphi(\bar x))+\frac{B}{1-\frac{\gamma\beta}{\alpha}-C(1-q))^j}$$
\end{theorem}
\begin{proof}
	Appealing to the gap deviation inequality \eqref{eqn:gap_comp}, we  compute
	\begin{align}
	\varphi( x_j)-\varphi(\bar x)&\leq \varphi_x( x_j)-\varphi_{x}(\bar x)+\gamma\beta\| x_j-\bar x\|\cdot\|x-\bar x\|\notag\\
	&\leq \varphi_x( x_j)-\varphi_{x}(\bar x)+\frac{\gamma\beta}{2}\| x_j-\bar x\|^2+\frac{\gamma\beta}{2}\|x-\bar x\|^2 \label{eqn:young_getit} \\
	&\leq \frac{\gamma\beta}{\alpha}(\varphi(x)- \varphi(\bar x))+\varphi_x( x_j)-\varphi_{x}(\bar x)+\frac{\gamma\beta}{2}\| x_j-\bar x\|^2, \label{eqn:str_conv_3}
	\end{align}
	where \eqref{eqn:young_getit} follows from Young's inequality and \eqref{eqn:str_conv_3} follows from strong convexity of $\varphi$.
	Next, taking into account \eqref{eqn:assumed_improv} with $y=\bar x$, we conclude		
	\begin{align*}
	\EE[\varphi( x_j)-\varphi(\bar x)]&\leq  \frac{\gamma\beta}{\alpha}(\varphi(x)-\varphi(\bar x))+C(1-q)^j\left(\varphi_x(x)-\varphi_{x}(\bar x)+\frac{\alpha}{2}\| x_j-\bar x\|^2\right)\\
	&\quad+B+\frac{\gamma\beta}{2}\EE\| x_j-\bar x\|^2.
	\end{align*}
	The gap deviation inequality \eqref{eqn:gap_comp} guarantees
	$$[\varphi_x(x)-\varphi_{x}(\bar x)]-[\varphi(x)-\varphi(\bar x)]\leq \gamma\beta\|x-\bar x\|^2\leq \frac{2\gamma\beta}{\alpha}(\varphi(x)-\varphi(\bar x)),$$
	while strong convexity of $\varphi$ implies
	$\frac{1}{2}\| x_j-\bar x\|^2\leq \frac{1}{\alpha}(\varphi( x_j)-\varphi(\bar x)).$
	We therefore deduce
	$$\EE[\varphi( x_j)-\varphi(\bar x)]\leq \frac{\frac{\gamma\beta}{\alpha}+C(1-q)^j\left(1+\frac{2\gamma\beta}{\alpha}\right)}{{1-\frac{\gamma\beta}{\alpha}-C(1-q))^j}}\cdot(\varphi(x)-\varphi(\bar x))+\frac{B}{1-\frac{\gamma\beta}{\alpha}-C(1-q))^j}.$$
	Invoking the upper bound $\frac{2\gamma\beta}{\alpha}\leq 1$ completes the proof.	
\end{proof}

Combining Algorithm~\ref{alg:accel_stage} with the minibatch restart strategy (Algorithm~\ref{alg:restart}) yields an overall scheme with the following efficiency estimates.

\begin{theorem}[Accelerated Stochastic Gradient]
	Suppose that Assumptions~\ref{assum:smoothness}, \ref{assum:perm_pred}, \ref{ass:strong_conv_perm}\ref{it:str_conv_uniform}, \ref{ass:smoothness}\ref{assump:smooth_center_2} hold and
	that we are in the regime $\rho<\frac{1}{2}$. Set the number of inner iterations $J=\sqrt{\frac{L}{\alpha}}\log\left(\frac{4}{\frac{1}{2}-\rho}\right)$. Then the iterates $u_t$ generated by Algorithm~\ref{alg:accel_stage} satisfy
	\begin{equation}\label{eqn:one_step_improv_fast_grad}
	\EE[\varphi(u_{t})-\varphi(u_0)]\leq \left(\frac{1}{2(1-\rho)}\right)^t\cdot\EE[\varphi(u_0)-\varphi(\bar x)]+\frac{32\sigma^2}{5(1-\frac{\rho}{1-\rho})\sqrt{\alpha L}}.
	\end{equation}
	Consequently, if we have available an estimate $\Delta\geq \|x_0-\bar x\|^2$, then Algorithm~\ref{alg:accel_stage} may be combined with the minibatch restart strategy (Algorithm~\ref{alg:restart}) by setting 
	$$h(\cdot)=\varphi(\cdot)-\varphi(\bar x),\quad \tau=1-\frac{1}{2(1-\rho)},\quad C=1,\quad B=\frac{32\sigma^2}{5(1-\frac{\rho}{1-\rho})\sqrt{\alpha L}}.$$
	The resulting procedure will generate a point $x$ satisfying $\EE[\varphi(x)-\min \varphi]\leq \varepsilon$ using 
	\begin{equation}\label{eqn:deployment3}
	\mathcal{O}\left(\left(1-\tfrac{1}{2(1-\rho)}\right)^{-1}\left(\log\left(\frac{\Delta}{\varepsilon}\right)+ \log\left(\frac{\sigma^2}{(1-\frac{\rho}{1-\rho})\varepsilon\sqrt{\alpha L}}\right)\right)\right)
	\end{equation}	
	deployments and 
	\begin{equation}\label{eqn:samples3}	
	\mathcal{O}\left(\left(1-\tfrac{1}{2(1-\rho)}\right)^{-1}\cdot\left(\sqrt{\frac{L}{\alpha}}\log\left(\frac{\Delta}{\varepsilon}\right)+\frac{\sigma^2}{(1-\tfrac{\rho}{1-\rho})\alpha\varepsilon}\right)\right)
	\end{equation}	
	samples.
\end{theorem}
\begin{proof}
	Fix an index $t=0,1,\ldots, T$ and set $x=u_t$. The convergence guarantees for the static stochastic accelerated gradient method \cite[Corollary~13]{kulunchakov2019generic} directly imply
	$$\EE_t[\varphi_x(x_{J})-\varphi_x(y)]\leq \left(1-\sqrt{\frac{\alpha}{L}}\right)^J\cdot \left(\varphi_x(x_0)-\varphi_{ x}(y)+\frac{\alpha}{2}\|x_{J}-y\|^2\right)+\frac{\sigma^2}{\sqrt{\alpha L}},$$
	for all points $y$, where $\EE_t$ denotes the conditional expectation on $u_t$. Theorem~\ref{thm:transf_guarantee} along with the tower rule for expectation therefore implies
	$$\EE_t[\varphi(u_{t+1})-\varphi(\bar x)]\leq \frac{\frac{\gamma\beta}{\alpha}+2(1-\sqrt{\alpha/L})^J}{{1-\frac{\gamma\beta}{\alpha}-(1-\sqrt{\alpha/L}))^J}}\cdot(\varphi(u_t)-\varphi(\bar x))+\frac{\sigma^2/\sqrt{\alpha L}}{1-\frac{\gamma\beta}{\alpha}-(1-\sqrt{\alpha/L}))^J}.$$
	The choice $J$ guarantees that the coefficient in front of $\varphi(u_t)-\varphi(\bar x)$ is at most $\frac{1}{2(1-\rho)}$ and the inequality $1-\rho-(1-\tau)^J\geq \frac{3}{4}-\frac{7}{8}\rho\geq \frac{5}{16}$ holds. We therefore conclude
	$$\EE_t[\varphi(u_{t+1})-\varphi(\bar x)]\leq \frac{1}{2(1-\rho)}\cdot(\varphi(u_t)-\varphi(\bar x))+\frac{16\sigma^2}{5\sqrt{\alpha L}}.$$
	Using the tower rule for expectations and iterating the recursion directly yields \eqref{eqn:one_step_improv_fast_grad}. 
	
	Following the notation of Lemma~\ref{lem:minbatch_restart}, the number of deployments is given by 
	$$\sum_{k=0}^K T_k=\left\lceil\sqrt{\frac{1}{\tau}}\log\left(\frac{2\Delta}{\varepsilon}\right)\right\rceil+\left\lceil 1+\log_2\left(\frac{32\sigma^2}{5(1-\frac{\rho}{1-\rho})\varepsilon\sqrt{\alpha L}}\right)\right\rceil\cdot\left\lceil\sqrt{\frac{1}{\tau}}\log(4)\right\rceil.$$
	The estimate \eqref{eqn:deployment3} follows immediately. The total number of samples used is 
	$$J\cdot\sum_{k=0}^K m_k T_k=\mathcal{O}\left(\frac{1}{\tau\sqrt{\kappa}}\log\left(\frac{4}{\frac{1}{2}-\rho}\right)\left(\log\left(\frac{\Delta}{\varepsilon}\right)+\frac{B}{\varepsilon}\right)\right).$$
	The claimed estimate \eqref{eqn:samples3} follows immediately.
\end{proof}

Thus, in the nearly optimal parameter regime $\rho<\frac{1}{2}$, (Algorithm~\ref{alg:accel_stage}) enjoys the same sample efficiency as its online counterpart, while requiring a number of deployments that is only logarithmic in the problem parameters.

\paragraph*{Acknowledgements:} We thank Celestine Mendler-D\"{u}nner, Moritz Hardt, Juan C. Perdomo, and Tijana Zrnic for useful comments and suggestions.

\bibliographystyle{plain}
\bibliography{bibliography}

\begin{thebibliography}{10}

\bibitem{ahmed2000strategic}
Shabbir Ahmed.
\newblock {\em Strategic planning under uncertainty: Stochastic integer
  programming approaches}.
\newblock PhD thesis, University of Illinois at Urbana-Champaign, 2000.

\bibitem{ajalloeian2020analysis}
Ahmad Ajalloeian and Sebastian~U Stich.
\newblock Analysis of sgd with biased gradient estimators.
\newblock {\em arXiv preprint arXiv:2008.00051}, 2020.

\bibitem{asi2019stochastic}
Hilal Asi and John~C Duchi.
\newblock Stochastic (approximate) proximal point methods: Convergence,
  optimality, and adaptivity.
\newblock {\em SIAM Journal on Optimization}, 29(3):2257--2290, 2019.

\bibitem{aybat2019universally}
Necdet~Serhat Aybat, Alireza Fallah, Mert Gurbuzbalaban, and Asuman Ozdaglar.
\newblock A universally optimal multistage accelerated stochastic gradient
  method.
\newblock In {\em Advances in Neural Information Processing Systems}, pages
  8525--8536, 2019.

\bibitem{bartlett1992learning}
Peter~L Bartlett.
\newblock Learning with a slowly changing distribution.
\newblock In {\em Proceedings of the fifth annual workshop on Computational
  learning theory}, pages 243--252, 1992.

\bibitem{bartlett2000learning}
Peter~L Bartlett, Shai Ben-David, and Sanjeev~R Kulkarni.
\newblock Learning changing concepts by exploiting the structure of change.
\newblock {\em Machine Learning}, 41(2):153--174, 2000.

\bibitem{bechavod2020causal}
Yahav Bechavod, Katrina Ligett, Zhiwei~Steven Wu, and Juba Ziani.
\newblock Causal feature discovery through strategic modification.
\newblock {\em arXiv preprint arXiv:2002.07024}, 2020.

\bibitem{beck}
A.~Beck and M.~Teboulle.
\newblock A fast iterative shrinkage-thresholding algorithm for linear inverse
  problems.
\newblock {\em SIAM J. Imaging Sci.}, 2(1):183--202, 2009.

\bibitem{beck2017first}
Amir Beck.
\newblock {\em First-order methods in optimization}.
\newblock SIAM, 2017.

\bibitem{bertsekas1997nonlinear}
Dimitri~P Bertsekas.
\newblock Nonlinear programming.
\newblock {\em Journal of the Operational Research Society}, 48(3):334--334,
  1997.

\bibitem{besbes2015non}
Omar Besbes, Yonatan Gur, and Assaf Zeevi.
\newblock Non-stationary stochastic optimization.
\newblock {\em Operations research}, 63(5):1227--1244, 2015.

\bibitem{besbes2019optimal}
Omar Besbes, Yonatan Gur, and Assaf Zeevi.
\newblock Optimal exploration--exploitation in a multi-armed bandit problem
  with non-stationary rewards.
\newblock {\em Stochastic Systems}, 9(4):319--337, 2019.

\bibitem{bruckner2012static}
Michael Br{\"u}ckner, Christian Kanzow, and Tobias Scheffer.
\newblock Static prediction games for adversarial learning problems.
\newblock {\em The Journal of Machine Learning Research}, 13(1):2617--2654,
  2012.

\bibitem{bubeck2014convex}
S{\'e}bastien Bubeck.
\newblock Convex optimization: Algorithms and complexity.
\newblock {\em arXiv preprint arXiv:1405.4980}, 2014.

\bibitem{dalvi2004adversarial}
Nilesh Dalvi, Pedro Domingos, Sumit Sanghai, and Deepak Verma.
\newblock Adversarial classification.
\newblock In {\em Proceedings of the tenth ACM SIGKDD international conference
  on Knowledge discovery and data mining}, pages 99--108, 2004.

\bibitem{davis2019stochastic}
Damek Davis and Dmitriy Drusvyatskiy.
\newblock Stochastic model-based minimization of weakly convex functions.
\newblock {\em SIAM Journal on Optimization}, 29(1):207--239, 2019.

\bibitem{devolder2014first}
Olivier Devolder, Fran{\c{c}}ois Glineur, and Yurii Nesterov.
\newblock First-order methods of smooth convex optimization with inexact
  oracle.
\newblock {\em Mathematical Programming}, 146(1-2):37--75, 2014.

\bibitem{DuchiSinger2009}
John Duchi and Yoram Singer.
\newblock Efficient online and batch learnng using forward backward splitting.
\newblock {\em Journal of Machine Learning Research}, 10:2899--2934, 2009.

\bibitem{dupacova2006optimization}
Jitka Dupacov{\'a}.
\newblock Optimization under exogenous and endogenous uncertainty.
\newblock {\em University of West Bohemia in Pilsen}, 2006.

\bibitem{dvurechensky2017universal}
Pavel Dvurechensky, Alexander Gasnikov, and Dmitry Kamzolov.
\newblock Universal intermediate gradient method for convex problems with
  inexact oracle.
\newblock {\em arXiv preprint arXiv:1712.06036}, 2017.

\bibitem{gama2014survey}
Jo{\~a}o Gama, Indr{\.e} {\v{Z}}liobait{\.e}, Albert Bifet, Mykola Pechenizkiy,
  and Abdelhamid Bouchachia.
\newblock A survey on concept drift adaptation.
\newblock {\em ACM computing surveys (CSUR)}, 46(4):1--37, 2014.

\bibitem{ghadimi2012optimal}
Saeed Ghadimi and Guanghui Lan.
\newblock Optimal stochastic approximation algorithms for strongly convex
  stochastic composite optimization i: A generic algorithmic framework.
\newblock {\em SIAM Journal on Optimization}, 22(4):1469--1492, 2012.

\bibitem{ghadimi2013optimal}
Saeed Ghadimi and Guanghui Lan.
\newblock Optimal stochastic approximation algorithms for strongly convex
  stochastic composite optimization, ii: shrinking procedures and optimal
  algorithms.
\newblock {\em SIAM Journal on Optimization}, 23(4):2061--2089, 2013.

\bibitem{hardt2016strategic}
Moritz Hardt, Nimrod Megiddo, Christos Papadimitriou, and Mary Wootters.
\newblock Strategic classification.
\newblock In {\em Proceedings of the 2016 ACM conference on innovations in
  theoretical computer science}, pages 111--122, 2016.

\bibitem{hazan2016intro}
Elad Hazan.
\newblock Introduction to online convex optimization.
\newblock {\em Foundations and Trends in Optimization}, 2(3-4):157--325, 2016.

\bibitem{hellemo2018decision}
Lars Hellemo, Paul~I Barton, and Asgeir Tomasgard.
\newblock Decision-dependent probabilities in stochastic programs with
  recourse.
\newblock {\em Computational Management Science}, 15(3-4):369--395, 2018.

\bibitem{jonsbraaten1998class}
Tore~W Jonsbr{\aa}ten, Roger~JB Wets, and David~L Woodruff.
\newblock A class of stochastic programs withdecision dependent random
  elements.
\newblock {\em Annals of Operations Research}, 82:83--106, 1998.

\bibitem{kaggle}
Kaggle.
\newblock {\em Give me some credit}.
\newblock 2012.
\newblock \url{https://www.kaggle.com/c/GiveMeSomeCredit/data}.

\bibitem{kantorovich1958space}
Leonid~Vitalyevich Kantorovich and Gennady~S Rubinstein.
\newblock On a space of completely additive functions.
\newblock {\em Vestnik Leningrad. Univ}, 13(7):52--59, 1958.

\bibitem{kuh1991learning}
Anthony Kuh, Thomas Petsche, and Ronald~L Rivest.
\newblock Learning time-varying concepts.
\newblock In {\em Advances in Neural Information Processing Systems}, pages
  183--189, 1991.

\bibitem{kulunchakov2019estimate}
Andrei Kulunchakov and Julien Mairal.
\newblock Estimate sequences for stochastic composite optimization: Variance
  reduction, acceleration, and robustness to noise.
\newblock {\em arXiv preprint arXiv:1901.08788}, 2019.

\bibitem{kulunchakov2019generic}
Andrei Kulunchakov and Julien Mairal.
\newblock A generic acceleration framework for stochastic composite
  optimization.
\newblock In {\em Advances in Neural Information Processing Systems}, pages
  12556--12567, 2019.

\bibitem{lan2012optimal}
Guanghui Lan.
\newblock An optimal method for stochastic composite optimization.
\newblock {\em Mathematical Programming}, 133(1-2):365--397, 2012.

\bibitem{McMahan2017}
Brendan McMahan.
\newblock A survey of algorithms and analysis for adaptive online learning.
\newblock {\em Journal of Machine Learning Research}, 18:1--50, 2017.

\bibitem{mendler2020stochastic}
Celestine Mendler-D{\"u}nner, Juan~C Perdomo, Tijana Zrnic, and Moritz Hardt.
\newblock Stochastic optimization for performative prediction.
\newblock {\em arXiv preprint arXiv:2006.06887}, 2020.

\bibitem{milli2019socialcost}
Smitha Milli, John Miller, Anca~D. Dragan, and Moritz Hardt.
\newblock The social cost of strategic classification.
\newblock In {\em Proceedings of the Conference on Fairness, Accountability,
  and Transparency}, pages 230--239, January 2019.

\bibitem{intro_lect}
Y.~Nesterov.
\newblock {\em Introductory lectures on convex optimization}, volume~87 of {\em
  Applied Optimization}.
\newblock Kluwer Academic Publishers, Boston, MA, 2004.
\newblock A basic course.

\bibitem{nest_orig}
Yu. Nesterov.
\newblock A method for solving the convex programming problem with convergence
  rate {$O(1/k^{2})$}.
\newblock {\em Dokl. Akad. Nauk SSSR}, 269(3):543--547, 1983.

\bibitem{nesterov2009primaldual}
Yurii Nesterov.
\newblock Primal-dual subgradient methods for convex problems.
\newblock {\em Mathematical Programming}, 120(1):221--259, 2009.

\bibitem{nesterov2018lectures}
Yurii Nesterov.
\newblock {\em Lectures on convex optimization}, volume 137.
\newblock Springer, 2018.

\bibitem{perdomo2020performative}
Juan~C Perdomo, Tijana Zrnic, Celestine Mendler-D{\"u}nner, and Moritz Hardt.
\newblock Performative prediction.
\newblock In {\em Proceedings of the International Conference on Machine
  Learning (ICML)}, 2020.
\newblock arXiv preprint arXiv:2002.06673.

\bibitem{RW98}
R.T. Rockafellar and R.J-B. Wets.
\newblock {\em Variational {A}nalysis}.
\newblock Grundlehren der mathematischen Wissenschaften, Vol 317, Springer,
  Berlin, 1998.

\bibitem{rubinstein1993discrete}
Reuven~Y Rubinstein and Alexander Shapiro.
\newblock {\em Discrete event systems: Sensitivity analysis and stochastic
  optimization by the score function method}.
\newblock John Wiley \& Sons Inc, 1993.

\bibitem{ryu2014stochastic}
Ernest~K Ryu and Stephen Boyd.
\newblock Stochastic proximal iteration: a non-asymptotic improvement upon
  stochastic gradient descent.
\newblock {\em Author website, early draft}, 2014.

\bibitem{shalevshwartz2012online}
Shai Shalev-Shwartz.
\newblock Online learning and online convex optimization.
\newblock {\em Foundations and Trends in Machine Learning}, 4(2):107--194,
  2012.

\bibitem{stonyakin2019inexact}
Fedor Stonyakin, Alexander Gasnikov, Alexander Tyurin, Dmitry Pasechnyuk, Artem
  Agafonov, Pavel Dvurechensky, Darina Dvinskikh, Alexey Kroshnin, and Victorya
  Piskunova.
\newblock Inexact model: A framework for optimization and variational
  inequalities.
\newblock {\em arXiv preprint arXiv:1902.00990}, 2019.

\bibitem{varaiya1988stochastic}
Pravin Varaiya and RJ-B Wets.
\newblock Stochastic dynamic optimization approaches and computation.
\newblock 1988.

\bibitem{xiao2010jmlr}
Lin Xiao.
\newblock Dual averaging methods for regularized stochastic learning and online
  optimization.
\newblock {\em Journal of Machine Learning Research}, 11:2543--2596, 2010.

\bibitem{zinkevich2003}
Martin Zinkevich.
\newblock Online convex programming and generalized infinitesimal gradient
  ascent.
\newblock In {\em Proceedings of the 20th International Conference on Machine
  Learning (ICML)}, pages 928--936, Washington DC, 2003.

\bibitem{zliobaite2016applications}
Indr{\.e} {\v Z}liobait{\.e}, Mykola Pechenizkiy, and Jo{\~a}o Gama.
\newblock An overview of concept drift applications.
\newblock In N.~Japkowica and J.~Stefanowski, editors, {\em Big Data Analysis:
  New Algorithms for a New Society}, pages 91--114. Springer, 2016.

\end{thebibliography}

\appendix

\section{Averaging lemma}
In this section, we investigate recursions of the form 
\begin{equation}\label{eqn:recurs_appe}
\delta_t\cdot \EE[h(x_{t})]\leq (1-c_1\delta_t)D_{t-1}-(1+c_2\delta_t)D_{t}+\omega_t\qquad \forall t\geq 1,
\end{equation}
where $h(\cdot)$ is a convex function, $c_1,c_2\in \R$ are some constants, and the sequences $\{\delta_t\}_{t\geq 1}$, $\{D_t\}_{t\geq 0}$, $\{\omega_t\}_{t\geq 1}$ are nonnegative. The expectation is taken over randomness in the points $x_t$, which are usually the iterates produced by a stochastic algorithm. In typical circumstances $\delta_t>0$ is a user-specified sequence (e.g. stepsize).
Our goal is to determine the rate at which the value $h(\hat x_t)$ tends to zero, where $\hat x_t$ is a running average of the iterates. 
 Most of the material in this section follows the discussion in \cite[Section A.2, A.3]{kulunchakov2019estimate} and \cite{ghadimi2012optimal}.

We begin with the following elementary lemma, which can be proved by induction. 
\begin{lemma}\label{lem:gamma_sum}
	Consider a sequence of  weights $\{\delta_t\}_{t\geq 0}$ in $(0,1)$. Then the partial products $\Gamma_t=\prod_{i=1}^t(1-\delta_i)$ satisfy the equation
	$\sum_{i=1}^t\frac{\delta_i}{\Gamma_i}+1=\frac{1}{\Gamma_t}.$
\end{lemma}

The convergence guarantees will be stated in terms of the augmented weights 
$$\hat \delta_t:=\frac{\delta_t(c_1+c_2)}{1+c_2\delta_t}\qquad\textrm{ and }\qquad\hat \Gamma_t:=\prod_{i=1}^t(1-\hat \delta_i),$$
and average iterates that are recursively defined as
$$\hat x_0:=x_0 \qquad \textrm{and}\qquad \hat{x}_t:=(1-\hat\delta_t)\hat x_{t-1}+\hat\delta_t x_t\qquad \forall t\geq 1.$$

The following lemma establishes the sought upper bound on the values $h(\hat x_{t})$. This result follows quickly by reducing to the special case $c_1=1$, $c_2=0$, for which the lemma was proved in \cite[Lemma 12]{kulunchakov2019estimate}.

\begin{lemma}[Averaging]\label{lem:average2}
Consider a convex function $h\colon\R^d\to\R\cup\{\infty\}$ and let $\{x_t\}_{t\geq 0}\subset\R^d$ be a sequence of random vectors in $\dom h$. Suppose that there are constants $c_1,c_2\in \R$ and nonnegative sequences $\{\delta_t\}_{t\geq 1}$, $\{D_t\}_{t\geq 0}$, and $\{\omega_t\}_{t\geq 1}$  satisfying \eqref{eqn:recurs_appe}.
 Suppose moreover the relations $c_1+c_2>0$, $1- c_1\delta_t>0$, and $1+c_2\delta_t> 0$  hold for all $t\geq 1$.
 Then the estimate holds:
\begin{equation}\label{eqn:main_eqn_end}
\frac{\EE [ h\left(\hat x_t\right)]}{c_1+c_2}+ D_t\leq  \hat \Gamma_t\left( \frac{h(x)}{c_1+c_2}+ D_0+\sum_{i=1}^t\frac{ \omega_i}{ \hat \Gamma_i(1+c_2\delta_i)}\right).
\end{equation}
\end{lemma}
\begin{proof}
The special case of this theorem in the setting $c_1=1$ and $c_2=0$	was proved in \cite[Lemma 12]{kulunchakov2019estimate}. We will reduce the general case to this setting. To this end, dividing  the inequality \eqref{eqn:recurs_appe} through by $1+c_2\delta_t$ yields
$$\hat \delta_t\cdot \EE[\hat h(x_{t})]\leq \left(1-\hat \delta_t\right) D_{t-1}-D_{t}+\hat\omega_t,
$$
where we set  $\hat h(x):=\frac{h(x)}{c_1+c_2}$ and $\hat\omega_t:=\frac{\omega_t}{1+c_2\delta_t}$. Noting the inclusion $\hat \delta_t\in (0,1)$, 
an application of 
 \cite[Lemma 12]{kulunchakov2019estimate} completes the proof.
\end{proof}

Observe that the efficiency estimate~\eqref{eqn:main_eqn_end} is guided by the augmented weights $\hat \delta_i$. Therefore, when applying Lemma~\ref{lem:average2}, it is most convenient to specify $\hat \delta_i$ rather than $ \delta_i$. For any desired value $\hat\delta\in (0,1)$, we may then simply set 
$\delta=\frac{\hat \delta}{c_1+c_2-c_2\hat\delta}.$  
Typical choices of the sequence $\hat \delta_t$ and the corresponding products $\hat\Gamma_t$ are summarized in Table~\ref{table:delta}.

\begin{table}[h!]
	\renewcommand{\arraystretch}{1.5}
	\begin{center}
		\begin{tabular}{ ||c|| c |c | c| c| c|}
			\hline
			$\hat\delta_t$ & $\hat \delta$ & $\frac{1}{t+1}$ & $\frac{2}{t+2}$ & $\min\{\frac{1}{t+1},\hat\delta\}$ & $\min\{\frac{2}{t+2},\hat\delta\}$ \\ 
			\hline
			$\hat\Gamma_t$ & $(1-\hat\delta)^t$ &  $\frac{1}{t+1}$ &  $\frac{2}{(t+1)(t+2)}$ &  $\begin{cases} (1-\hat\delta)^t &\mbox{if } t < t_0 \\
			\frac{\hat\Gamma_{t_0-1}t_0}{t+1} & \mbox{if } t \geq t_0 \end{cases}$ & $\begin{cases} (1-\hat\delta)^t &\mbox{if } t < t_0' \\
			\frac{\hat\Gamma_{t_0'-1}t_0'(t_0'+1)}{(t+1)(t+2)} & \mbox{if } t \geq t_0' \end{cases}$\\  
			\hline
		\end{tabular}
	\end{center}
	\caption{The table describes the values of $\hat\Gamma_t=\prod_{i=1}^t(1-\hat\delta_i)$ under various common choices of $\hat\delta_t$. Here, we assume $\hat\delta\in (0,1)$ and set $t_0=\left\lceil\frac{1}{\hat\delta}-1\right\rceil$, $t_0'=\left\lceil\frac{2}{\hat\delta}-2\right\rceil$.}\label{table:delta}
\end{table}

It is now straightforward to evaluate the right side of \eqref{eqn:main_eqn_end} under the different choices of $\hat\delta_t$, as specified in Table~\ref{table:delta}. For simplicity, we record the resulting estimate only in the case that $\hat \delta_t$ is constant across the iterations.

\begin{corollary}[Constant parameter]\label{cor:const_param}
Assume the setting of Lemma~\ref{lem:average2}, and suppose that $\delta_t=\delta$ and $\omega_i=\omega$ are constant. 
Then the estimate holds:
\begin{equation}\label{eqn:lin_conver_app}
\EE [ h\left(\hat x_t\right)]+ (c_1+c_2)D_t\leq  \left(\frac{1- c_1\delta}{1+c_2\delta}\right)^t\left( h(x)+ (c_1+c_2)D_0\right)+\frac{\omega}{\delta}.
\end{equation}
\end{corollary}
\begin{proof}The result follows directly from Lemma~\ref{lem:average2} and algebraic manipulations.
Namely, we compute 
$$\sum_{i=1}^t\frac{\hat \Gamma_t \omega}{ \hat \Gamma_i(1+c_2\delta)}=\frac{\omega}{\delta(c_1+c_2)}\sum_{i=1}^t\frac{\hat \Gamma_t\delta(c_1+c_2)}{ \hat \Gamma_i(1+c_2\delta)}=\frac{\omega}{\delta(c_1+c_2)}\sum_{i=1}^t\frac{\hat \Gamma_t\hat \delta}{ \hat \Gamma_i}=\frac{\omega}{\delta(c_1+c_2)}\cdot(1-\hat \Gamma_t),$$
where the last equality follows from Lemma~\ref{lem:gamma_sum}. An application of Lemma~\ref{lem:average2} completes the proof.
\end{proof}

\section{Stagewise scheme for improved efficiency}\label{sec:restart_improv_eff}
This section describes two restart schemes for improving the efficiency of constant-step stochastic algorithms. We follow the discussion in \cite[Appendix B.1]{kulunchakov2019generic}, though closely related ideas can be found for example in \cite{ghadimi2013optimal,aybat2019universally}. There are two  complementary approaches. The first is based on restarting the constant step algorithm with exponentially increasing minibatches in order to decrease the variance of gradient estimators. The second approach  restarts the algorithm with geometrically decreasing step-sizes. We discuss these two strategies in turn. 

\subsection{Minibatch restart}
  Suppose that we have available a stochastic algorithm $\mathcal{A}(y,m,T)$ that generates a point $y_T$ satisfying 
\begin{equation}\label{eqn:effic_lin_sib}
\EE[h(y_T)]\leq C(1-\tau)^T h(y_0)+\frac{B}{m},
\end{equation}
where $h$ is a nonnegative function and $C,B>0$ and $\tau\in(0,1)$ are some constants. In concrete circumstances, $h(y)$ may denote the function gap $\varphi(y)-\min\varphi$ or the square distance to the solution $\|y-\bar x\|^2$; the point $y_0$ specifies the initialization and $T$ is the number of iterations; $B>0$ is proportional to the variance $\sigma^2$ of stochastic gradients used by the algorithm, which can be reduced by a factor of $m\in \mathbb{N}$ by using a minibatch of size $m$. Thus the overall sample complexity of the execution is proportional to $m\cdot T$.
For example, the stochastic gradient method on a smooth strongly convex function $\varphi$ satisfies \eqref{eqn:effic_lin_sib}
with $h(y)=\varphi(y)-\min \varphi$, $C=1$, $B=\sigma^2/L$, and $\tau=\alpha/L$.   If the stochastic gradients are used in minibatches of size $m$, the variance $\sigma^2$ shrinks by a factor of $m$. 

The minibatch restart procedure, formally described in Algorithm~\ref{alg:restart}, runs in stages by repeatedly calling $\mathcal{A}(y_k,m_k,T_k)$. Specifically, in stage $k$, the method
sets $m_k=2^k$ and finds a point $y$ satisfying $\EE[h(y_T)]\leq 2\cdot \frac{B}{m_k}$, while using the output from the previous stage as a warmstart.

\SetKwProg{Init}{init}{}{}
\begin{algorithm}[H]
	\KwIn{ $y_0\in\R^d$, $B,C>0$, $\tau\in (0,1)$, estimate $\Delta\geq h(y_0)$,  
	accuracy $\varepsilon>0$, algorithm $\mathcal{A}(y,m,T)$ satisfying \eqref{eqn:effic_lin_sib}.}

    \SetKwProg{myproc}{Initialize:}{}{}
      \myproc{Set  $y=\mathcal{A}(y_0,m_0,T_0)$ with $m_0=1$ and $T_0=\tau^{-1}\cdot\log(\frac{2C\Delta}{\varepsilon})$.}
		
	Set $K=\left\lceil 1+\log_2\left(\frac{B}{\varepsilon}\right)\right\rceil$.

	{\bf Step} $k=1,\ldots, K$: 
$$\textrm{Set } y=\mathcal{A}(y,m_k,T_k)\quad\textrm{ with }\quad m_k=2^k,~ T_k=\lceil\tau^{-1}\cdot\log(4C)\rceil.$$

{\bf Return} $y$.
	
	\caption{Minibatch restart}\label{alg:restart}
\end{algorithm}

The overall efficiency of the scheme is summarized in the following lemma. The proof can be found in \cite[Appendix B.1]{kulunchakov2019generic}.

\begin{lemma}[Minibatch restart]\label{lem:minbatch_restart}
The point $y$ returned by Algorithm~\ref{alg:restart} satisfies $\EE[h(y)]\leq\varepsilon$ and the efficiency estimate holds:	
$$\sum_{k=0}^K m_k T_k =\mathcal{O}\left(\frac{1}{\tau}\cdot\log\left(\frac{2C\Delta}{\varepsilon}\right)+\frac{B\log(4C)}{\tau\varepsilon}\right).$$	
\end{lemma}

Algorithm~\ref{alg:restart} is effective when working with minibatches of the loss functions is computationally feasible. There are, however, important situations where each iteration using a minibatch becomes prohibitively expensive computationally. For example, a proximal point update on a single loss function $\ell(\cdot,z)$ may be computable in closed form, whereas computing the proximal point of an average $\frac{1}{s}\sum_{i=1}^s \ell(x,z_i)$ may computationally challenging. The following section provides an alternative strategy to minibatching that is based on geometrically decreasing the step-size used by the algorithm.

\subsection{Geometric decay schedule}
Suppose that we have available a stochastic algorithm
$\mathcal{A}(y_0,\delta,T)$, such that as long as $\delta<\delta_0$, the method generates a point $y_T$ satisfying 
\begin{equation}\label{eqn:geo_decay}
\EE[h(y_T)]\leq C\left(1-\psi(\delta)\right)^T h(y_0)+D\delta,
\end{equation}
where $h$ is a nonnegative function, $C,D>0$ and $\delta_0\in (0,1)$ are some constants that are specific to the algorithm, and $\psi$ is a function mapping $[0,\delta_0)$ into $(0,1)$. In typical circumstances, $h(y)$ may denote the function gap $\varphi(y)-\min\varphi$ or the square distance to the solution $\|y-\bar x\|^2$; the coefficient
 $\delta$ is proportional to a step-size parameter that the user is free to choose; the function  $\psi(\delta)=c\delta$ is linear for some constant $c>0$.
The procedure, formally described in Algorithm~\ref{alg:geo_decay}, runs in stages by repeatedly calling $\mathcal{A}(y_k,\delta_k,T_k)$. Specifically, in stage $k$, the method
sets $\delta_k=\delta_0 2^{-k}$ and finds a point $y$ satisfying $\EE[h(y_T)]\leq 2\cdot D\delta_k$, while using the output from the previous stage as a warmstart.

\begin{algorithm}[H]
	\KwIn{$y_0\in\R^d$, $C,D>0$, $\delta_0\in (0,1)$, estimate $\Delta\geq h(y_0)$,  
	accuracy $\varepsilon>0$, algorithm $\mathcal{A}(y,\delta,T)$ satisfying \eqref{eqn:geo_decay}.}

\myproc{Set  $y_0=\mathcal{A}(y_0,\delta_0,T_0)$ with  $T_0=\frac{1}{\psi(\delta_0)}\cdot\log(\frac{2C\Delta}{\varepsilon})$.}
	
	Set $K=\left\lceil 1+\log_2\left(\frac{D\delta_0}{\varepsilon}\right)\right\rceil$.

	{\bf Step} $k=1,\ldots, K$: 
$$\textrm{Set } y_k=\mathcal{A}(y_{k-1},\delta_k,T_k)\quad\textrm{ with }\quad \delta_k=2^{-k}\delta_0,~ T_k=\left\lceil\frac{1}{\psi(\delta_k)}\cdot\log(4C)\right\rceil.$$

{\bf Return} $y_K$.
	
	\caption{Geometric decay schedule}\label{alg:geo_decay}
\end{algorithm}

The overall efficiency of the scheme is summarized in the following lemma. The proof can be found in \cite[Appendix B.1]{kulunchakov2019generic}.

\begin{lemma}[Geometric decay]\label{lem:geo_restart}
The point $y$ returned by Algorithm~\ref{alg:geo_decay} satisfies $\EE[h(y_K)]\leq\varepsilon$ and the efficiency estimate holds:	
$$\sum_{k=0}^K T_k =\left\lceil\frac{1}{\psi(\delta_0)}\cdot\log\left(\frac{2C\Delta}{\varepsilon}\right)\right\rceil+\sum_{k=1}^K \left\lceil\frac{\log(4C)}{\psi(2^{-k}\delta_0)}\right\rceil.$$	
In particular, when $\psi$ has the form $\psi(\delta)=c\delta$ for some constants $c\in (0,\frac{1}{\delta_0})$, the estimate becomes
\begin{equation}\label{eqn:specializing_restart}
\sum_{k=0}^K T_k =\mathcal{O}\left(\frac{1}{c\delta_0}\cdot\log\left(\frac{2C\Delta}{\varepsilon}\right)+\frac{D\log(4C)}{\varepsilon c}\right).
\end{equation}
\end{lemma}

\section{Proof of Theorem~\ref{thm:accel}}\label{sec:accel_tech}
The argument we present is based on the technique of estimate sequences, originally introduced by Nesterov \cite[Section 2.2.1]{intro_lect} for deterministic algorithms, and recently extended by \cite{kulunchakov2019estimate} to stochastic settings. In particular, we closely follow the notation and the proof outline of \cite[Section 4.1]{kulunchakov2019estimate}.

Recall that Algorithm~\ref{alg:sgd_acc}  amounts to the recursion 
$$	\left\{\begin{aligned}
x_{t}&=\argmin_x ~\langle g_t,x\rangle +r(x)+\frac{1}{2\eta_t}\|x-y_{t-1}\|^2, \\
\beta_t&=\tfrac{\delta_t(1-\delta_t)\eta_{t+1}}{\eta_t\delta_{t+1}+\eta_{t+1}\delta_t^2}\\
y_t&=x_t+\beta_t(x_t-x_{t-1})
\end{aligned}\right\},
$$
where $\delta_t$ and $\gamma_t$ are defined in \eqref{eqn:acc_gamma_delta}, with $\gamma_0\geq \hat{\alpha}$. Henceforth, define $\Gamma_t:=\prod_{i=1}^t (1-\delta_i)$.

\paragraph{Estimate Sequences.} We next set up the machinery of estimate sequences. To this end, define the auxiliary vectors
$$\tilde g_t:=\eta^{-1}_t(y_{t-1}-x_t)\qquad \textrm{and}\qquad r'_t:=\tilde g_t-g_t\in \partial r(x_t),$$
and the ``local models''
$$l_t(x):= f(y_{t-1})+\langle g_t, x-y_{t-1}\rangle+\frac{\hat \alpha}{2}\|x-y_{t-1}\|^2+r(x_t)+\langle r'_t,x-x_t\rangle.$$
 Note that Lemma~\ref{lem:approx_subgrad_ineq} guarantees the bound $\EE[l_t(\bar x)]\leq \varphi(\bar x)$. This  is worth emphasizing: in expectation, the models  $l_t$ evaluated at $\bar x$ lower bound the minimal value $\varphi(\bar x)$. 
 
 Estimate sequences are constructed by aggregating the ``local models'' along the iterations. Namely,  fix a vector $v_0\in \R^d$ and define the sequence of functions:
\begin{align}
d_0(x)&:=\varphi(x_0)+\frac{\gamma_0}{2}\|x-v_0\|^2,\\
d_t(x)&:=(1-\delta_t)d_{t-1}(x)+\delta_t l_t(x)\qquad \forall t\geq 1,\label{eqn:induct_est_seq}
\end{align}
Since each function $d_t$ is a spherical quadratic, we may write it in standard form 
\begin{equation}\label{eqn:canonnical_est_seq}
d_t(x)=d^*_t+\frac{\gamma_t}{2}\|x-v_t\|^2,
\end{equation}
for some vectors $v_t\in\R^d$ and real $d_t^*\in\R$. The following relationship between the sequences $v_t$, $x_t$, and $y_t$ follows from algebraic manipulations and is classical \cite[Lemma~1]{kulunchakov2019estimate}:
\begin{equation}\label{eqn:key-recurs_rela}
(x_{t-1}-y_{t-1})+\frac{\delta_t\gamma_{t-1}}{\gamma_t}(v_{t-1}-y_{t-1})=0\qquad \forall t\geq 1.
\end{equation}

Henceforth, we will let $\EE_t[\cdot]$ denote the expectation conditioned on $g_0,\ldots,g_{t-1}$. We begin with the following basic lemma showing that lower bounds on the deviations $\EE[d^*_t]- \EE[\varphi(x_{t})]$ directly control the progress of the algorithm.

\begin{lemma}[Basic estimate sequence bound]\label{lem:basic_lem}
	Suppose that there exists a sequence of  numbers $\xi_t\geq 0$ satisfying $\EE[d^*_t]\geq \EE[\varphi(x_{t})]-\xi_{t}$ for each $t\geq 0$.	Then for each $t\geq 0$ the estimate holds:
	$$\EE\left[\varphi(x_{t})-\varphi(\bar x)+\frac{\gamma_t}{2}\|\bar x-v_t\|^2\right]\leq \Gamma_t\left[\varphi(x_0)-\varphi(\bar x)+\frac{\gamma_0}{2}\|x_0-\bar x\|^2\right]+\xi_{t}.$$
\end{lemma}
\begin{proof}
	Lemma~\ref{lem:approx_subgrad_ineq} guarantees the bound $\EE_t[l_t(\bar x)]\leq \varphi(\bar x)$. Therefore from the definition of the estimate sequence, we have
	$$\EE_t[d_{t}(\bar x)]=(1-\delta_t) d_{t-1}(\bar x)+\delta_t \EE_t[l_t(\bar x)]\leq (1-\delta_t)d_{t-1}(\bar x)+\delta_t \varphi(\bar x)$$
	Using the tower rule for expectations, subtracting $\varphi(\bar x)$ from both sides, and unrolling the recursion, yields
	$$\EE[d_{t}(\bar x)]-\varphi(\bar x)\leq (1-\delta_t)(\EE[d_{t-1}(\bar x)]-\varphi(\bar x))\leq  \Gamma_t( d_{0}(\bar x)-\varphi(\bar x)).$$
	Using the expression \eqref{eqn:canonnical_est_seq} for $d_t$ therefore gives
	$$\EE\left[d_t^*-\varphi(\bar x)+\frac{\gamma_t}{2}\|\bar x-v_t\|^2\right]\leq \Gamma_t( d_{0}(\bar x)-\varphi(\bar x)).$$
	Taking into account the assumed  bound, $\EE[d_t^{*}]\geq \EE[\varphi(x_{t})]-\xi_{t}$,  completes the proof.
\end{proof}

In light of Lemma~\ref{lem:basic_lem}, our immediate task is to find conditions guaranteeing  the inequality $\EE[d^*_t]\geq \EE[\varphi(x_{t})]-\xi_t$ for some deterministic sequence $\xi_t$.
This is the content of the following theorem, whose proof we postpone to Section~\ref{sec:proof_hard_acc}.
\begin{theorem}[Accelerated stochastic gradient method]\label{thm:accele_stoch_grad}
	Suppose  $\eta_t\leq \frac{1}{4L}$ and that for all $t=0,\ldots,T$ the inequality holds:
	\begin{equation}\label{eqn:suff_accel}
	\sqrt{1+\frac{\delta_t}{\hat \alpha\eta_t}}
	\leq \frac{{\hat\alpha/B}}{\sqrt{32}(1+\eta_t B)}.
	\end{equation}
	Then for all $t=0,\ldots,T$, the estimate
	$\EE[d^*_t]\geq \EE[\varphi(x_{t})]-\xi_t$ holds, where $\xi_{t}$ are defined recursively as $\xi_{t}=(1-\delta_t)\xi_{t-1}+\frac{9}{8}\eta_t\sigma^2$.
\end{theorem}

Let us now deduce Theorem~\ref{thm:accel} from Theorem~\ref{thm:accele_stoch_grad}. Namely, setting $\eta_t=\frac{1}{4L}$ and $\gamma_0=\hat \alpha$ results in the expressions $\delta_t=\sqrt{\hat\alpha\eta_t}=(1/2)\sqrt{\hat\alpha/L}$ for all indices $t$. Therefore the sufficient condition \eqref{eqn:suff_accel} amount to 
$$\sqrt{1+2\sqrt{\frac{L}{\hat\alpha}}}\leq \frac{\frac{1-2\rho}{\rho}}{\sqrt{32}(1+\frac{\alpha}{4L}\cdot\rho)}.$$
Using the estimate $1+\frac{\alpha}{4L}\cdot\rho\leq 2$ on the right and $\rho\leq \frac{1}{3}$ on the left, it suffices to ensure 
$$\sqrt{1+2\sqrt{3\kappa}}\leq \frac{2}{2\sqrt{32}}\frac{1-2\rho}{\rho}.$$
Solving for $\rho$ yields exactly the parameter regime assumed in Theorem~\ref{thm:accel}. An application of Lemma~\ref{lem:basic_lem} yields therefore the guarantee
$$\EE\left[\varphi(x_{t})-\varphi(\bar x)\right]\leq 2\Gamma_t\left[\varphi(x_0)-\varphi(\bar x)\right]+\xi_{t}.$$
Unrolling the recursion for $\xi_{t}$ and using Lemma~\ref{lem:gamma_sum} yields
$$\xi_{t}=\frac{9}{8}\Gamma_t\sum_{i=1}^t \frac{\eta_i\sigma^2}{\Gamma_i}=\frac{9\sigma^2}{16\sqrt{L\hat\alpha}}(1-\Gamma_t).$$
This completes the proof of Theorem~\ref{thm:accel}.

\subsection{Proof of Theorem~\ref{thm:accele_stoch_grad}}\label{sec:proof_hard_acc}
	The argument we present closely parallels that of \cite[Theorem 3]{kulunchakov2019estimate}, which in turn builds on Nesterov's original treatment in \cite[p. 78]{intro_lect}. 
	Assume by induction $\EE[d_{t-1}^*]\geq \EE[\varphi(x_{t-1})]-\xi_{t-1}$
	for some constant $\xi_{t-1}\geq 0$. We aim to establish an analogous estimate for the next iterate.
	The same algebraic manipulations as in the very beginning of \cite[Theorem 3]{kulunchakov2019estimate} apply verbatim, yielding  the lower bound
	\begin{equation}\label{eqn:key_lower_bd}
	\begin{aligned}
	d_{t}^*&\geq (1-\delta_t)d^*_{t-1}+\delta_t l_t(y_{t-1})-\frac{\eta_t}{2}\| \tilde g_t\|^2\\
	&\qquad +\tfrac{\delta_t(1-\delta_t)\gamma_{t-1}}{\gamma_t}\left(\langle \tilde g_t, v_{t-1}-y_{t-1} \rangle+\frac{\hat{\alpha}}{2}\|v_{t-1}-y_{t-1}\|^2\right).
	\end{aligned}
	\end{equation}
	In the static setting (both deterministic \cite{intro_lect} and stochastic \cite{kulunchakov2019estimate}), the term $\|v_{t-1}-y_{t-1}\|^2$ is lower-bounded by zero and ignored. We will instead carry this term forward in order to offset the bias.
	To simplify notation, define the bias 
	$\Delta_t:=\nabla f(y_{t-1})- \EE[g_t]$.
	Combining \eqref{eqn:induct_est_seq}, \eqref{eqn:key_lower_bd}, and the inductive hypothesis yields
	\begin{equation}\label{eqn:first_key-est}
	\begin{aligned}
	\EE[d_{t}^*]&\geq (1-\delta_t)[\EE[\varphi(x_{t-1})]-\xi_{t-1}]+\delta_t \EE[l_t(y_{t-1})]-\frac{\eta_t}{2}\EE[\| \tilde g_t\|^2]\\
	&\qquad+\tfrac{\delta_t(1-\delta_t)\gamma_{t-1}}{\gamma_t}\EE\left(\langle \tilde g_t, v_{t-1}-y_{t-1} \rangle+\frac{\hat{\alpha}}{2}\|y_{t-1}-v_{t-1}\|^2\right).
	\end{aligned}
	\end{equation}
	Next, using convexity, we conclude
	\begin{align*}
	\EE[\varphi(x_{t-1})]&\geq \EE\left[f(y_{t-1})+\langle \nabla f(y_{t-1}),x_{t-1}-y_{t-1}\rangle+r(x_t)+\langle r'_t,x_{t-1}-x_t\rangle\right]\\
	&= \EE\left[l_t(y_{t-1})+\langle \tilde{g}_t,x_{t-1}-y_{t-1}\rangle+\langle \Delta_t,x_{t-1}-y_{t-1}\rangle\right].
	\end{align*}
	Combining this estimate with \eqref{eqn:first_key-est} yields 
	\begin{align*}
	\EE[d_{t}^*]&\geq \EE[l_t(y_{t-1})]-(1-\delta_t)\xi_{t-1}-\frac{\eta_t}{2}\EE[\| \tilde g_t\|^2]\\
	&\quad+(1-\delta_t)\EE\left[ \langle \tilde{g}_t, (x_{t-1}-y_{t-1})+\frac{\delta_t\gamma_{t-1}}{\gamma_t}(v_{t-1}-y_{t-1}) \rangle\right]\\
	&\quad+ \frac{\hat{\alpha}\delta_t(1-\delta_t)\gamma_{t-1}}{2\gamma_t}\EE\bigl[\|y_{t-1}-v_{t-1}\|^2\bigr] +(1-\delta_t)\EE\langle \Delta_t,x_{t-1}-y_{t-1}\rangle
	\end{align*}
	The quantity $(x_{t-1}-y_{t-1})+\frac{\delta_t\gamma_{t-1}}{\gamma_t}(v_{t-1}-y_{t-1})$ is zero by construction \eqref{eqn:key-recurs_rela}. We conclude 
	\begin{align}
	\EE[d_{t}^*]&\geq \EE[l_t(y_{t-1})]-(1-\delta_t)\xi_{t-1}-\frac{\eta_t}{2}\EE_t[\| \tilde g_t\|^2]\notag\\
	&\quad+ (1-\delta_t)\left(\frac{\hat{\alpha}\gamma_{t-1}\delta_t}{2\gamma_t}\EE[\|y_{t-1}-v_{t-1}\|^2] +\EE\langle \Delta_t,x_{t-1}-y_{t-1}\rangle\right).\label{eqn:last_line_young}
	\end{align}

	The next two lemmas lower bound $\EE[l_t(y_{t-1})]$ and the term \eqref{eqn:last_line_young}.
	
	\begin{lemma}\label{lem:young_second}
		For every index $t\geq 1$ the estimate holds:
		$$ \tfrac{\hat{\alpha}\gamma_{t-1}\delta_t}{2\gamma_t}\|y_{t-1}-v_{t-1}\|^2 +\langle \Delta_t,x_{t-1}-y_{t-1}\rangle\geq -\tfrac{\gamma_{t-1}\delta_t}{2\gamma_t\hat \alpha}\cdot\| \Delta_t\|^2.$$
	\end{lemma}

	\begin{lemma}\label{lem:key_est_accel}
		For every index $t$, the estimate holds:	
		$$\EE[\varphi(x_t)]\leq \EE[l_t(y_{t-1})]+\left(\frac{L\eta_t^2}{2}-\frac{3\eta_t}{4}\right)\EE[\|\tilde{g}_t\|^2]+\eta_t\sigma^2+\eta_t\EE[\|\Delta_t\|^2].$$
	\end{lemma}

	Combining the estimate \eqref{eqn:last_line_young} and Lemmas~\ref{lem:young_second}, \ref{lem:key_est_accel}, we obtain
	\begin{align*}
	\EE[d_{t}^*&\geq \EE \varphi(x_{t})-(1-\delta_t)\xi_{t-1}-\eta_t\sigma^2-P_t,
	\end{align*}
	where we define the error term $$P_t:=\left(\frac{L\eta_t^2}{2}-\frac{\eta_t}{4}\right)\EE[\|\tilde{g}_t\|^2]+\eta_t\left(1+\tfrac{(1-\delta_t)\gamma_{t-1}\delta_t}{2\gamma_t\hat \alpha\eta_t}\right)\EE\| \Delta_t\|^2.$$
	Our final goal is to show the estimate $P_t\leq c\eta_t\sigma^2$ for some constant $c$, which will complete the induction by setting $\xi_{t}:=(1-\delta_t)\xi_{t-1}+(1+c)\sigma^2\eta_t$. To this end, let us simplify the expression for $P_t$ by using \eqref{eqn:acc_gamma_delta} to write $$\frac{(1-\delta_t)\gamma_{t-1}\delta_t}{2\gamma_{t}\hat \alpha\eta_t}=\frac{1}{2}\cdot\frac{(1-\delta_t)\gamma_{t-1}}{\gamma_t}\cdot\frac{\delta_t}{\hat\alpha\eta_t}=\frac{1}{2}\left(1-\frac{\delta_t\hat\alpha}{\gamma_{t}}\right)\cdot\frac{\delta_t}{\hat\alpha\eta_t}=\frac{1}{2}\left(\frac{\delta_t}{\hat\alpha\eta_t} -\frac{\delta_t^2}{\gamma_t\eta_t}\right)=\frac{1}{2}\left(\frac{\delta_t}{\hat\alpha\eta_t} -1\right).$$
	Thus we arrive at the expression
	$$P_t:=\left(\frac{L\eta_t^2}{2}-\frac{\eta_t}{4}\right)\EE[\|\tilde{g}_t\|^2]+\frac{\eta_t}{2}\left(1+\frac{\delta_t}{\hat\alpha\eta_t}\right)\EE\| \Delta_t\|^2.$$
	The strategy is now to show that $\EE[\|\tilde{g}_t\|^2]$ is much larger than $\|y_{t-1}-\bar x\|^2$ while $\EE\| \Delta_t\|^2$ is much smaller than $\|y_{t-1}-\bar x\|^2$. This is the content of the following lemma.

	\begin{lemma}\label{lem:lower_moment_bound}
		For each index $t\geq 1$, the estimate holds:
\begin{align*}
\EE\bigl[\|\tilde g_t\|^2\bigr] 
& \geq \frac{{\hat\alpha}^2}{8(1+\eta_t B)^2}\EE\bigl[\|y_{t-1}-\bar x\|^2 \bigr] - \sigma^2,
\end{align*}		
		
	\end{lemma}

Using Lemma~\ref{lem:lower_moment_bound} and the estimates $\eta_t\leq \frac{1}{4L}$ and
	$\EE\|\Delta_t\|^2\leq B^2\EE\|y_{t-1}-\bar x\|^2$ we conclude
	\begin{align*}
	P_t=& \left(\frac{L\eta_t^2}{2}-\frac{\eta_t}{4}\right)\EE[\|\tilde{g}_t\|^2]+ \eta_t\left(1+\tfrac{(1-\delta_t)\gamma_{t-1}\delta_t}{2\gamma_t\hat \alpha\eta_t} \right)\EE\|\Delta_t\|^2\\
	&\leq \eta_t\left(-\tfrac{1}{8}\EE[\|\tilde{g}_t\|^2] +\tfrac{1}{2}\left(1+\tfrac{\delta_t}{\hat \alpha\eta_t} \right)\EE\|\Delta_t\|^2\right)\\
	&\leq \tfrac{\eta_t}{8}\sigma^2 + \eta_t\left(-\tfrac{{\hat\alpha}^2}{64(1+\eta_tB^2} +\tfrac{1}{2}\left(1+\tfrac{\delta_t}{\hat \alpha\eta_t}\right)B^2\right)\EE\|y_{t-1}-\bar x\|^2\\
	&\leq \tfrac{\eta_t}{8}\sigma^2.
	\end{align*}
		This completes the induction by setting 
		$\xi_{t}=(1-\delta_t)\xi_{t-1}+\frac{9}{8}\eta_t\sigma^2.$
		
	\bigskip
	\bigskip
	\bigskip
	\bigskip

\section{Proofs of technical lemmas in Section~\ref{sec:accel_tech}}
\subsection{Proof of Lemma~\ref{lem:young_second}}
		Using the expression, $x_{t-1}-y_{t-1}=\frac{\delta_t\gamma_{t-1}}{\gamma_t}(y_{t-1}-v_{t-1})$ and Young's inequality, we deduce
\begin{align*}
\tfrac{\hat{\alpha}\gamma_{t-1}\delta_t}{2\gamma_t}\|y_{t-1}-v_{t-1}\|^2 +\langle \Delta_t,x_{t-1}-y_{t-1}\rangle&=\tfrac{\gamma_{t-1}\delta_t}{\gamma_t}\left(\frac{\hat{\alpha}}{2}\|y_{t-1}-v_{t-1}\|^2+\langle \Delta_t,y_{t-1}-v_{t-1}\rangle\right)\\
&\geq -\frac{\gamma_{t-1}\delta_t}{2\gamma_t\hat \alpha}\cdot\| \Delta_t\|^2.
\end{align*}
The proof is complete.

\subsection{Proof of Lemma~\ref{lem:key_est_accel}}
		The argument follows similar reasoning as \cite[Lemma 2]{kulunchakov2019estimate}, with a careful accounting for the bias. Using smoothness, we estimate
\begin{equation}\label{eqn:key-est_accelera}
\begin{aligned}
\varphi(x_t)&\leq \EE[f(y_{t-1})+\langle \nabla f(y_{t-1}),x_t-y_{t-1}\rangle+\frac{L}{2}\|x_t-y_{t-1}\|^2+r(x_t)]\\
&=\EE[f(y_{t-1})+\langle g_t,x_t-y_{t-1}\rangle+\frac{L}{2}\|x_t-y_{t-1}\|^2+r(x_t)]+\EE[\langle  \EE[g_t]-g_t,x_t-y_{t-1}\rangle]\\
&\qquad +\EE[\langle \nabla f(y_{t-1})-\EE[g_t],x_t-y_{t-1}\rangle],
\end{aligned}
\end{equation}
where the last equality follows from algebraic manipulations.
Next, we compute 
\begin{align*}
\EE\langle  \EE[g_t]-g_t,x_t-y_{t-1}\rangle= \EE\langle  \EE[g_t]-g_t,x_t\rangle= \EE\langle  \EE[g_t]-g_t,x_t-w_{t}\rangle,
\end{align*}
where we define $w_t:=\prox_{\eta_t r}(y_{t-1}-\eta_t\EE[g_t])$. Taking into account that the proximal map $\prox_{\eta_t r}(\cdot)$ is nonexpansive, we deduce 
$$\|x_t-w_{t}\|=\|\prox_{\eta_t r}(y_{t-1}-\eta_t g_t)-\prox_{\eta_t r}(y_{t-1}-\eta_t\EE[g_t])\|\leq \eta_t\|g_t-\EE[g_t]\|.$$
Therefore continuing \eqref{eqn:key-est_accelera} we obtain
\begin{align*}
\varphi(x_t)
&\leq \EE[f(y_{t-1})+\langle g_t,x_t-y_{t-1}\rangle+\frac{L}{2}\|x_t-y_{t-1}\|^2+r(x_t)]+\eta_t\sigma^2\\
&\qquad +\EE[\langle \nabla f(y_{t-1})-\EE[g_t],x_t-y_{t-1}\rangle]\\
&=\EE[l_{t}(y_{t-1})+\langle \tilde g_t,x_t-y_{t-1}\rangle+\frac{L}{2}\|x_t-y_{t-1}\|^2]+\eta_t\sigma^2\\
&\qquad +\EE[\langle \nabla f(y_{t-1})-\EE[g_t],x_t-y_{t-1}\rangle]\\
&=\EE[l_t(y_{t-1})]+\left(\frac{L\eta_t^2}{2}-\eta_t\right)\EE[\|\tilde{g}_t\|^2]+\eta_t\sigma^2\\
&\qquad +\EE[\langle \nabla f(y_{t-1})-\EE[g_t],x_t-y_{t-1}\rangle].
\end{align*}
Finally, Young's inequality yields
\begin{align*}
\EE[\langle \nabla f(y_{t-1})-\EE[g_t],x_t-y_{t-1}\rangle]&\leq \eta_t\|\nabla f(y_{t-1})-\EE[g_t]\|^2+\frac{1}{4\eta}\|x_t-y_{t-1}\|^2\\
&=\eta\|\nabla f(y_{t-1})-\EE[g_t]\|^2+\frac{\eta_t}{4}\|\tilde g_t\|^2,
\end{align*}
thereby completing the proof.

\subsection{Proof of Lemma~\ref{lem:lower_moment_bound}}
The result will follow quickly from the following stand-alone lemma.

	\begin{lemma}\label{lem:one_step_strong}
	Fix a constant $\eta\leq 2/L$, a point $y\in\R^d$, and a vector $v\in \R^d$ satisfying $\|v-\nabla f(y)\|\leq B\|y-\bar x\|$. Define the proximal gradient update and the displacement vector:
	$$y^{+}=\prox_{\eta r}(y-\eta v)\qquad \textrm{and}\qquad \tilde{g}=\eta^{-1}(y-y^+).$$
	Then the estimate holds:
	\begin{equation}\label{eqn:error_bound}
	\|\tilde g\|\geq \frac{\hat \alpha}{2(1+\eta B)}\|y-\bar x\|.
	\end{equation}
\end{lemma}

Before proving Lemma~\ref{lem:one_step_strong}, let us see how it implies Lemma~\ref{lem:lower_moment_bound}---the result we are after.
Setting $y=y_{t-1}$ and $v=\EE[g_t]$ in Lemma~\ref{lem:one_step_strong} yields
\begin{equation}\label{eqn:key_prox_applic}
\|\eta^{-1}_t(y_{t-1}-\prox_{\eta_t r}(y_{t-1}-\eta_t\EE[g_t]))\|\geq \frac{\hat \alpha}{2(1+\eta B)}\|y_{t-1}-\bar x\|\qquad \forall t\geq 1.
\end{equation}
Therefore, we now compute 
\begin{align*}
\|\tilde{g}_t\|&= \eta_t^{-1}\|y_{t-1}-\prox_{\eta_t r}(y_{t-1}-\eta_t g_t)\|\\
&\geq \eta_t^{-1}\|y_{t-1}-\prox_{\eta_t r}(y_{t-1}-\eta_t \EE[g_t])\|-\eta_t^{-1}\|\prox_{\eta_t r}(y_{t-1}-\eta_t g_t)-\prox_{\eta_t r}(y_{t-1}-\eta_t \EE[g_t])\|\\
&\geq \eta_t^{-1}\|y_{t-1}-\prox_{\eta_t r}(y_{t-1}-\eta_t \EE[g_t])\|-\|g_t-\EE[g_t]\|,\\
&\geq \frac{\hat \alpha}{2(1+\eta_t B)}\|y_{t-1}-\bar x\|-\|g_t-\EE[g_t]\|,
\end{align*}
where the first inequality follows from the reverse triangle inequality, the second inequality uses that the proximal map is nonexpansive, and the third follows from \eqref{eqn:key_prox_applic}.
Rearranging and using the inequality $2a^2+2b^2\geq (a+b)^2$  yields 
\begin{align*}
2\|\tilde g_t\|^2 + 2\|g_t-\EE[g_t]\|^2\geq (\|\tilde g_t\| + \|g_t-\EE[g_t]\|)^2\geq \left(\frac{\hat \alpha}{2(1+\eta B)}\|y_{t-1}-\bar x\|\right)^2.
\end{align*}
Taking expectations of both sides completes the proof of Lemma~\ref{lem:lower_moment_bound}.
It remains to verify Lemma~\ref{lem:one_step_strong}.

\begin{proof}[Proof of Lemma~\ref{lem:one_step_strong}]
	To simplify notation, define the error $\Delta= v-\nabla f(y)$.
	Since $y^+$ is the minimizer of the $\eta^{-1}$-strongly convex function $\langle v, \cdot-y\rangle+r+\frac{1}{2\eta}\|\cdot-y\|^2$, we deduce for every $x\in \R^d$ the estimate holds:
	\begin{equation}\label{eqn:thrre_point}
	\begin{aligned}
	\langle v,y^+-y\rangle+\frac{1}{2\eta}\|y^+-y\|^2+r(y^+)\leq \langle v,x-y\rangle+\frac{1}{2\eta}\|x-y\|^2+r(x)-\frac{1}{2\eta}\|y^+-x\|^2
	\end{aligned}.
	\end{equation}
	Next, algebraic manipulations yield the equality 
	\begin{align*}
	\|y^+-y\|^2+\| y^+-x\|^2-\| x-y\|^2&=2\langle  y-y^+,x-y\rangle+2\|y-y^+\|^2\\
	&=2\eta\langle  \tilde g, x-y\rangle+2\eta^2\|\tilde g\|^2.
	\end{align*}
	Combining this estimate with \eqref{eqn:thrre_point} yields
	\begin{equation}\label{eqn:next_est}
	\langle v,y^+-y\rangle+r(y^+)\leq \langle v,x-y\rangle+r(x)+ \langle  \tilde{g}, y-x\rangle-\eta\|\tilde{g}\|^2.
	\end{equation}
	Smoothness of $f$ in turn guarantees
	\begin{align*}
	\langle v,y^+-y\rangle&=\langle \nabla f(y),y^+-y\rangle+\langle v-\nabla f(y),y^+-y\rangle\\
	&\geq f(y^+)-f(y)-\frac{L}{2}\|y^+-y\|^2-\eta\langle \Delta,\tilde g\rangle.
	\end{align*}
	Combining this estimate with \eqref{eqn:next_est} therefore yields 
	\begin{equation}\label{eqn:fina:lem_eqn}
	\begin{aligned}
	\varphi(y^+)\leq f(y)+&\langle v,x-y\rangle+r(x)+ \langle  \tilde g, y-x\rangle
	-\left(\eta-\frac{L\eta^2}{2}\right)\|\tilde g\|^2+\eta\langle \Delta,\tilde g\rangle.
	\end{aligned}
	\end{equation}
	Set now $x=\bar x$ in \eqref{eqn:fina:lem_eqn} and using strong convexity observe the estimate 
	\begin{align*}
	\langle v,\bar x-y\rangle&=\langle \nabla f(y),\bar x-y\rangle+\langle v-\nabla f(y),\bar x-y\rangle\\
	&\leq f(\bar x)-f(y)-\frac{\hat\alpha}{2}\|y-\bar x\|^2.
	\end{align*}
	We thus deduce 
	\begin{equation}
	\begin{aligned}
	\varphi(y^+)\leq \varphi(\bar x)+& \langle  \tilde g, y-\bar x\rangle-\frac{\hat \alpha}{2}\|y-\bar x\|^2-\left(\eta-\frac{L\eta^2}{2}\right)\|\tilde g\|^2+\eta\langle \Delta,\tilde g\rangle.
	\end{aligned}
	\end{equation}	
	Lower-bounding $\varphi(y^+)$ by $\varphi(\bar x)$ and rearranging yields
	$$\langle  \tilde g, y-\bar x\rangle\geq \tfrac{\hat \alpha}{2}\|y-\bar x\|^2+\left(\eta-\tfrac{L\eta^2}{2}\right)\|\tilde g\|^2-\eta\langle \Delta,\tilde g\rangle.$$
	Using Cauchy-Schwarz, lower bounding the term $\left(\eta-\frac{L\eta^2}{2}\right)\|\tilde g\|^2$ by zero, and dividing through by $\|y-\bar x\|$ yields 
	\begin{align*}
	\|\tilde g\|&\geq \frac{\tilde \alpha}{2}\|y-\bar x\|-\frac{\eta\|\Delta\|\|\tilde g\|}{\|y-\bar x\|}\geq \frac{\hat \alpha}{2}\|y-\bar x\|-\eta B \|\tilde g\|.
	\end{align*}
	Rearranging completes the proof of \eqref{eqn:error_bound}.
\end{proof}

\section{Estimate \eqref{eqn:grad_rel_accur} implies an angle condition}\label{sec:angle_cond}
Squaring both sides of \eqref{eqn:grad_rel_accur}, expanding, and dividing by $\|\nabla f_x(x)\|\cdot\|\nabla f_{\bar x}(x)\|$ yields
$$2\left\langle \frac{\nabla f_x(x)}{\|\nabla f_x(x)\|},\frac{\nabla f_{\bar x}(x)}{\|\nabla f_{\bar x}(x)\|}\right\rangle\geq \frac{\|\nabla f_x(x)\|}{\|\nabla f_{\bar x}(x)\|}+(1-\rho^2)\frac{\|\nabla f_{\bar x}(x)\|}{\|\nabla f_x(x)\|}.$$
Lower bounding the right side with the estimate $a+b\geq 2\sqrt{ab}$, we conclude 
$\left\langle \frac{\nabla f_x(x)}{\|\nabla f_x(x)\|},\frac{\nabla f_{\bar x}(x)}{\|\nabla f_{\bar x}(x)\|}\right\rangle\geq \sqrt{1-\rho^2}$, as  claimed.

\section{Numerical experiments on strategic classification}\label{sec:strat_class}
This section describes the experimental setup of strategic classification used in \cite{perdomo2020performative,mendler2020stochastic} and in the current work. 
 Specifically, we begin with the Kaggle data set \cite{kaggle}, which contains a historical financial record of $15000$ individuals. Each individual is described by ten features $a\in \R^{10}$, with the label $b\in \{0,1\}$  indicating whether $90$ days have passed since the person experienced delinquency. In the experiments, we subsample $n/2$ records labeled with $0$ and with $1$, respectively. We normalize all features to have zero mean and unit variance.
The goal is to learn a classifier parametrized by $x\in \R^{10}$ that accurately predicts the label. As in \cite{perdomo2020performative,mendler2020stochastic} we use the regularized logistic loss:
$$\frac{1}{n}\sum_{i=1}^n -b_i\cdot x^{\top}a_i+\log(1+\exp(x^{\top}a_i))+\frac{\alpha}{2}\|x\|^2,$$
for some parameter $\alpha>0$. We isolate three features (utilization of credit lines, number of open credit lines, number of real estate loans) that the individuals will strategically adapt in reaction to a deployed classifier. Namely, when a classifier parametrized by $x$ is deployed, each individual updates the features $a$ as 
$$a_S=a_S-\gamma\cdot x_S,$$
where $S$ is the index set of the three strategic features. This update is equivalent to \eqref{eqn:strat_update} with the linear utility function $u(x,a)=-\langle a,x\rangle$ and quadratic cost $c(a',a)=\frac{1}{2\gamma}\|a'-a\|^2$. The label of each individual stay the same.
It is shown in \cite[Section B.2]{perdomo2020performative} that the gradient of the population objective  is Lipschitz continuous in $x$ with constant $L=\frac{1}{4n}\sum_{i=1}^n\|a_i\|^2+\alpha$ and is Lipschitz continuous in $(a,b)$ with parameter $\beta=2$. Therefore, we may estimate $\rho=\frac{2\gamma}{\alpha}.$ Therefore the interesting parameter for the strategic effects is $\gamma\in (0,\frac{\alpha}{2})$. This estimate is fairly crude and experimentally we see that numerical methods can work well in a much wider parameter regime.

\end{document}